\documentclass[11pt]{article}
\usepackage{amssymb}
\usepackage{amsmath}
\usepackage{amsthm}
\usepackage{bbm,bm}
\usepackage{multirow}
\usepackage{enumerate}
\usepackage{graphicx}
\usepackage{color}
\usepackage{hyperref}
\usepackage{stmaryrd}
\usepackage{relsize}
\usepackage{xcolor}
\usepackage{comment}
\usepackage{epstopdf}
\usepackage{mathtools}
\usepackage{booktabs}

\usepackage{caption}

\SetSymbolFont{stmry}{bold}{U}{stmry}{m}{n} 
\hypersetup{
    colorlinks,
    linkcolor={red!50!black},
    citecolor={blue!50!black},
    urlcolor={blue!80!black}
}
\numberwithin{equation}{section}

\newtheorem{theorem}{Theorem}[section]

\newtheorem{lemma}[theorem]{Lemma}

\newcommand\R{\mathrm{R}}
\newcommand{\tr}{\mathrm{tr}}
\renewcommand{\O}{\Omega}
\newcommand{\qon}{{\quad\hbox{on}\quad}}
\newcommand\qin{{\quad\hbox{in}\quad}}
\newcommand\qan{{\quad\hbox{and}\quad}}

\newcommand{\I}{\mathbb{I}}

\renewcommand{\div}{\mathrm{div}}
\newcommand{\bdiv}{\mathbf{div}}

\newcommand{\mathbcal}[1]{\boldsymbol{\mathcal{#1}}}

\newcommand{\rmd}{{\rm d}}
\newcommand{\bB}[1]{\boldsymbol{ #1 }}
\newcommand{\mbb}[1]{\mathbb{ #1 }}

\newcommand{\bBc}[1]{\bB{\mathcal{#1}}}
\newcommand{\mbf}[1]{\mathbf{#1}}
\newcommand{\divr}{{\rm div}}

\newcommand{\fbk}{f_{\rm bk}}
\newcommand{\bL}{\mathbf{L}}
\renewcommand{\L}{\mathrm{L}}

\newcommand{\bH}{\mathbf{H}}
\newcommand{\bQ}{\mathbf{Q}}
\newcommand{\mH}{\mathbb{H}}
\newcommand{\mX}{\mathbb{X}}
\newcommand{\bS}{\mathbf{S}}

\newcommand{\Hdivtr}{\mbb{H}_{0}(\bdiv_s;\O)}
\newcommand{\Ltwotr}{\mbb{L}_{\texttt{tr}}^2(\O)}
\usepackage{amsfonts}
\usepackage{yfonts}
\newcommand{\mfk}[1]{\mathfrak{#1}}
\makeatletter
\DeclareFontEncoding{LS1}{}{}
\DeclareFontSubstitution{LS1}{stix}{m}{n}
\DeclareMathAlphabet{\mathcc}{LS1}{stixcal}{m}{n}

\makeatother
\theoremstyle{definition}
\theoremstyle{remark}

\numberwithin{equation}{section}
\newtheorem{assumption}{Assumption}
\newcommand{\tmass}{\bB{r}}
\newcommand{\theat}{\bB{t}}
\newcommand{\vtmass}{\vec{\bB{r}}}
\newcommand{\vtheat}{\vec{\bB{t}}}
\newcommand{\vsmass}{\vec{\bB{s}}}
\newcommand{\vsheat}{\vec{\bB{q}}}
\newcommand{\vsigmass}{\bB{z}}         
\newcommand{\vsigheat}{\widetilde{\bB{\sigma}}}
\newcommand{\datacons}{\Psi^{\tt data}_{\tt mass}}
\newcommand{\dataheat}{\Psi^{\tt data}_{\tt heat}}
\newcommand{\dataflow}{\Psi^{\tt data}_{\tt flow}}

\usepackage{enumitem}
\usepackage{calc}
\usepackage{changepage}
\usepackage{authblk}

\makeatletter
\renewcommand{\AB@affilsepx}{\\[0.7em]}
\makeatother

\begin{document}

\title{\bf Mixed finite element methods for a heat-assisted model of sedimentation}
\author[a,$*$]{{\sc Julio Careaga}}
\affil[a]{Bernoulli Institute, University of Groningen, Nijenborgh 9
9747 AG Groningen, The Netherlands}

\author[b]{{\sc Cristian Inzunza}}

\affil[b]{GIANuC$^2$ and Departamento de Matemática y Física Aplicadas, Universidad\\ Católica de la Santísima Concepción, Casilla 297,
Concepción, Chile 
}
\affil[*]{\it Corresponding author, email: {\tt j.c.careaga.solis@rug.nl}}
\date{ }

\maketitle

\begin{abstract}

\noindent A model for the consolidation-sedimentation process incorporating thermal effects is analysed. The model is motivated by novel applications to heat-assisted sedimentation of mineral suspensions.
A system of convection-diffusion equations for the concentration and temperature are coupled to flow equations for the volume-average velocity and pressure. The model is supplemented with Dirichlet and mixed boundary conditions.
A quasi-stationary approach of the system is analysed by means of a mixed variational formulation, where additional variables such as the pseudostress tensor and variables for the gradient of the temperature and concentration are introduced. The existence and uniqueness of the system is shown making use of the Babu{\v s}ka-Brezzi theory and Banach's fixed-point theorem, under hypotheses of smallness of the data and additional regularity on the solutions. A discrete formulation based on finite element spaces is proposed and analysed, and the existence of solutions is achieved thanks to Brouwer's fixed-point theorem. Additionally, C\'ea and a priori estimates are derived for the mixed formulation. Finally, numerical examples are provided to demonstrate the performance of the model and to confirm the theoretical rates of convergence.
\end{abstract}

\section{Introduction}

\subsection{Scope}

We are interested in the numerical approximation and modelling of sedimentation processes of monodispersed solid particles immersed on a viscous fluid. This process is present in a large variety of industrial and biological applications. These applications range from mineral processing \cite{Bustos1999}, geothermal energy \cite{Hernandez2025}, wastewater treatment \cite{Li2017}, and food industry \cite{Yavuz2025} to the sedimentation in rivers, lakes and estuaries \cite{Burchard2018}. To simulate the process, mathematical models written in terms of partial differential equations (PDEs) are derived from the conservation of mass and momentum balances for the phases involved, solid and liquid, or for an averaged combination of them.
When considering a non-isolated system, heat sources induce changes in temperature which consequently affect the physical properties of the materials \cite{Julien2010}. In this regard, sedimentation is not an exception, and these changes in temperature directly affect the fluid flow and therefore the sedimentation process. Then, to properly model the process, an additional equation for the temperature needs to be taken into account.

In this work, we analyse a mathematical model that incorporates several prominent effects arising in the heat-assisted sedimentation processes. Firstly, the hindered moving effect on a solid particle caused by its surrounding particles \cite{Kynch1952}, and compression of the solid phase which together describe the sedimentation-consolidation phenomena of a suspension. Secondly, the strain rate of deformation of the mixture described by the viscous stress tensor with varying viscosity as a function of the concentration of solid particles and temperature of the mixture. Thirdly, the influence of the concentration of solid particles on the heat capacity of the suspension and heat transfer coefficient.
Following \cite{Burger2012b}, the pseudo-stationary conservation of mass can be formulated in terms of the volume fraction of solid particles $\phi=\phi(\bB{x},t)$, with $\bB{x}$ the vector of spatial coordinates and $t$ time, as the following non-linear convection-diffusion equation
\begin{align}\label{eq:intro:eq1}
\delta_1\phi - \div\big(-\phi\bB{u} - f_{\rm bk}(\phi)\bB{k} + \mathcal{D}_1(\phi)\nabla\phi\big) & = l_1\,,
\end{align}
where $\delta_1>0$ is constant, $\bB{u}$ is the so-called volume-averaged velocity,  $f_{\rm bk}\coloneqq f_{\rm bk}(\phi)$ is the Kynch batch flux density function, $\mathcal{D}_1\coloneqq \mathcal{D}_1(\phi)$ is the nonlinear $\phi$-dependent diffusion function which relates to the compression of solid particles, and $l_1$ is a given source term function. Here, $\bB{k}$ is a constant unitary vector pointing in the direction of gravity, which is $\bB{k}=(0,-1)^{\tt t}$ in two spatial dimensions. The nonlinear flux function $f_{\rm bk}$, which is related to a hindered-settling velocity, was already part of Kynch's model early in 1952 \cite{Kynch1952} and plays a central role in many one-dimensional models of sedimentation \cite{Careaga2021, SDAMM_SBR1,SDAMM_SBR2,Careaga2023SBR}. The diffusion function $\mathcal{D}_1$ is related to the effective solid stress function which drives the compression effect of the solid particles. Depending on the definition of $\mathcal{D}_1$, the character of Equation~\eqref{eq:intro:eq1} can drastically change leading to a parabolic-strongly degenerate equation. To overcome this additional difficulty, we consider the case $\mathcal{D}_1(\phi)>d_0$ for all $\phi$ with $d_0$ a positive constant.
The temperature of the suspension $\theta = \theta(\bB{x},t)$ is described by the following quasi-stationary convection-diffusion heat equation with variable diffusion:
\begin{align}\label{eq:intro:eq2}
 \delta_2\theta - \div\big(-\theta\bB{u} + \mathcal{D}_2(\phi)\nabla\theta\big) = l_2\,,
\end{align}
where $\delta_2>0$ is constant and $l_2$ is a source term function, and $\mathcal{D}_2\coloneqq \mathcal{D}_2(\phi)$ is the $\phi$-dependent non-linear diffusion function, which depends on the heat capacity of the suspension and heat transfer coefficients.
As in the concentration equation \eqref{eq:intro:eq1}, there is a linear convective term related to the bulk flow of the mixture and a $\phi$-dependent diffusion term. Although $\mathcal{D}_2$ depends nonlinearly on $\phi$, Equation~\eqref{eq:intro:eq2} is linear in $\theta$. The volume-averaged flow velocity of the mixture $\bB{u}$ is determined by the quasi-stationary flow equations:
\begin{align}
 (\delta_3\mathbb{I}+\delta_4\mbf{K}^{-1}(\phi))\bB{u} - \bdiv\big(\mu(\phi,\theta)\bB{e}(\bB{u}) - p\I\big) & = \bB{l}_3 + \mbf{f}(\phi,\theta)\,, \label{eq:intro:eq3}\\[1ex]
\div(\bB{u}) & = 0\,, \label{eq:intro:eq4}
\end{align}
where $\delta_3,\delta_4>0$ are positive constants, $\mbf{K}$ is the permeability tensor \cite{Alvarez2026}, $\mu\coloneqq \mu(\phi,\theta)$ is the viscosity function of the mixture, $\bB{l}_3$ is a prescribed vector-valued source term, and $\mbf{f}\coloneqq \mbf{f}(\phi,\theta)$ is a source term which involves the gravity force and a $\theta$-dependent part arising after making a Boussinesq approximation in \cite{Careaga2025}. Note that Equations \eqref{eq:intro:eq1}, \eqref{eq:intro:eq2} and \eqref{eq:intro:eq3}, can be seen as semi-discrete approximations through backward Euler method on the respective time derivatives $\partial_t\phi$, $\partial_t\theta$ and $\partial_t \bB{u}$ of their transient counterparts.

System \eqref{eq:intro:eq1}--\eqref{eq:intro:eq4} is an extension of the sedimentation-consolidation model in \cite{Alvarez2026} to the thermodynamic case, and can be related to a quasi-steady state version of the PDE system in \cite{Careaga2025}. Furthermore, the coupling between the equations is through coefficient functions and source terms, while the main nonlinearities are associated to the volume fraction $\phi$ and temperature $\theta$. The governing equations are supplemented with Neumann (concentration equation \eqref{eq:intro:eq1}), Dirichlet (flow equations \eqref{eq:intro:eq2}--\eqref{eq:intro:eq3}) and mixed (heat equation \eqref{eq:intro:eq3}) boundary conditions. Notably, the heat-assisted sedimentation represents an application in which all three types of boundary conditions arise naturally in the mathematical modelling of the process.

\noindent{\bf Main aims.}
The first aim of this work concerns to the rigorous analysis of a mixed weak formulation of the coupled nonlinear system composed by the concentration, heat and flow equations \eqref{eq:intro:eq1}--\eqref{eq:intro:eq4}, combining the mentioned boundary conditions.
The second aim is related to the development of a mixed finite element scheme for the approximation of the system, providing existence and a priori estimate results, and showing numerical investigation of the rates of convergence and simulations. To the best of our knowledge, this is the first work in which this multiphysics problem composed by the nonlinear convection-diffusion equation \eqref{eq:intro:eq1} coupled to the heat \eqref{eq:intro:eq2} and flow equations \eqref{eq:intro:eq3} and \eqref{eq:intro:eq4} through coefficient functions and source terms is analysed.

\subsection{Related works}

The study of sedimentation processes involves several challenges both from the mathematical analysis and numerical approximations perspectives, and also from a modelling point of view. One of the first models of sedimentation was developed in \cite{Kynch1952}, where a first-order nonlinear conservation law in one spatial dimension was proposed for the batch sedimentation process. The main assumption in this model is that the flux function $\fbk$ only depends on the local solids volume fraction $\phi$. Continuously operated one-dimensional models of sedimentation were studied in \cite{Chancelier1994,Diehl1995,Diehl1996}, where the equation for solids-volume fraction is given by a non-homogeneous conservation law with discontinuous flux and point source terms. Compression effects, which occurs when sediment particles get into contact with each other and form a network for high concentrations, was incorporated as a diffusion term, for instance, in \cite{Burger2005a,Chauchat2013,Toorman1999}.

Unlike one-dimensional models of sedimentation, in which the volume-average velocity is known, in the multi-dimensional case the flow velocity is no-longer a given scalar quantity but determined by a momentum balance, forming a coupled PDE system between a transport equation for the concentration and flow equations for the velocity and pressure. In \cite{Burger2012a}, a convection-diffusion equation for the concentration coupled to Stokes system for the volume-average velocity and excess-pore pressure are approximated by means of a finite volume element scheme. The scheme couples a discontinuous Galerkin scheme with a stabilized finite element method for the Stokes equations in an axisymmetric approach. A steady state model of sedimentation-consolidation is analysed in \cite{Alvarez2016b}, where the theory of mixed-finite element methods in Hilbert spaces is used to show well-posedness, and suitable inf-sup stable discrete spaces are proposed for the approximation of the system. In \cite{Burger2020}, a model of sedimentation--consolidation designed to simulate  continuously operated sedimentation tanks including the effect of a rotating device assisting the removal of sediment is studied. An $H(div)$-conforming method for the coupled problem is proposed in which the concentration of solid particles is approximated using a continuous Galerkin finite element space.
A solids velocity based model is studied in \cite{Careaga2023}, in which the structure of the flow equations includes a relation between the divergence of the velocity field and  second-order partial derivatives of the pressure. The weak formulation of the flow equations are analysed using coupled saddle point problems and Babu$\rm\check{s}$ka--Brezzi theory, and the model is approximated combining a finite volume scheme with mixed finite elements. A Banach spaces-based mixed finite element approach is employed in \cite{Alvarez2026} to analyse a steady-state model of  sedimentation-consolidation. The model and numerical approximation yields to a momentum-conserving fully–mixed approach, and the theoretical results are based on the Babu$\rm\check{s}$ka--Brezzi theory for perturbed saddle-point problems in Banach spaces along with suitable regularity assumptions and fixed-point arguments.

Models of sedimentation including temperature variations can be found in  \cite{Careaga2025,Reyes2019, Motta2024}. In \cite{Reyes2019}, the authors study the heat-assisted sedimentation process coupling a heat equation to mass and momentum balances for the solid and liquid phase. The model in \cite{Careaga2025} consists of a nonlinear convection equation for the concentration of solids, flow equations for the volume-average velocity, and a heat equation for the temperature, in which the coupling is through coefficient functions and source terms. A continuum model for the sedimentation process of weighting agents in non-Newtonian suspensions including temperature variations is proposed in \cite{Motta2024}.

For a complete introduction to the theory of mixed finite element methods, we refer to \cite{Gatica2014}, while Banach-spaces based mixed formulations of coupled flow and convection-diffusion equations can be found in \cite{Benavides2020,Careaga2025a,Colmenares2026}. In particular, in \cite{Benavides2020} a nonlinear stationary convection-diffusion equation coupled to flow equations with concentration dependent viscosity function that can be related to \eqref{eq:intro:eq1} coupled to \eqref{eq:intro:eq3} is analysed using a dual-mixed formulation for the fluid and primal formulation for the concentration equation. The analysis of a model including the coupling between heat and flow equations through a nonlinear heat diffusion is studied in \cite{Careaga2025a} by means of coupled saddle point formulations in Banach spaces. A fully mixed formulation for the Navier--Stokes equations with concentration dependent viscosity coupled with a conservation law for microorganisms concentration is studied in \cite{Colmenares2026}. In these works, a common fixed-point strategy is employed to establish the well-posedness of the coupled system.

\subsection{Outline of the paper}

The rest of the paper is structured as follows. In Section~\ref{sec:preliminaries} we provide standard notation for functional spaces and differential operators as well as a description of the governing equations, functions and parameters used in the paper. Section~\ref{sec:continuous:formulation} is dedicated to the derivation of the weak formulation of the concentration, heat and flow equations by means of decoupled saddle point problems, and establish the fully coupled weak formulation.
In Section~\ref{sec:CF:WP}, the solvability analysis of the continuous formulation is treated, where the well-posedness analysis of each decoupled saddle point problem is presented. The main result of the section is the existence and uniqueness of the fully coupled PDE system, which is carried out using a fixed-point argument under an additional regularity assumption. In Section~\ref{sec:DF:WP}, we introduce the discrete formulation of the decoupled saddle point problems, including their solvability analysis, and show an existence result for the fully coupled discrete problem. A priori estimates, including C\'ea estimates and rates of convergence are provided in Section~\ref{sec:a:priori}. Finally, numerical examples showing the performance of the proposed mixed finite element scheme and theoretical rates of convergence are included in Section~\ref{sec:numerical:examples}, while conclusions are provided in Section~\ref{sec:conclusions}.

\section{Model formulation}\label{sec:preliminaries}

\subsection{Further notation}

Throughout the course of the paper, the set $\Omega\subset\R^d$ with $d\in\{2,3\}$ will correspond to a polyhedral bounded domain, with boundary $\Gamma:=\partial\Omega$ and outward unit normal vector $\bB{\nu}$. Scalar functional spaces will be denoted by Roman letters, vector spaces by bold letters and tensorial spaces by blackboard bold letters. Given $k\in[1,+\infty)$, we adopt the standard notation ${\rm L}^k(\Omega)$ for the Lebesgue spaces, and ${\rm W}^{k,s}(\Omega)$, with $s\geq0$, for the Sobolev spaces of exponents $k$ and $s$. The norms of ${\rm L}^k(\Omega)$ and ${\rm W}^{s,k}(\Omega)$ are $\|\cdot\|_{0,k;\O}$ and $\|\cdot\|_{s,k;\O}$, respectively, and $|\cdot|_{s,k;\O}$ refers to the standard seminorm of ${\rm W}^{s,k}(\Omega)$. In addition, we define the following vector and tensor functional spaces:
\begin{align}
\mbf{H}(\divr_k;\Omega)
&:= \left\{ \bB{w}\in\mbf{L}^2(\Omega): \quad \divr(\bB{w})\in\L^k(\Omega) \right\}\,, \label{def:hdiv:vector} \\[1ex]
\mbb{H}(\mbf{div}_k;\Omega)
&:= \left\{ \bB{\tau}\in\mbb{L}^2(\Omega): \quad \mbf{div}(\bB{\tau})\in\mbf{L}^k(\Omega) \right\}\,,\label{def:hdiv:tensor}\\[1ex]
\mbb{H}_0(\bdiv_k;\O)
&:= \left\{ \bB{\tau}\in\mbb{H}(\bdiv_k;\O): \, \int_\O\tr(\bB{\tau})=0 \right\}\,.
\end{align}
Here, $\divr(\cdot)$ denotes the divergence operator in $\R^d$ acting on vector fields, whereas $\mbf{div}(\cdot)$ corresponds to its tensorial extension acting along the rows of a tensor and ${\rm tr}(\cdot)$ corresponds to the trace function. The norms of the above spaces are respectively defined by
\begin{alignat}{2}
\|\bB{w}\|_{\divr_k;\O}
&:= \|\bB{w}\|_{0;\O} + \|\divr(\bB{w})\|_{0,k;\O}
&&\qquad
\forall\, \bB{w}\in\mbf{H}(\divr_k;\Omega)\,, \label{def:norm:hdiv:vector} \\[1ex]
\|\bB{\tau}\|_{\mbf{div}_k;\O}
&:=
\|\bB{\tau}\|_{0;\O} + \|\mbf{div}(\bB{\tau})\|_{0,k;\O}
&&\qquad
\forall\, \bB{\tau}\in\mbb{H}(\mbf{div}_k;\Omega)\,. \label{def:norm:hdiv:tensor}
\end{alignat}
For the particular case $k=2$, we simply write $\mbf{H}(\divr;\Omega)$ and $\mbb{H}(\mbf{div};\Omega)$ for \eqref{def:hdiv:vector} and \eqref{def:hdiv:tensor}, respectively, with the associated norms $\|\cdot\|_{\divr;\O}$ and
$\|\cdot\|_{\mbf{div};\O}$. In case of no ambiguity regarding the functional space, we directly write the norms as $\|\cdot\|$ without
subscripts. Additionally, given a subset of the boundary $\Gamma_\star\subseteq\Gamma$, we define the following subspaces
\begin{align*}
\bH_{\Gamma_\star}(\div_k;\O) :=
\big\{ \bB{w}\in\bH(\div_k;\O): \, \bB{w}\cdot\bB{\nu}=0 \text{ on }\Gamma_\star \big\}\,,\qquad
\mathrm{H}^{1/2}(\Gamma_\star) := \left\{ \gamma_0(v)|_{\Gamma_\star}:\, v\in\mathrm{H}^1(\O) \right\},
\end{align*}
where $\gamma_0:\mathrm{H}^1(\O)\to\mathrm{H}^{1/2}(\Gamma)$ denotes the trace operator. The trace space $\mathrm{H}^{1/2}(\Gamma_\star)$ is endowed with the quotient norm
\begin{align*}
\|\chi\|_{1/2,\Gamma_\star}
:= \inf \big\{ \|v\|_{1;\O}: \quad v\in\mathrm{H}^1(\O), \quad \gamma_0(v)|_{\Gamma_\star}=\chi \big\}\,,
\end{align*}
and its dual space is denoted by $\mathrm{H}^{-1/2}(\Gamma_\star) := \bigl(\mathrm{H}^{1/2}(\Gamma_\star)\bigr)'$. The corresponding vector-valued trace spaces are defined componentwise. In addition, for every $\chi\in\mathrm{H}^{1/2}(\Gamma_\star)$, we define
\begin{align}
\left\langle
\bB{w}\cdot\bB{\nu},\chi \right\rangle_{\Gamma_\star}
:=
\left\langle \bB{w}\cdot\bB{\nu},\widetilde{\chi} \right\rangle_\Gamma\qquad  \forall \bB{w}\in \bH_{\Gamma\setminus\Gamma_\star}(\div_k;\O)\,, \label{def:restricted:duality}
\end{align}
where $\langle\cdot,\cdot\rangle_\Gamma$ denotes the duality pairing between $\mathrm{H}^{-1/2}(\Gamma)$ and $\mathrm{H}^{1/2}(\Gamma)$, or between their vector-valued counterparts, and $\widetilde{\chi}\in\mathrm{H}^{1/2}(\Gamma)$ is any extension of $\chi$ to $\Gamma$, as $\bB{w}\cdot\bB{\nu}=0$ on $\Gamma\setminus\Gamma_\star$. Moreover, there exists a constant
$C_{\Gamma_\star}>0$ such that
\begin{align*}
\left| \left\langle \bB{w}\cdot\bB{\nu},\chi \right\rangle_{\Gamma_\star} \right|
\leq
C_{\Gamma_\star} \|\bB{w}\|_{\div_k;\O} \|\chi\|_{1/2,\Gamma_\star}\,.
\end{align*}
The following useful integration by parts formulas for functions in the spaces $\mbf{H}(\divr_k;\Omega)$ and $\mbb{H}(\mbf{div}_k;\Omega)$ with $k\geq2d/(d+2)$ were introduced in \cite{Caucao2023}:
\begin{alignat}{2}
\left\langle
\bB{w}\cdot\bB{\nu},\varphi \right\rangle_\Gamma
& = \int_\O \left( \bB{w}\cdot\nabla\varphi + \varphi\,\bdiv(\bB{w}) \right)
&&\qquad \forall\, (\bB{w},\varphi) \in \bH(\div_k;\O)\times\mathrm{H}^1(\O)\,, \label{eq:IbyP:Hdivp:vector} \\
\left\langle \bB{\tau}\bB{\nu},\bB{v} \right\rangle_\Gamma
&= \int_\O \left( \bB{\tau}:\nabla\bB{v} + \bB{v}\cdot\bdiv(\bB{\tau}) \right)
&&\qquad \forall\, (\bB{\tau},\bB{v}) \in \mH(\bdiv_k;\O)\times\mbf{H}^1(\O)\,, \label{eq:IbyP:Hdivp:tensor}
\end{alignat}
where the colon ``$:$'' denotes the matrix inner product. Duality pairings over a relatively open subset $\Gamma_\star\subseteq\Gamma$ are understood according to \eqref{def:restricted:duality}.

Given a tensor $\bB{\tau}\in\R^{d\times d}$, we let $\bB{\tau}^{\rm t}$ be its transpose, and define
the deviatoric part of $\bB{\tau}$, denoted by the superscript ``$\rmd$'', by $\bB{\tau}^{\rmd} := \bB{\tau} -  d^{-1}\tr(\bB{\tau})\mbb{I}$, where $\mbb{I}$ is the identity matrix of order $d$. The spaces of
$\mathrm{L}^2$-symmetric, $\mathrm{L}^2$-skew-symmetric and zero-trace tensors are respectively denoted by $\mathbb{L}^2_{\texttt{sym}}(\Omega)$, $\mathbb{L}^2_{\texttt{skew}}(\Omega)$ and $\Ltwotr $. Furthermore, we denote by $\mathrm{L}^k_0(\Omega)$ the space of zero-mean functions in $\mathrm{L}^k(\Omega)$.

\subsection{Governing equations} \label{sec:governing:equations}

We start by introducing the unknowns and physical quantities of interest in our model. We define $\phi:=\phi(\bB{x})$ as the volume fraction of solid particles, where $\bB{x}$ denotes a coordinate point in $\Omega$. Then, for our monodispersed model, the volume fraction corresponding to the fluid phase at each $\bB{x}$ is given by $1-\phi(\bB{x})$. For the velocity field, we introduce the volume-average velocity as $\bB{u} := \bB{u}(\bB{x})$, which is a vector field in $\R^d$, and the unknown $p:=p(\bB{x})$ represents the excess pore pressure. The temperature, meanwhile, is denoted by $\theta = \theta(\bB{x})$. The governing model is stated as the following system of coupled PDEs:
\begin{subequations} \label{syst:main}
\begin{alignat}{2}
\delta_1\phi - \div\big(-\phi\bB{u} - f_{\rm bk}(\phi)\bB{k} + \mathcal{D}_1(\phi)\nabla\phi\big) & = l_1&& \text{ in }\Omega\,, \label{eq:main:mass}\\[1ex]
\delta_2\theta - \div\big(-\theta\bB{u} + \mathcal{D}_2(\phi)\nabla\theta\big) & = l_2 && \text{ in }\Omega\,,\label{eq:main:heat}\\[1ex]
(\delta_3\mathbb{I}+\delta_4\mbf{K}^{-1}(\bB{x},\phi))\bB{u} - \bdiv\big(\mu(\phi,\theta)\bB{e}(\bB{u}) - p\I\big) & =  \bB{l}_3+\mbf{f}(\phi,\theta)\quad && \text{ in }\Omega\,, \label{eq:main:flow1}\\[1ex]
\div(\bB{u}) & = 0 && \text{ in }\Omega\,, \label{eq:main:flow2}
\end{alignat}
\end{subequations}
where $f_{\rm bk}:= f_{\rm bk}(\phi)$ is the Kynch batch flux density function \cite{Kynch1952}, $\mathcal{D}_1: = \mathcal{D}_1(\phi)$ is a positive nonlinear diffusion term modelling compression effects and $\bB{k}$ is the unitary vector pointing in the direction of gravity. Function $\mathcal{D}_2:=\mathcal{D}_2(\phi)$ models the thermal diffusivity, which is assumed to nonlinearly depend on $\phi$. At the flow equation \eqref{eq:main:flow1}, function $\mu:= \mu(\phi,\theta)$ corresponds to the viscosity of the mixture, $\mathbf{K}:=\mathbf{K}(\bB{x},\phi)$ is the permeability tensor \cite{Alvarez2026}, whose inverse is symmetric and assumed to be uniformly positive definite and $\mathbf{K}^{-1}(\bB{x},\phi)\in \mbb{L}^{\infty}$ for all $\phi$ and $\bB{x}\in \Omega$, and $\mathbf{f}:=\mathbf{f}(\phi,\theta)$ is a source term depending on both $\phi$ and $\theta$ \cite{Careaga2025a,Reyes2019}. Functions $l_1:=l_1(\bB{x})$, $l_2 = l_2(\bB{x})$ and $\bB{l}_3 = \bB{l}_3(\bB{x})$ are space dependent source terms and $\delta_1$, $\delta_2$, $\delta_3$ and $\delta_4$ are non-negative real constants. System~\eqref{syst:main} is supplemented with the following boundary conditions:
\begin{subequations} \label{syst:bcs}
\begin{alignat}{2}
 \big(-\phi\bB{u} - f_{\rm bk}(\phi)\bB{k} + \mathcal{D}_1(\phi)\nabla\phi\big) \cdot \bB{\nu}  & = \Phi &&\qquad\qon\Gamma\,, \label{eq:bc:mass}\\[1ex]
 \big(-\theta\bB{u} + \mathcal{D}_2(\phi)\nabla\theta \big)\cdot \bB{\nu} & = 0    &&\qquad\qon\Gamma_{\rm N}\,,\label{eq:bc:heat:1}\\[1ex]
 \theta  & = \theta_{\rm D} &&\qquad\qon \Gamma_{\rm D}\,, \label{eq:bc:heat:2}\\[1ex]
  \bB{u} & = \bB{u}_{\rm D} &&\qquad\qon \Gamma\,, \label{eq:bc:u:1}
\end{alignat}
\end{subequations}
where $\Phi\in \mathrm H^{-1/2}(\Gamma)$ and $\bB{u}_{\rm D}\in \mbf{H}^{1/2}(\Gamma)$.
Now, noticing that $\delta_1 \phi - l_1$, $\delta_2\theta - l_2$ and $\delta_3 \bB{u} - \bB{l}_3$ may be related to a backward Euler time discretization of $\partial_t \phi$, $\partial_t \theta$ and $\partial_t \bB{u}$, respectively, the model equations~\eqref{syst:main} can be interpreted as a quasi-stationary system of PDEs. In what follows, we will indistinguishably refer to Equation~\eqref{eq:intro:eq1} as the concentration or mass equation as it also corresponds to the mass balance for the solid phase.

To achieve the solvability analyses in the upcoming sections, we introduce some assumptions for the model functions.
\begin{assumption}
The model functions in system~\eqref{syst:main} satisfy the following bounds:
\begin{itemize}
 \item[1.] There exist positive real constants  $f_{\max}$, $d_1$, $d_2$, $D_1$ and $D_2$ such that for all $\phi$, functions $f_{\rm bk}$, $\mathcal{D}_1$ and $\mathcal{D}_2$ are bounded by:
\begin{align}
0 \, \leq \, f_{\rm bk}(\phi) \, \leq f_{\rm max}\,,\qquad
d_{1} \, \leq \, \mathcal{D}_1(\phi) \, \leq  D_1\,,\qquad
d_{2} \, \leq \, \mathcal{D}_2(\phi) \, \leq \, D_2\,. \label{eq:data:bounds}
\end{align}
\item[2.] There exist positive real constants $\kappa_{\min}$ and $\kappa_{\max}$ such that for all $\phi$, tensor $\mbf{K}^{-1}$ satisfies
\begin{align}
\kappa_{\min}  |\bB{y}|^2 \leq \bB{y}^{\tt t}\mbf{K}^{-1}(\bB{x},\phi)\bB{y}\leq \kappa_{\max}|\bB{y}|^2 \qquad \forall\bB{x}\in \Omega,\,\forall \bB{y}\in {\rm R}^{d}\,.
\end{align}
\item[3.] There exist positive real (Lipschitz) constants $L_{\rm fbk}$, $L_{\mathcal{D}_1}$, $L_{\mathcal{D}_2}$, $L_{\mu}$ and $L_{\mbf{K}}$, such that functions $f_{\rm bk}$, $\mathcal{D}_1$, $\mathcal{D}_2$, $\mu$ and tensor function $\mbf{K}^{-1}$, satisfy the following Lipschitz continuity inequalities for all $\phi_1$, $\phi_2$, $\theta_1$ and $\theta_2$:
\begin{equation}\label{eq:LC:data}
\begin{aligned}
 |f_{\rm bk}(\phi_1)-f_{\rm bk}(\phi_2)| & \leq L_{\rm fbk}|\phi_1-\phi_2|\,,\\
 |\mathcal{D}_i(\phi_1)-\mathcal{D}_i(\phi_2)| &\, \leq \, L_{\mathcal{D}_i}|\phi_1-\phi_2|\,,\quad \text{for }i\in\{1,2\}\,,\\
 |\mu(\phi_1,\theta_1)-\mu(\phi_2,\theta_2)| & \leq L_{\mu}\left(|\phi_1-\phi_2|+|\theta_1-\theta_2|\right)\,,\\
 |\mbf{K}^{-1}(\bB{x},\phi_1)-\mbf{K}^{-1}(\bB{x},\phi_2)| & \, \leq \, L_{\mbf{K}}|\phi_1-\phi_2|\,.
\end{aligned}
\end{equation}
\end{itemize}
\end{assumption}

\section{Continuous saddle-point problem formulations} \label{sec:continuous:formulation}

\subsection{Concentration equation}
\label{sec:CF:concentration}

To derive a mixed formulation of Equation \eqref{eq:main:mass}, we introduce two additional unknowns: the flux $\vsigmass$, given by the expression inside the divergence operator in \eqref{eq:main:mass}, and the concentration gradient $\tmass$. The purpose of introducing these unknowns is to seek the concentration $\phi$ in a weaker space than $\mathrm{H}^1(\O)$. Thus, \eqref{eq:main:mass} and \eqref{eq:bc:mass} can be equivalently rewritten as the following system for $(\tmass,\vsigmass,\phi)$:
\begin{subequations} \label{syst:mass:decoupled}
\begin{align}
\vsigmass = -\phi\bB{u} - \fbk(\phi)\bB{k} +\mathcal{D}_1(\phi)\tmass\,,
\qquad
\tmass = \nabla\phi\,,\qquad \delta_1\phi-\div(\vsigmass) = l_1 \qin\O\,, & \label{eq:PDEmixed:mass} \\
\vsigmass\cdot\bB{\nu} = \Phi \qon\Gamma\,.& \label{cc:PDEmixed:mass}
\end{align}
\end{subequations}
We now identify appropriate functional spaces for the unknowns in system \eqref{syst:mass:decoupled}.
Testing the first equation in \eqref{eq:PDEmixed:mass} against
$\bB{s}\in\bL^2(\O)$ yields
\begin{align}
- \int_\O\vsigmass\cdot\bB{s} - \int_\O\phi\bB{u}\cdot\bB{s} + \int_\O\mathcal{D}_1(\phi)\tmass\cdot\bB{s}
= \int_\O\fbk(\phi)\bB{k}\cdot\bB{s} \qquad \forall\,\bB{s}\in\bL^2(\O)\,. \label{syst:weak:mass:aux}
\end{align}
For the convective term in the second integral on the left-hand side of \eqref{syst:weak:mass:aux} to be well
defined, we require $\phi\bB{u}\in\bL^2(\O)$. This condition is satisfied by seeking $\phi\in\L^\rho(\O)$ and $\bB{u}\in\bL^r(\O)$, with $\rho$ and $r$ such that $\tfrac{1}{\rho} + \tfrac{1}{r} = \tfrac{1}{2}$ for $r,\rho>2$. Indeed, thanks to H\"older inequality, we have
\begin{align*}
\|\phi\bB{u}\|_{0;\O} \leq \|\phi\|_{0,\rho;\O} \|\bB{u}\|_{0,r;\O}\,.
\end{align*}
Next, we test equation $\tmass=\nabla\phi$ against $\bB{\varsigma}\in\bH(\div_\varrho;\O)$, and applying the integration by parts formula \eqref{eq:IbyP:Hdivp:vector}, we formally obtain
\begin{align}
\int_\O\tmass\cdot\bB{\varsigma} = \int_\O\nabla\phi\cdot\bB{\varsigma} = -\int_\O\phi\,\div(\bB{\varsigma}) + \left\langle
\bB{\varsigma}\cdot\bB{\nu}, \gamma_0(\phi)\right\rangle_\Gamma\,, \label{eq:t1:IbyP}
\end{align}
where $\varrho$ corresponds to the conjugate exponent of $\rho$. This identity is well defined provided that $\div(\bB{\varsigma})\in\L^\varrho(\O)$ and the normal trace $\bB{\varsigma}\cdot\bB{\nu}$ belongs to $\mathrm{H}^{-1/2}(\Gamma)$. The latter follows from the continuous embedding $\mathrm{H}^1(\O)\hookrightarrow\L^\rho(\O)$, which holds for every finite $\rho\geq1$ when $d=2$, and for $1\leq\rho\leq6$ when $d=3$. Note that the continuous embedding $\bH^1(\O)\hookrightarrow\bL^r(\O)$ also holds for every finite $r\geq1$ when $d=2$, and for $1\leq r\leq6$ when $d=3$. Therefore, defining $s$ and $\varrho$ as the conjugate exponents of $r$ and $\rho$, respectively, the admissible ranges of the Sobolev exponents depending on the spatial dimension are
\begin{align}
\big\{ r,\rho\in(2,+\infty),\quad s,\varrho\in(1,2)\quad\text{if }d=2\big\}\quad\text{or}\quad
\big\{r,\rho\in[3,6],\quad s,\varrho\in[6/5,3/2]\quad\text{if }d=3\big\}\,.\label{def:ranges}
\end{align}
We now introduce the boundary unknown $\xi:=-\gamma_0(\phi)\in\mathrm{H}^{1/2}(\Gamma)$, where $\gamma_0$ denotes the trace operator on $\mathrm{H}^1(\O)$ \cite{Gatica2014}. Consequently, \eqref{eq:t1:IbyP} can be equivalently written as
\begin{align}
\int_\O\tmass\cdot\bB{\varsigma} + \int_\O\phi\,\div(\bB{\varsigma}) + \left\langle \bB{\varsigma}\cdot\bB{\nu},\xi \right\rangle_\Gamma
= 0 \qquad \forall\,\bB{\varsigma}\in\bH(\div_\varrho;\O)\,. \label{eq:weak:gradient:mass}
\end{align}
In turn, testing the third equation in \eqref{eq:PDEmixed:mass} against $\psi\in\L^\rho(\O)$, and the boundary condition \eqref{cc:PDEmixed:mass} against $\chi\in\mathrm{H}^{1/2}(\Gamma)$, we obtain
\begin{subequations} \label{syst:saddle:mass}
\begin{alignat}{2}
\int_\O\psi\,\div(\vsigmass) - \int_\O\delta_1\phi\psi & = -\int_\O l_1\psi
&&\qquad
\forall\,\psi\in\L^\rho(\O)\,, \\[1ex]
\left\langle \vsigmass\cdot\bB{\nu},\chi \right\rangle_\Gamma &= \left\langle\Phi,\chi\right\rangle_\Gamma
&&\qquad \forall\,\chi\in\mathrm{H}^{1/2}(\Gamma)\,.
\end{alignat}
\end{subequations}
To formulate equations \eqref{syst:weak:mass:aux}, \eqref{eq:weak:gradient:mass} and \eqref{syst:saddle:mass} as a saddle-point problem, we introduce the product spaces $\bH_1 := \bL^2(\O)\times\bH(\divr_\varrho;\O)$ and $\bQ_1 := \L^\rho(\O)\times\mathrm{H}^{1/2}(\Gamma)$,
and the following trial and test vectors:
\begin{alignat*}{4}
\vtmass &:= (\tmass,\vsigmass)\in\bH_1      &&\qan&
\vsmass &:= (\bB{s},\bB{\varsigma})\in\bH_1\,, \\
\vec{\bB{\phi}} &:= (\phi,\xi)\in\bQ_1      &&\qan&
\vec{\bB{\psi}} &:= (\psi,\chi)\in\bQ_1\,.
\end{alignat*}
Then, the weak formulation of \eqref{syst:mass:decoupled}, and consequently that of decoupled Equation \eqref{eq:main:mass} together with \eqref{eq:bc:mass}, is given by the following nonlinear perturbed saddle-point problem: Given $\bB{u}\in\bL^r(\O)$, find $(\vtmass,\vec{\bB{\phi}})\in\bH_1\times\bQ_1$ such that
\begin{equation}
\label{syst:weak:mass}
\begin{array}{cclclcll}
a_\phi\big(\vtmass,\vsmass\big) &+& b\big(\vsmass,\vec{\bB{\phi}}\big) &-&
\widetilde b_{\bB{u}} \big(\vsmass,\vec{\bB{\phi}}\big) &=& \mathcal{F}_\phi(\vsmass)
& \qquad \forall\,\vsmass\in\bH_1\,, \\[1ex]
b\big(\vtmass,\vec{\bB{\psi}}\big) &-& c\big(\vec{\bB{\phi}},\vec{\bB{\psi}}\big)
&& &=& \mathcal{G}\big(\vec{\bB{\psi}}\big) &\qquad \forall\,\vec{\bB{\psi}}\in\bQ_1\,.
\end{array}
\end{equation}
where for given $\varphi\in\L^\rho(\O)$ and $\bB{w}\in\bL^r(\O)$, the bilinear forms
$a_\varphi:(\bH_1)^2\to\R$, $b, b_{\bB{w}}:\bH_1\times\bQ_1\to\R$ and $c:(\bQ_1)^2\to\R$, and functionals $\mathcal{F}_\varphi:\bH_1\to\R$ and $\mathcal{G}:\bQ_1\to\R$ are defined by
\begin{subequations}
\label{def:bforms:mass}
\begin{alignat}{2}
a_\varphi\big(\vtmass,\vsmass\big) :={}
\int_\O \mathcal{D}_1(\varphi)\tmass\cdot\bB{s} -\int_\O\vsigmass\cdot\bB{s} +\int_\O\tmass\cdot\bB{\varsigma}, &\qquad &&
b\big(\vsmass,\vec{\bB{\psi}}\big) :={}
\int_\O\psi\,\div(\bB{\varsigma}) + \left\langle \bB{\varsigma}\cdot\bB{\nu},\chi \right\rangle_\Gamma\,, \label{def:a:b:mass}\\ 
\widetilde b_{\bB{w}}\big(\vsmass,\vec{\bB{\psi}}\big) :={}
\int_\O\psi\,\bB{w}\cdot\bB{s}, &\qquad&& 
c\big(\vec{\bB{\phi}},\vec{\bB{\psi}}\big) :={}
\int_\O\delta_1\phi\psi\,, \label{def:btilde:c:mass} \\ 
\mathcal{F}_\varphi(\vsmass) :={}
\int_\O \fbk(\varphi)\bB{k}\cdot\bB{s}, &\qquad&& 
\mathcal{G}\big(\vec{\bB{\psi}}\big) :={}
-\int_\O l_1\psi + \left\langle\Phi,\chi\right\rangle_\Gamma\,, \label{def:F:G:mass}  
\end{alignat}
\end{subequations}
for all $\vtmass,\vsmass\in\bH_1$ and $\vec{\bB{\phi}},\vec{\bB{\psi}}\in\bQ_1$.
The product spaces $\bH_1$ and $\bQ_1$ are endowed with the norms
\begin{align*}
\|\vsmass\|_{\bH_1} &:= \|\bB{s}\|_{0;\O} + \|\bB{\varsigma}\|_{\div_\varrho;\O}
\quad
\forall\,\vsmass\in\bH_1\,, \qquad
\|\vec{\bB{\psi}}\|_{\bQ_1} := \|\psi\|_{0,\rho;\O} + \|\chi\|_{1/2,\Gamma}
\quad
\forall\,\vec{\bB{\psi}}\in\bQ_1\,.
\end{align*}
Thanks to Cauchy--Schwarz and H\"older inequalities, together with the boundedness of the normal trace operator and the continuous injection $i_\rho:\mathrm{H}^1(\O)\hookrightarrow\L^\rho(\O)$, we have
\begin{subequations} \label{eq:bounded:mass}
\begin{alignat}{5}
| a_\varphi\big(\vtmass,\vsmass\big)|
&\leq \|a\| \|\vtmass\|_{\bH_1} \|\vsmass\|_{\bH_1}\,, \qquad &
\big| b\big(\vsmass,\vec{\bB{\psi}}\big) \big|
&\leq \|b\| \|\vsmass\|_{\bH_1} \|\vec{\bB{\psi}}\|_{\bQ_1}\,, &&\\
| \widetilde b_{\bB{w}} \big(\vsmass,\vec{\bB{\psi}}\big) |
&\leq \|\bB{w}\|_{0,r;\O} \|\vsmass\|_{\bH_1} \|\vec{\bB{\psi}}\|_{\bQ_1}\,, \qquad &
\big| c\big(\vec{\bB{\phi}},\vec{\bB{\psi}}\big) \big|
&\leq \|c\| \|\vec{\bB{\phi}}\|_{\bQ_1} \|\vec{\bB{\psi}}\|_{\bQ_1}\,, &&\\
\big| \mathcal{F}_\varphi(\vsmass) \big|
&\leq \|\mathcal{F}\| \|\vsmass\|_{\bH_1}\,, \qquad &
\big| \mathcal{G}\big(\vec{\bB{\psi}}\big) \big|
&\leq \|\mathcal{G}\| \|\vec{\bB{\psi}}\|_{\bQ_1}\,, &&
\end{alignat}
\end{subequations}
for all $\vtmass,\vsmass\in \bH_1$ and $\vec{\bB{\phi}},\vec{\bB{\psi}}\in \bQ_1$, where the constants above are
\begin{alignat*}{2}
\|a\| &:= \max\{1,D_1\}\,,\qquad &
\|b\| &:= \max\{1,\|i_\rho\|\}\,,\qquad
\|c\| := \delta_1|\O|^{(\rho-2)/\rho}\,, \\
\|\mathcal{F}\| &:= f_{\max}|\O|^{1/2}\,,\qquad &
\|\mathcal{G}\| &:= |\O|^{(\rho-2)/(2\rho)} \|l_1\|_{0;\O} + \|\Phi\|_{-1/2,\Gamma}\,.
\end{alignat*}

\subsection{Heat equation} \label{sec:CF:heat}

We begin by observing that the heat equation \eqref{eq:main:heat} has a structure similar to that of the concentration equation \eqref{eq:main:mass}. Therefore, we follow the same approach to derive its mixed formulation. We introduce the heat flux $\vsigheat$, given by the expression inside the divergence operator in \eqref{eq:main:heat}, and the temperature gradient $\theat$. Thus, the triplet $(\theat,\vsigheat,\theta)$ satisfies
\begin{align}
\vsigheat = -\theta\bB{u} + \mathcal{D}_2(\phi)\theat\,,
\qquad
\theat = \nabla\theta\,,
\qquad
\delta_2\theta - \div(\vsigheat) = l_2 &\qin \O\,,  \label{eq:PDEmixed:heat}\\
\theta = \theta_{\rm D} \qon\Gamma_{\rm D}\,,
\qquad
\vsigheat\cdot\bB{\nu} = 0 &\qon\Gamma_{\rm N}\,. \label{cc:PDEmixed:heat}
\end{align}
Testing the three equations in \eqref{eq:PDEmixed:heat} against $\bB{q}\in\bL^2(\O)$, $\widetilde{\bB{\tau}}\in
\bH_{\Gamma_{\rm N}}(\div_\varrho;\O)$, and $\varkappa\in\L^\rho(\O)$, respectively, and applying integration by parts to the term containing $\nabla\theta$, we obtain
\begin{subequations} 
\begin{alignat*}{2}
- \int_\O\vsigheat\cdot\bB{q} - \int_\O\theta\bB{u}\cdot\bB{q} + \int_\O\mathcal{D}_2(\phi)\theat\cdot\bB{q}
&= 0 && \qquad \forall\,\bB{q}\in\bL^2(\O)\,, \\
\int_\O\theat\cdot\widetilde{\bB{\tau}} + \int_\O\theta\,\div(\widetilde{\bB{\tau}})
&= \left\langle \widetilde{\bB{\tau}}\cdot\bB{\nu}, \theta_{\rm D} \right\rangle_{\Gamma_{\rm D}}
&&\qquad \forall\, \widetilde{\bB{\tau}} \in\bH_{\Gamma_{\rm N}}(\div_\varrho;\O)\,,\\
\int_\O\varkappa\,\div(\vsigheat) - \int_\O\delta_2\theta\varkappa
&= -\int_\O l_2\varkappa
&&\qquad \forall\,\varkappa\in\L^\rho(\O)\,.
\end{alignat*}
\end{subequations}
We now introduce the product space $\bH_2 := \bL^2(\O) \times \bH_{\Gamma_{\rm N}}(\div_\varrho;\O)$,
and the trial and test vectors
\begin{align*}
\vtheat := (\theat,\vsigheat)\in\bH_2\,,   \qquad
\vsheat := (\bB{q},\widetilde{\bB{\tau}})\in\bH_2\,.
\end{align*}
The weak formulation of \eqref{eq:PDEmixed:heat}--\eqref{cc:PDEmixed:heat}, and consequently that of \eqref{eq:main:heat}, is given by the following perturbed saddle-point problem: Given $\bB{u}\in\bL^r(\O)$ and $\phi\in\L^\rho(\O)$, find $(\vtheat,\theta)\in\bH_2\times\L^\rho(\O)$ such that
\begin{equation}
\label{syst:weak:heat}
\begin{array}{cclclcll}
\mfk{a}_\phi\big(\vtheat,\vsheat\big)
&+&
\mfk{b}\big(\vsheat,\theta\big)
&-&
\mfk{\widetilde{b}}_{\bB{u}}\big(\vsheat,\theta\big)
&=&
\mfk{F}(\vsheat)
&\qquad
\forall\,\vsheat\in\bH_2\,, \\[1ex]
\mfk{b}\big(\vtheat,\varkappa\big)
&-&
\mfk{c}(\theta,\varkappa)
&&
&=&
\mfk{G}(\varkappa)
&\qquad
\forall\,\varkappa\in\L^\rho(\O)\,,
\end{array}
\end{equation}
where for $\varphi\in\L^\rho(\O)$ and $\bB{w}\in\bL^r(\O)$, the bilinear forms $\mfk{a}_\varphi:(\bH_2)^2\to\R$, $\mfk{b},\mfk{\widetilde{b}}_{\bB{w}}:\bH_2\times\L^\rho(\O)\to\R$ and $\mfk{c}:(\L^\rho(\O))^2\to\R$, and functionals $\mfk{F}:\bH_2\to\R$ and $\mfk{G}:\L^\rho(\O)\to\R$ are defined by
\begin{subequations}
\begin{alignat*}{2}
\mfk{a}_\varphi\big(\vtheat,\vsheat\big)
:={} \int_\O\mathcal{D}_2(\varphi)\theat\cdot\bB{q} - \int_\O\vsigheat\cdot\bB{q} + \int_\O\theat\cdot\widetilde{\bB{\tau}}\,, &\qquad &&
\mfk{b}(\vsheat,\varkappa)
:={} \int_\O\varkappa\,\div(\widetilde{\bB{\tau}})\,,\\
\mfk{c}(\theta,\varkappa)
:={} \int_\O\delta_2\theta\varkappa\,, &\qquad &&
\mfk{\widetilde{b}}_{\bB{w}}(\vsheat,\varkappa)
:={} \int_\O\varkappa\bB{w}\cdot\bB{q}\,,\\
\mfk{F}(\vsheat)
:={} \left\langle \widetilde{\bB{\tau}}\cdot\bB{\nu}, \theta_{\rm D} \right\rangle_{\Gamma_{\rm D}},&\qquad &&
\mfk{G}(\varkappa)
:={} -\int_\O l_2\varkappa\,,
\end{alignat*}
\end{subequations}
for all $\vtheat,\vsheat\in\bH_2$ and $\theta,\varkappa\in\L^\rho(\O)$. The product space $\bH_2$ is endowed with the norm
\begin{align*}
\|\vsheat\|_{\bH_2}
:=
\|\bB{q}\|_{0;\O}
+
\|\widetilde{\bB{\tau}}\|_{\div_\varrho;\O}
\qquad
\forall\,\vsheat\in\bH_2\,.
\end{align*}
Similarly to the bounds found in \eqref{eq:bounded:mass}, applying the Cauchy--Schwarz and H\"older inequalities, the boundedness of the normal trace operator and the continuous injection $i_\rho$, among others, we obtain the bounds
\begin{subequations}
\label{eq:bounded:heat}
\begin{alignat}{3}
| \mfk{a}_\varphi\big(\vtheat,\vsheat\big)|
&\leq \|\mfk{a}\| \|\vtheat\|_{\bH_2} \|\vsheat\|_{\bH_2}\,,  &\qquad &&
| \mfk{b}\big(\vsheat,\varkappa\big) |
&\leq \|\mfk{b}\| \|\vsheat\|_{\bH_2} \|\varkappa\|_{0,\rho;\O}\,,\\
| \mfk{\widetilde{b}}_{\bB{w}} \big(\vsheat,\varkappa\big) |
&\leq \|\mfk{\widetilde{b}}\| \|\bB{w}\|_{0,r;\O} \|\vsheat\|_{\bH_2} \|\varkappa\|_{0,\rho;\O}\,, &\qquad &&
| \mfk{c}(\theta,\varkappa) |
&\leq \|\mfk{c}\| \|\theta\|_{0,\rho;\O} \|\varkappa\|_{0,\rho;\O}\,,\\
| \mfk{F}(\vsheat) |
&\leq \|\mfk{F}\| \|\vsheat\|_{\bH_2}\,, &\qquad &&
|\mfk{G}(\varkappa)|
&\leq \|\mfk{G}\| \|\varkappa\|_{0,\rho;\O}\,,
\end{alignat}
\end{subequations}
for all $\vtheat,\vsheat\in\bH_2$ and $\theta,\varkappa\in\L^\rho(\O)$, where the constants above are
\begin{alignat*}{2}
\begin{aligned}
\|\mfk{a}\| &:= \max\{1,D_2\}, \quad &
\|\mfk{b}\| &:= 1, \quad
\|\mfk{\widetilde{b}}\| := 1, \quad
\|\mfk{c}\| := \delta_2|\O|^{(\rho-2)/\rho},\\
\|\mfk{F}\| &:= \max\{1,\|i_\rho\|\}, \|\theta_{\rm D}\|_{1/2,\Gamma_{\rm D}},\quad &
\|\mfk{G}\| &:= |\O|^{(\rho-2)/(2\rho)} \|l_2\|_{0;\O}.
\end{aligned}
\end{alignat*}

\subsection{Flow equations} \label{sec:CF:flow}

The third system corresponds to the flow equations given by the momentum balance \eqref{eq:main:flow1} and the continuity equation \eqref{eq:main:flow2}. We follow the mixed formulation introduced in \cite{Gatica2023} in which the pressure $p$ can be eliminated from the formulation. To this end, we introduce the symmetric part of the velocity gradient $\bB{\zeta}:=\bB{\varepsilon}(\bB{u})$, vorticity $\bB{\gamma}:=\frac12\left(\nabla\bB{u}-\nabla\bB{u}^{\tt t}\right)$, and the tensor associated with the viscous stress $\bB{\sigma}\in \Hdivtr$, and rewrite \eqref{eq:main:flow1}--\eqref{eq:main:flow2} equivalently as follows:
\begin{subequations}
\label{syst:PDEmixed:flow}
\begin{alignat}{2}
\left( \delta_3\I + \delta_4\mbf{K}^{-1}(\phi) \right)\bB{u} - \bdiv(\bB{\sigma})
&= \bB{l}_3 + \mbf{f}(\phi,\theta) && \qin\O\,, \label{eq:PDEmixed:flow:momentum} \\
\bB{\zeta} &= \nabla\bB{u}-\bB{\gamma} && \qin\O\,, \label{eq:PDEmixed:flow:gradient} \\
\bB{\sigma}^{\rmd} &= \mu(\phi,\theta)\bB{\zeta} && \qin\O\,, \label{eq:PDEmixed:flow:constitutive} \\
\bB{u} &= \bB{u}_{\rm D} && \qon\Gamma\,. \label{cc:PDEmixed:flow}
\end{alignat}
\end{subequations}
Testing \eqref{eq:PDEmixed:flow:momentum}--\eqref{eq:PDEmixed:flow:constitutive} against $\bB{w}\in\bL^r(\O)$, $\bB{\eta}\in\Ltwotr$, and $\bB{\tau}\in\Hdivtr$, respectively, integrating over $\O$, and applying integration by parts, we obtain the following weak formulation:
\begin{subequations}
\label{syst:pre:flow}
\begin{alignat}{2}
\int_\O \bB{w}\cdot\bdiv(\bB{\sigma}) - \int_\O \left( \delta_3\I + \delta_4\mbf{K}^{-1}(\phi) \right) \bB{u}\cdot\bB{w}
& = -\int_\O (\bB{l}_3+\mbf{f}(\phi,\theta))\cdot\bB{w}
&&\qquad
\forall\,\bB{w}\in\bL^r(\O)\,, \label{eq:pre:flow:1} \\[1ex]
-\int_\O \bB{\sigma}^{\rmd}:\bB{\eta} + \int_\O \mu(\phi,\theta)\bB{\zeta}:\bB{\eta}
& = 0
&&\qquad
\forall\,\bB{\eta}\in\Ltwotr\,, \label{eq:pre:flow:2} \\[1ex]
\int_\O \bB{\zeta}:\bB{\tau}^{\rmd} + \int_\O \bB{u}\cdot\bdiv(\bB{\tau}) + \int_\O \bB{\gamma}:\bB{\tau}
& = \left\langle \bB{\tau}\bB{\nu}, \bB{u}_{\rm D} \right\rangle_\Gamma
&&\qquad
\forall\,\bB{\tau}\in\Hdivtr\,, \label{eq:pre:flow:3} \\[1ex]
\int_\O \bB{\sigma}:\widetilde{\bB{\gamma}} &= 0
&&\qquad
\forall\, \widetilde{\bB{\gamma}} \in\mbb{L}^2_{\tt skew}(\O)\,. \label{eq:pre:flow:4}
\end{alignat}
\end{subequations}
The last equation imposes the symmetry of $\bB{\sigma}$ weakly, with $\widetilde{\bB{\gamma}} $ test function. In turn, to rewrite \eqref{syst:pre:flow} as a saddle-point problem, we introduce the product spaces $\mX := \Ltwotr\times\Hdivtr$ and $\mbb{M} := \bL^r(\O)\times\mbb{L}^2_{\tt skew}(\O)$,
and the following trial and test vectors:
\begin{alignat*}{3}
\vec{\bB{\zeta}} &:= (\bB{\zeta},\bB{\sigma})\in\mX   &&\qan&
\vec{\bB{\eta}}  &:= (\bB{\eta},\bB{\tau})\in\mX\,, \\
\vec{\bB{u}}     &:= (\bB{u},\bB{\gamma})\in\mbb{M}   &&\qan&
\vec{\bB{w}}     &:= (\bB{w},\widetilde{\bB{\gamma}})\in\mbb{M}\,.
\end{alignat*}
Adding \eqref{eq:pre:flow:2} to \eqref{eq:pre:flow:3}, and then \eqref{eq:pre:flow:1} to \eqref{eq:pre:flow:4}, respectively, we obtain the following linear perturbed saddle-point problem: Given $\phi,\theta\in\L^\rho(\O)$, find $(\vec{\bB{\zeta}},\vec{\bB{u}})\in\mX\times\mbb{M}$ such that
\begin{equation} \label{syst:weak:flow}
\begin{array}{ccccll}
\mbf{a}_{\phi,\theta}
\big(\vec{\bB{\zeta}},\vec{\bB{\eta}}\big)
&+&
\mbf{b}
\big(\vec{\bB{\eta}},\vec{\bB{u}}\big)
&=&
\mbf{G}\big(\vec{\bB{\eta}}\big)
&\qquad
\forall\,\vec{\bB{\eta}}\in\mX\,,
\\[1ex]
\mbf{b}
\big(\vec{\bB{\zeta}},\vec{\bB{w}}\big)
&-&
\mbf{c}_\phi
\big(\vec{\bB{u}},\vec{\bB{w}}\big)
&=&
\mbf{F}_{\phi,\theta}
\big(\vec{\bB{w}}\big)
&\qquad
\forall\,\vec{\bB{w}}\in\mbb{M}\,,
\end{array}
\end{equation}
where for $\varphi,\vartheta\in\L^\rho(\O)$, the bilinear forms $\mbf{a}_{\varphi,\vartheta}:\mX^2\to\R$, $\mbf{b}:\mX\times\mbb{M}\to\R$ and $\mbf{c}_\varphi:\mbb{M}^2\to\R$, and functionals $\mbf{G}:\mX\to\R$ and $\mbf{F}_{\varphi,\vartheta}:\mbb{M}\to\R$ are defined by
\begin{alignat}{2}\label{def:bforms:flow}
\begin{aligned}
\hspace{-0.2cm}\mbf{a}_{\varphi,\vartheta}\big(\vec{\bB{\zeta}},\vec{\bB{\eta}}\big)
:=  \int_\O \mu(\varphi,\vartheta) \bB{\zeta}:\bB{\eta} - \int_\O \bB{\sigma}^{\rmd}:\bB{\eta} + \int_\O \bB{\zeta}:\bB{\tau}^{\rmd}\,, &\quad &&
\mbf{b}\big(\vec{\bB{\eta}},\vec{\bB{w}}\big)
:= \int_\O\big( \bB{w}\cdot\bdiv(\bB{\tau}) + \bB{\tau}:\widetilde{\bB{\gamma}}\big)\,, \\
\mbf{c}_\varphi \big(\vec{\bB{u}},\vec{\bB{w}}\big)
:= \int_\O \left( \delta_3\I + \delta_4\mbf{K}^{-1}(\varphi) \right) \bB{u}\cdot\bB{w}\,,
&\quad &&
\mbf{G}\big(\vec{\bB{\eta}}\big)
:= \left\langle \bB{\tau}\bB{\nu}, \bB{u}_{\rm D} \right\rangle_\Gamma\,,\\
\mbf{F}_{\varphi,\vartheta}\big(\vec{\bB{w}}\big)
:= -\int_\O (\bB{l}_3+\mbf{f}(\varphi,\vartheta))\cdot\bB{w}\,.
\end{aligned}
\end{alignat}
The product spaces $\mX$ and $\mbb{M}$ are endowed with the norms
\begin{align*}
\|\vec{\bB{\eta}}\|_{\mX} &:= \|\bB{\eta}\|_{0;\O} + \|\bB{\tau}\|_{\bdiv_s;\O} \quad
\forall\,\vec{\bB{\eta}}\in\mX\,, \qquad
\|\vec{\bB{w}}\|_{\mbb{M}} := \|\bB{w}\|_{0,r;\O} + \|\widetilde{\bB{\gamma}}\|_{0;\O} \quad
\forall\,\vec{\bB{w}}\in\mbb{M}\,.
\end{align*}
As in the previous sections, we can bound the bilinear forms and functionals in \eqref{def:bforms:flow} as follows
\begin{subequations}
\begin{alignat*}{2}
| \mbf{a}_{\varphi,\vartheta} \big(\vec{\bB{\zeta}},\vec{\bB{\eta}}\big) |
&\leq \|\mbf{a}\| \|\vec{\bB{\zeta}}\|_{\mX} \|\vec{\bB{\eta}}\|_{\mX}\,, &\quad
|\mbf{b} \big(\vec{\bB{\eta}},\vec{\bB{w}}\big) |
&\leq \|\mbf{b}\|\|\vec{\bB{\eta}}\|_{\mX} \|\vec{\bB{w}}\|_{\mbb{M}}\,, \quad
|\mbf{c}_\varphi \big(\vec{\bB{u}},\vec{\bB{w}}\big)|
\leq \|\mbf{c}\| \|\vec{\bB{u}}\|_{\mbb{M}} \|\vec{\bB{w}}\|_{\mbb{M}}\,, \\
\big|\mbf{G}\big(\vec{\bB{\eta}}\big)\big|
&\leq \|\mbf{G}\| \|\vec{\bB{\eta}}\|_{\mX}\,, &\quad
\big| \mbf{F}_{\varphi,\vartheta} \big(\vec{\bB{w}}\big)\big|
&\leq \|\mbf{F}_{\varphi,\vartheta}\| \|\vec{\bB{w}}\|_{\mbb{M}}\,,
\end{alignat*}
\end{subequations}
for all $\vec{\bB{\zeta}},\vec{\bB{\eta}}\in\mX$ and all $\vec{\bB{u}},\vec{\bB{w}}\in\mbb{M}$, where
\begin{align*}
\|\mbf{a}\| &:= \max\{1,\mu_{\max}\}\,, \qquad
\|\mbf{b}\| := 1\,, \qquad
\|\mbf{c}\| := \left( \delta_3+\delta_4\kappa_{\max} \right) |\O|^{(r-2)/r}\,,\\
\|\mbf{G}\| &:= \max\{1,\|\bB{i}_r\|\} \|\bB{u}_{\rm D}\|_{1/2;\Gamma}\,,
\end{align*}
with $\bB{i}_r:\bH^1(\Omega)\mapsto \bL^r(\Omega)$ being the corresponding continuous injection, and
\begin{align*}
&\|\mbf{F}_{\varphi,\vartheta}\| := \|\bB{l}_3+\mathbf f(\phi,\vartheta)\|_{0,s;\Omega}
 \leq \|\bB{l}_3\|_{0,s;\O} +  \|\mbf{f}\| \big( \|\varphi\|_{0,\rho;\O} + \|\vartheta\|_{0,\rho;\O} \big)\,.
\end{align*}

\section{Continuous solvability analysis} \label{sec:CF:WP}

To analyse the well-posedness of system \eqref{syst:main} by means of a fully coupled mixed finite element formulation, and later on by a fixed-point strategy, we study in this section the linearized versions of the weak formulations of \eqref{syst:weak:mass} and \eqref{syst:weak:heat}, and \eqref{syst:weak:flow}. The existence and uniqueness of solutions for each linearized decoupled weak formulation is carried out thanks to the Babu$\rm\check{s}$ka-Brezzi theory in Banach spaces \cite{Bernardi1988,Correa2022} and following the ideas of \cite{Alvarez2026} and \cite{Careaga2025a}.

\subsection{Fixed-point operator}

In order to determine a suitable operator connecting the solutions of the three sub-problems \eqref{syst:weak:mass}, \eqref{syst:weak:heat} and \eqref{syst:weak:flow}, we begin by introducing the three solution operators defined by each saddle-point problem. We let $\mathrm S_1,\mathrm{S}_2:\mathrm L^\rho(\O)\times \bL^r(\O)\to \mathrm L^\rho(\O)$ and $\mbf{S}:(\L^\rho(\O))^2\to \bL^r(\O)$ be the solution operators defined for all $\varphi,\vartheta\in \mathrm L^\rho(\O)$ and $\bB{w}\in \bL^r(\O)$ by
\begin{equation*}
\mathrm S_1(\varphi,\bB{w})\,:=\, \phi\,, \qquad \mathrm{S}_2(\varphi,\bB{w}):= \theta\,,\qquad \bS(\varphi,\vartheta) := \bB{u}\,,
\end{equation*}
where $\phi\in \mathrm L^\rho(\O)$, $\theta \in \mathrm L^\rho(\O)$ and $\bB{u}\in\bL^r(\O)$ are components of the respective solutions to the following linearized continuous problems:
\begin{itemize}[leftmargin=\widthof{\textbf{\quad(CP.3)} }, itemsep=-5pt]
 \item[\textbf{(CP.1)}] Find $(\vtmass,\vec{\bB{\phi}})\in \bH_1\times \bQ_1$ such that:
\begin{equation}\label{syst:FP:weak:mass}
\begin{array}{cclclcll}
a_{\varphi}\big(\vtmass,\vsmass\big) & + & b \big(\vsmass, \vec{\bB{\phi}}\big) & - &\widetilde{b}_{\bB{w}}(\vsmass, \vec{\bB{\phi}}) & = & \mathcal{F}_{\varphi} (\vsmass) &\qquad  \forall\, \vsmass\in \mathbf H_1\,,\\[1ex]
  b \big(\vtmass,\vec{\bB{\psi}}\big)  &-& c(\vec{\bB{\phi}},\vec{\bB{\psi}}) & &  &=& \mathcal{G}(\vec{\bB{\psi}})&\qquad  \forall\, \vec{\bB{\psi}}\in \mathbf Q_1\,.
\end{array}
\end{equation}
\item[\textbf{(CP.2)}] Find $(\vtheat,\theta)\in \bH_2\times \L^\rho(\O)$ such that
\begin{equation}\label{syst:FP:weak:heat}
\begin{array}{cclclcll}
\mfk{a}_{\varphi}\big(\vtheat,\vsheat\big) &+&\mfk{b} \big(\vsheat,\theta\big) &-& \mfk{\widetilde{b}}_{\bB{w}}(\vsheat,\theta) & = & \mfk{F} (\vec{\bB{q}}) &\qquad  \forall\, \vec{\bB{q}}\in \mathbf H_2\,,\\[1ex]
\mfk{b} \big(\vtheat,\varkappa\big) &-& \mfk{c}(\theta,\varkappa) & &  &=& \mfk{G}(\varkappa)&\qquad  \forall\, \varkappa\in \mathrm{L}^\rho(\O)\,.
\end{array}
\end{equation}
\item[\textbf{(CP.3)}] Find $(\vec{\bB{\zeta}},\vec{\bB{u}})\in \mX\times \mbb{M}$ such that
\begin{equation} \label{syst:FP:weak:flow}
\begin{array}{ccccll}
\mbf{a}_{\varphi,\vartheta}(\vec{\bB{\zeta}},\vec{\bB{\eta}}) &+& \mbf{b}(\vec{\bB{\eta}},\vec{\bB{u}}) & = & \mbf{G}(\vec{\bB{\eta}}) &\qquad \forall\,\vec{\bB{\eta}}\in \mX\,,\\[1ex]
\mbf{b}(\vec{\bB{\zeta}},\vec{\bB{w}}) &-& \mbf{c}_{\varphi}(\vec{\bB{u}},\vec{\bB{w}}) & = & \mbf{F}_{\varphi,\vartheta}(\vec{\bB{w}}) & \qquad \forall\, \vec{\bB{w}}\in \mbb{M}\,.
\end{array}
\end{equation}
\end{itemize}
\begin{figure}[t!]
\centering
 \includegraphics[scale=0.8]{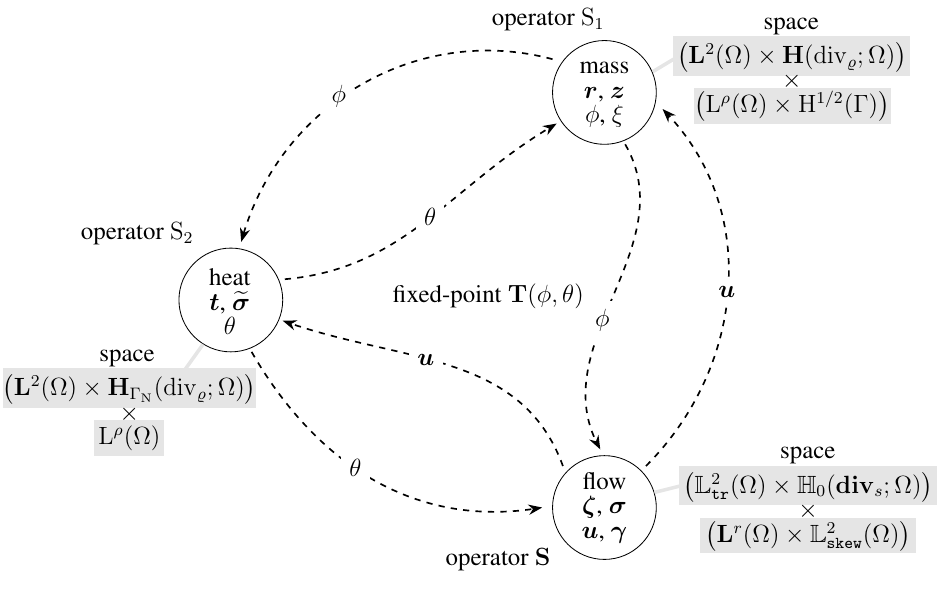}
 \caption{Diagram of the coupled system and fixed-point structure.}\label{fig:fixed:point}
\end{figure}
The fully coupled formulation composed of the continuous problems \textbf{(CP.1)}, \textbf{(CP.2)} and \textbf{(CP.3)} is established as follows: Find the three pairs $(\vec{\bB{r}},\vec{\bB{\phi}})\in \bH_{1}\times\mbf{Q}_{1}$, $(\vec{\bB{t}},\theta)\in \bH_{2}\times \mathrm{L}^\rho(\O)$, and $(\vec{\bB{\zeta}},\vec{\bB{u}})\in \mbb{X}\times \mbb{M}$ such that
\begin{equation}
\label{syst:weak:coupled}
\begin{array}{cclclcll}
a_{\phi}\big(\vtmass,\vsmass\big) & + & b\big(\vsmass,\vec{\bB\phi}\big) & - & \widetilde b_{\bB u} \big(\vsmass,\vec{\bB\phi}\big)
& = & \mathcal F_{\phi}(\vsmass) & \quad \forall\,\vsmass\in\bH_{1}\,,\\[1ex]
b\big(\vtmass,\vec{\bB\psi}\big) & - & c\big(\vec{\bB\phi},\vec{\bB\psi}\big) &&
& = & \mathcal G(\vec{\bB\psi})    & \quad \forall\,\vec{\bB\psi}\in\bQ_{1}\,,\\[1ex]
\mfk a_{\phi} \big(\vtheat,\vsheat\big) & + & \mfk b\big(\vsheat,\theta\big) & - & \mfk{\widetilde b}_{\bB u} \big(\vsheat,\theta\big)
& = & \mfk F(\vsheat) &\quad \forall\,\vsheat\in\bH_{2}\,,\\[1ex]
\mfk b\big(\vtheat,\varkappa\big) & - & \mfk c(\theta,\varkappa) &&
& = & \mfk G(\varkappa) &\quad \forall\,\varkappa\in\mathrm{L}^\rho(\O)\,, \\[1ex]
\mbf a_{\phi,\theta} \big(\vec{\bB\zeta},\vec{\bB\eta}\big) & + & \mbf b\big(\vec{\bB\eta},\vec{\bB u}\big) &&
& = & \mbf G(\vec{\bB\eta}) &\quad \forall\,\vec{\bB\eta}\in\mX\,,\\[1ex]
\mbf b\big(\vec{\bB\zeta},\vec{\bB w}\big) & - & \mbf c_{\phi} \big(\vec{\bB u},\vec{\bB w}\big) &&
& = & \mbf F_{\phi,\theta} \big(\vec{\bB w}\big) &\quad \forall\,\vec{\bB w}\in\mbb M\,.
\end{array}
\end{equation}
Finally, introducing the fixed-point operator $\mbf{T}:(\L^\rho(\O))^2\to (\L^\rho(\O))^2$ defined by
\begin{equation} \label{def:T}
\mbf{T}(\varphi,\vartheta):= \Big(\mathrm{S}_1\big(\varphi,\bS(\varphi,\vartheta)\big),\,\mathrm{S}_2\big(\mathrm{S}_1(\varphi,\bS(\varphi,\vartheta)),\bS(\varphi,\vartheta)\big)\Big)\,,
\end{equation}
we conclude that solving the coupled problem given by \eqref{syst:weak:coupled} is equivalent to finding a fixed point to the operator $\mbf{T}$, that is, finding $(\phi,\theta)\in (\L^\rho(\O))^2$ such that
\begin{equation*}
\mbf{T}(\phi,\theta) = (\phi,\theta)\,.
\end{equation*}
The fixed-point method, including the unknowns and functional spaces involved are schematised in the diagram in Figure~\ref{fig:fixed:point}.
\subsection{Well-posedness concentration equation}
\label{sec:CF:WP:concentration}

We follow the approach in \cite{Gatica2025} to establish that the operator ${\rm S}_1$ is well-defined. Specifically, we first apply \cite[Theorem 2.1]{Bernardi1988} to establish the stability of $a_\varphi$ on the kernel induced by $b$.
Then, we apply \cite[Theorem 3.2]{Gatica2023} to the restriction of the problem to the zero-mean subspace. The stability on the full concentration space is subsequently obtained by decomposing the concentration and boundary variables into a zero-mean component and a constant component controlled by the bilinear form $c$. Finally, the convective term is treated as a sufficiently small perturbation, and the Banach--Ne\v{c}as--Babu\v{s}ka theorem (cf.~\cite[Theorem 2.6]{Ern2004}) yields the well-posedness of \eqref{syst:FP:weak:mass}.

In what follows, we consider $\varphi\in\L^\rho(\O)$ and $\bB{w}\in\bL^r(\O)$ as given functions, and define the bounded global bilinear form $\bB{A}_\varphi: (\bH_1\times\bQ_1)^2 \to \R$ by
\begin{equation*}
\begin{aligned}
\bB{A}_\varphi
\big( (\vtmass,\vec{\bB{\phi}}), (\vsmass,\vec{\bB{\psi}}) \big)
:={}
a_\varphi(\vtmass,\vsmass) +
b(\vsmass,\vec{\bB{\phi}}) +
b(\vtmass,\vec{\bB{\psi}}) -
c(\vec{\bB{\phi}},\vec{\bB{\psi}})\,,
\end{aligned}
\end{equation*}
for all $(\vtmass,\vec{\bB{\phi}}),(\vsmass,\vec{\bB{\psi}})\in\bH_1\times\bQ_1$.
The boundedness of $\bB{A}_\varphi$ follows directly from the
boundedness of $a_\varphi$, $b$, and $c$
(cf.~\eqref{eq:bounded:mass}). Then, we observe that $a_\varphi$ and $c$ are positive
semi-definite. More precisely, for
$\vsmass=(\bB{s},\bB{\varsigma})\in\bH_1$ and
$\vec{\bB{\psi}}=(\psi,\chi)\in\bQ_1$, there hold
\begin{alignat*}{2}
a_\varphi(\vsmass,\vsmass)\geq d_1 \|\bB{s}\|_{0;\O}^2 \geq 0\,,\qquad
c(\vec{\bB{\psi}},\vec{\bB{\psi}})
= \delta_1 \|\psi\|_{0;\O}^2 \geq 0\,.
\end{alignat*}
Furthermore, $c$ is symmetric. Now, introducing the closed subspace
$ \bQ_{1,0} := \L^\rho_0(\O)\times\mathrm H^{1/2}(\Gamma) \subset\bQ_1$,
the kernel of $b$ with respect to its first argument, when its second argument is restricted to $\bQ_{1,0}$, is defined as
\begin{align*}
\mathcal{V} &:= \Big\{\vsmass\in\bH_1:\, b(\vsmass,\vec{\bB{\psi}})=0 \quad \forall\,\vec{\bB{\psi}}\in\bQ_{1,0} \Big\} = \bL^2(\O)\times\mathcal V_0,\\
\mathcal V_0 &:= \Big\{\bB{\varsigma}\in\bH(\div_\varrho;\O):\, \div(\bB{\varsigma})=0 \qin\O \qan \bB{\varsigma}\cdot\bB{\nu}=0 \qon\Gamma \Big\}.
\end{align*}
The next lemma provides the inf-sup conditions for the bilinear form $a_\varphi$.

\begin{lemma}
\label{lemma:infsup:a:mass}
Let $\varphi\in\L^\rho(\O)$, where $\rho$ lies within the range specified in \eqref{def:ranges}.
Then, there exists a constant $\alpha_1>0$, independent of $\varphi$, such that
\begin{alignat}{2}
\sup_{\substack{ \vsmass_{*}\in\mathcal V\setminus\{\vec{\bB{0}}\} }}
\frac{ a_\varphi(\vtmass,\vsmass_{*}) }{ \|\vsmass_{*}\|_{\bH_1} }
\geq \alpha_1 \|\vtmass\|_{\bH_1}
\qquad\text{and} \qquad
\sup_{\substack{ \vtmass_{*}\in\mathcal V\setminus\{\vec{\bB{0}}\} }}
\frac{ a_\varphi(\vtmass_{*},\vsmass) }{\|\vtmass_{*}\|_{\bH_1}}
\geq \alpha_1 \|\vsmass\|_{\bH_1},
\label{eq:IS:a:mass}
\end{alignat}
for all $\vtmass,\vsmass\in\mathcal{V}$.
\end{lemma}

\begin{proof}
We first observe that $a_\varphi$ can be regarded as the global bilinear form of a nested saddle-point problem with $a_\varphi(\vtmass,\vsmass) = \bar a_\varphi(\tmass,\bB{s})+\bar b_1(\bB{s},\vsigmass)+\bar b_2(\tmass,\bB{\varsigma})$, where $\bar a_\varphi:\bL^2(\O)\times\bL^2(\O)\to\R$, and $\bar b_1,\bar b_2: \bL^2(\O)\times\mathcal V_0 \to\R$, are bilinear forms defined by
$$
\bar a_\varphi(\tmass,\bB{s})
:=
\int_\O
\mathcal D_1(\varphi) \tmass\cdot\bB{s}\,,\quad
\bar b_1(\bB{s},\bB{\varsigma})
:=
-\int_\O
\bB{\varsigma}\cdot\bB{s}\,,
\qan
\bar b_2(\tmass,\bB{\varsigma})
:=
\int_\O
\tmass\cdot\bB{\varsigma}\,.
$$
Thanks to the lower bound of $\mathcal D_1$ (cf.~\eqref{eq:data:bounds}), $\bar a_\varphi$ is $\bL^2(\O)$-elliptic. In turn, for each $\bB{\varsigma}\in\mathcal V_0$, choosing $\bB{s}=-\bB{\varsigma}$ for $\bar b_1$ and $\tmass=\bB{\varsigma}$ for $\bar b_2$ we respectively have
\begin{align*}
\sup_{\substack{
\bB{s}\in\bL^2(\O)\setminus\{\bB{0}\}}}
\frac{ \bar b_1(\bB{s},\bB{\varsigma})}{ \|\bB{s}\|_{0;\O} }
\geq
\|\bB{\varsigma}\|_{0;\O}
=
\|\bB{\varsigma}\|_{\div_\varrho;\O}\,, \qquad
\sup_{\substack{
\tmass\in\bL^2(\O)\setminus\{\bB{0}\}
}}
\frac{ \bar b_2(\tmass,\bB{\varsigma}) }{ \|\tmass\|_{0;\O} }
\geq
\|\bB{\varsigma}\|_{0;\O}
=
\|\bB{\varsigma}\|_{\div_\varrho;\O}\,.
\end{align*}
Hence, \cite[Theorem 2.1]{Bernardi1988} implies the existence of $\alpha_1>0$, independent of $\varphi$, such that \eqref{eq:IS:a:mass} hold.
\end{proof}

We next establish the inf-sup condition for $b$ on the zero-mean
subspace $\bQ_{1,0}$.

\begin{lemma} \label{lemma:infsup:b:mass}
There exists a constant $\beta_1>0$, depending only on $\O$, $\rho$, and the corresponding trace and Neumann stability constants, such that
\begin{align}
\sup_{\substack{
\vsmass\in\bH_1\setminus\{\vec{\bB{0}}\}
}}
\frac{ b(\vsmass,\vec{\bB{\psi}}) }{ \|\vsmass\|_{\bH_1} }
\geq
\beta_1
\|\vec{\bB{\psi}}\|_{\bQ_1}
\qquad
\forall\,\vec{\bB{\psi}}\in\bQ_{1,0}\,.
\label{eq:IS:b:mass}
\end{align}
\end{lemma}
\begin{proof}
The proof can be found in Appendix~\ref{appendix:proof:infsup:b:mass}.
\end{proof}

To show the global inf-sup condition for $\bB{A}_\varphi$, we proceed as follows. We first define the constant vector $\vec{\bB{e}} = (1,-1)\in\bQ_1$ and consider the decomposition $\bQ_1 = \bQ_{1,0}\oplus \text{span}\{\vec{\bB{e}}\}$, which implies that for all $\vec{\bB{\phi}},\vec{\bB{\psi}}\in\bQ_1$ there exist unique $\vec{\bB{\phi}}_0=(\phi_0,\xi),\vec{\bB{\psi}}_0=(\psi_0,\chi)\in \bQ_{1,0}$ such that
\begin{equation*}
\vec{\bB{\phi}} = \vec{\bB{\phi}}_0 + m_\phi\vec{\bB{e}}\,,\qquad \vec{\bB{\psi}} = \vec{\bB{\psi}}_0 + n_\psi\vec{\bB{e}}\qquad \text{with}\qquad
m_\phi :=
\frac{1}{|\O|} \int_\O \phi\,,\qquad n_\psi := \frac{1}{|\O|} \int_\O \psi\,.
\end{equation*}
Furthermore, directly from the definition of the bilinear forms $b$ and $c$ in \eqref{def:bforms:mass}, and the fact that $\phi_0$ and $\psi_0$ have zero mean, we have
\begin{align*}
b\big(\vsmass,\vec{\bB{e}}\big) &= 0\,,\qquad
c\big( \vec{\bB{\phi}}_0, \vec{\bB{e}} \big) = c\big( \vec{\bB{e}}, \vec{\bB{\psi}}_0 \big) = 0\,,
\qquad
c\big( \vec{\bB{e}}, \vec{\bB{e}} \big) = \delta_1 |\O|\,,
\end{align*}
for all $\vec{\bB{s}}\in \bH_1$. Combining these identities, we get
\begin{align}
&\bB{A}_\varphi\big((\vtmass,\vec{\bB{\phi}}), (\vsmass,\vec{\bB{\psi}}) \big)
=
\bB{A}_{\varphi,0} \big( (\vtmass,\vec{\bB{\phi}}_0), (\vsmass,\vec{\bB{\psi}}_0) \big) - \delta_1 |\O| m_\phi n_\psi\,, \label{eq:A:decomp:mass}
\end{align}
where bilinear form $\bB{A}_{\varphi,0}$ is defined as the restriction of $\bB{A}_{\varphi}$ to the subspace $(\bH_1\times\bQ_{1,0})^2$, which controls the zero-mean components, whereas $c$ controls the remaining constant direction. The inf-sup conditions established in Lemmas~\ref{lemma:infsup:a:mass} and \ref{lemma:infsup:b:mass}, together with the positive semi-definiteness of $a_\varphi$ and $c$, allow us to apply \cite[Theorem 3.2]{Gatica2023} to $\bB{A}_{\varphi,0}$. Hence, $\bB{A}_{\varphi,0}$ satisfies inf-sup conditions in both arguments with constant $\alpha_{\bB{A},0}>0$, depending only on $\alpha_1$, $\beta_1$, $\|a\|$, and $\|c\|$. Then, for a fixed pair $(\vtmass,\vec{\bB{\phi}}_0)\in\bH_1\times \bQ_{1,0}$ with $(\vtmass,\vec{\bB{\phi}}_0)\neq (\vec{\bB{0}},\vec{\bB{0}})$, thanks to the mentioned inf-sup condition of $\bB{A}_{\varphi,0}$, there exists $(\vsmass,\vec{\bB{\psi}}_0)\in \bH_1\times\bQ_{1,0}$ such that $\|(\vsmass,\vec{\bB{\psi}}_0)\| =
\|(\vtmass,\vec{\bB{\phi}}_0)\|$ and
\begin{align}
\dfrac{\bB{A}_{\varphi,0} \big( (\vtmass,\vec{\bB{\phi}}_0), (\vsmass,\vec{\bB{\psi}}_0) \big)}{ \|(\vsmass,\vec{\bB{\psi}}_0)\|}
\,=\,
\dfrac{\bB{A}_{\varphi,0} \big( (\vtmass,\vec{\bB{\phi}}_0), (\vsmass,\vec{\bB{\psi}}_0) \big)}{ \|(\vtmass,\vec{\bB{\phi}}_0)\|}
&\,\geq\,
\frac{\alpha_{\bB{A},0}}{2} \|(\vtmass,\vec{\bB{\phi}}_0)\|\,. \label{eq:restricted:test:estimate:mass}
\end{align}
Here, the factor $1/2$ accounts for the possibility that the supremum over the second argument in the left-most term in \eqref{eq:restricted:test:estimate:mass} is not attained. Note that for the zero case  $(\vtmass,\vec{\bB{\phi}}_0) = (\vec{\bB{0}},\vec{\bB{0}})$, we set $(\vsmass,\vec{\bB{\psi}}_0) = (\vec{\bB{0}},\vec{\bB{0}})$, so that \eqref{eq:restricted:test:estimate:mass} remain valid. Now, choosing the extended test function $\vec{\bB{\psi}}_{\rm ext} := \vec{\bB{\psi}}_0 - m_\phi\vec{\bB{e}}\in\bQ_1$ and using  \eqref{eq:restricted:test:estimate:mass}, we have
\begin{align}
\begin{aligned}
\hspace{-0.20cm}
\bB{A}_\varphi
\big( (\vtmass,\vec{\bB{\phi}}), (\vsmass,\vec{\bB{\psi}}_{\rm ext}) \big)
 = \bB{A}_{\varphi,0} \big( (\vtmass,\vec{\bB{\phi}}_0), (\vsmass,\vec{\bB{\psi}}_0) \big)
+ \delta_1 |\O| m_\phi^2   \geq
\frac{\alpha_{\bB{A},0}}{2} \|(\vtmass,\vec{\bB{\phi}}_0)\|^{2}
+
\delta_1 |\O| m_\phi^2\geq \widehat{\alpha}_{\bB{A}} R_\phi^2 , \label{eq:global:test:lower:mass}
\end{aligned}
\end{align}
where $\widehat{\alpha}_{\bB{A}} := \min \big\{ \tfrac{1}{2}\alpha_{\bB{A},0}, \delta_1 |\O| \big\}$ and
$R_\phi^2 := \|(\vtmass,\vec{\bB{\phi}}_0)\|^{2} + |m_\phi|^2$.
Combining \eqref{eq:global:test:lower:mass} with the boundedness of $\bB{A}_{\varphi}$, we can show that $(\vsmass,\vec{\bB{\psi}}_{\rm ext})$ is nonzero. On the other hand, directly from the definition of $\vec{\bB{\psi}}_{\rm ext}$ and linear decomposition of $\vec{\bB{\phi}}$, we obtain
\begin{align}
\|(\vsmass,\vec{\bB{\psi}}_{\rm ext})\|
& \,\leq\, \|(\vsmass,\vec{\bB{\psi}}_0)\| + |m_\phi| \|(\vec{\bB{0}},\vec{\bB{e}})\|
\,\leq\, C_{\rm ext} R_\phi\,, \label{eq:global:test:norm:mass}\\
\|(\vtmass,\vec{\bB{\phi}})\| & =  \|(\vtmass,\vec{\bB{\phi}}_0) + m_\phi (\vec{\bB{0}},\vec{\bB{e}})\|  \leq C_{\rm ext} R_\phi\,, \label{eq:global:trial:norm:mass}
\end{align}
where $C_{\rm ext}>0$ depends only on the norm of $\vec{\bB{e}}$.
Thus, combining \eqref{eq:global:test:lower:mass} and \eqref{eq:global:trial:norm:mass}, and since $(\vsmass,\vec{\bB{\psi}}_{\rm ext})$ is an admissible nonzero test function, we conclude the inf-sup condition of $\bB{A}_\varphi$ on the first argument
\begin{align}
\sup_{\substack{ (\vsmass,\vec{\bB{\psi}}) \in\bH_1\times\bQ_1\setminus\{(\vec{\bB{0}},\vec{\bB{0}})\}}}
\frac{ \bB{A}_\varphi \big( (\vtmass,\vec{\bB{\phi}}), (\vsmass,\vec{\bB{\psi}}) \big) }{
\|(\vsmass,\vec{\bB{\psi}})\| }
\geq
\frac{ \bB{A}_\varphi \big( (\vtmass,\vec{\bB{\phi}}), (\vsmass,\vec{\bB{\psi}}_{\rm ext}) \big) }{
\|(\vsmass,\vec{\bB{\psi}}_{\rm ext})\|}
&\geq
\frac{\widehat{\alpha}_{\bB{A}}}{C_{\rm ext}} R_\phi
\geq \frac{\widehat{\alpha}_{\bB{A}}}{C_{\rm ext}^2} \|(\vtmass,\vec{\bB{\phi}})\|\,.
\label{eq:GIS:mass:aux:1}
\end{align}
for all $(\vtmass,\vec{\bB{\phi}})\in\bH_1\times\bQ_1$, where $\alpha_{\bB{A}} := \widehat{\alpha}_{\bB{A}}/C_{\rm ext}^2 > 0$.
The inf-sup condition with respect to the second argument follows analogously to the derivation of \eqref{eq:GIS:mass:aux:1}, with $(\vsmass,\vec{\bB{\psi}})$ playing the role of $(\vtmass,\vec{\bB{\phi}})$.

Having the inf-sup conditions for $\bB{A}_{\varphi}$, we can now show the inf-sup condition for the global bilinear form $\bB{A}_{\varphi,\bB{w}}$ related to \eqref{syst:weak:mass}, that incorporates the remaining bilinear form $\widetilde{b}_{\bB{w}}$ for $\bB{w}\in\bL^r(\Omega)$, which is defined by $\bB{A}_{\varphi,\bB{w}}: (\bH_1\times\bQ_1)^2 \to\R$:
\begin{align}
&\bB{A}_{\varphi,\bB{w}} \big( (\vtmass,\vec{\bB{\phi}}), (\vsmass,\vec{\bB{\psi}}) \big)
:=
\bB{A}_\varphi \big( (\vtmass,\vec{\bB{\phi}}), (\vsmass,\vec{\bB{\psi}}) \big)
-
\widetilde b_{\bB{w}} \big( \vsmass,\vec{\bB{\phi}} \big)\,.\label{def:A:mass}
\end{align}
Then, using the boundedness of $\widetilde{b}_{\bB{w}}$ \eqref{eq:bounded:mass}, the inf-sup condition \eqref{eq:GIS:mass:aux:1} and following the same approach used, for instance, in \cite{Careaga2025a}, for $\|\bB{w}\|_{0,r;\O} \leq \alpha_{\bB{A}}/2$, we obtain the inf-sup condition for $\bB{A}_{\varphi,\bB{w}}$:
\begin{align}
&\sup_{\substack{(\vsmass,\vec{\bB{\psi}})\in\bH_1\times\bQ_1\setminus\{(\vec{\bB{0}},\vec{\bB{0}})\}}}
\frac{
\bB{A}_{\varphi,\bB{w}}\big( (\vtmass,\vec{\bB{\phi}}), (\vsmass,\vec{\bB{\psi}}) \big)}{
\|(\vsmass,\vec{\bB{\psi}})\|}
\geq
\Big(\alpha_{\bB{A}} - \|\bB{w}\|_{0,r;\O} \Big)
\|(\vtmass,\vec{\bB{\phi}})\|\geq
\frac{\alpha_{\bB{A}}}{2}\|(\vtmass,\vec{\bB{\phi}})\|,\label{eq:GIS:mass:1}
\end{align}
The inf-sup condition in the second argument can be obtained following the same reasoning. Thus, we are now in position to establish the well-posedness of the operator $\mathrm{S}_1$, equivalently the existence and uniqueness of solution to \eqref{syst:FP:weak:mass}.

\begin{lemma}
\label{lemma:solv:mass}
Let $r,\rho$ and their conjugate exponents $s,\varrho$ satisfy the conditions in \eqref{def:ranges}, together with $1/\rho+1/r=1/2$. Then, for each $(\varphi,\bB{w})\in\L^\rho(\O)\times\bL^r(\O)$ satisfying $\|\bB{w}\|_{0,r;\O} \leq \alpha_{\bB{A}}/2$, there exists a unique $(\vtmass,\vec{\bB{\phi}})\in\bH_1\times\bQ_1$ solving \eqref{syst:FP:weak:mass}. Consequently, the operator
$\mathrm S_1$ is well-defined by $\mathrm S_1(\varphi,\bB{w}):=\phi$. Moreover, there exists a constant $C_{\mathrm S_1}>0$, independent of $(\varphi,\bB{w})$, such that
\begin{align}
\|\mathrm S_1(\varphi,\bB{w})\|_{0,\rho;\O}
= \|\phi\|_{0,\rho;\O} &\leq \|\vtmass\|_{\bH_1} + \|\vec{\bB{\phi}}\|_{\bQ_1}
\leq C_{\mathrm S_1} \datacons\,. \label{eq:apriori:mass}
\end{align}
where $\datacons:=f_{\max} + \|l_1\|_{0;\O} + \|\Phi\|_{-1/2,\Gamma}$.
\end{lemma}
\begin{proof}
The inf-sup conditions for $\bB{A}_{\varphi,\bB{w}}$ deduced above, together with the boundedness of $\bB{A}_{\varphi,\bB{w}}$, allow us to apply the
Banach--Ne\v{c}as--Babu\v{s}ka theorem (cf.~\cite[Theorem 2.6]{Ern2004}). Hence, \eqref{syst:FP:weak:mass} admits a unique solution and
\begin{align*}
\|\vtmass\|_{\bH_1} + \|\vec{\bB{\phi}}\|_{\bQ_1}
&\leq \frac{2}{\alpha_{\bB{A}}}
\left( \|\mathcal F_\varphi\|_{\bH_1'} + \|\mathcal G\|_{\bQ_1'} \right).
\end{align*}
Using the bounds for $\mathcal F_\varphi$ and $\mathcal G$ given in \eqref{eq:bounded:mass}, and absorbing the powers of $|\O|$ into a
constant, we obtain \eqref{eq:apriori:mass}. The constant $C_{\mathrm S_1}$ depends only on $\O$, $\rho$, and
$\alpha_{\bB{A}}$, and also depends on the constants entering in the bounds of the
bilinear forms, but is independent of $(\varphi,\bB{w})$.
\end{proof}

\subsection{Well-posedness heat equation} \label{sec:CF:WP:heat}

To analyse the linearized weak formulation of the heat equation \eqref{syst:FP:weak:heat}, we follow the same reasoning used in the well-posedness analysis of \eqref{syst:FP:weak:mass} from the previous
section. We begin by assuming that $\varphi\in\mathrm{L}^\rho(\O)$ and $\bB{w}\in\bL^r(\O)$ are given functions and introduce the global bilinear form associated with the unperturbed saddle-point problem, $\mfk{A}_{\varphi}: (\bH_2\times\L^\rho(\O))^2\to\R$, defined by
\begin{equation*}
\begin{aligned}
\mfk{A}_{\varphi}
\big((\vtheat,\theta),(\vsheat,\varkappa)\big)
\,:=\,
\mfk{a}_{\varphi}(\vtheat,\vsheat) + \mfk{b}(\vsheat,\theta) + \mfk{b}(\vtheat,\varkappa) - \mfk{c}(\theta,\varkappa)\,,
\end{aligned}
\end{equation*}
for all $(\vtheat,\theta),(\vsheat,\varkappa)\in \bH_2\times\L^\rho(\O)$. From the definitions of $\mfk{a}_{\varphi}$ and $\mfk{c}$, together with the uniform lower bound of $\mathcal{D}_2$, we readily obtain
\begin{align*}
\begin{aligned}
\mfk{a}_{\varphi}(\vsheat,\vsheat) \,\geq\, d_2\|\bB{q}\|^2_{0;\O} \,\geq\,0\,, \qquad
\mfk{c}(\varkappa,\varkappa) \,=\, \delta_2\|\varkappa\|_{0;\O}^2 \,\geq\,0\,,
\end{aligned}
\end{align*}
for all $\vsheat\in\bH_2$ and $\varkappa\in\L^\rho(\O)$. In addition, $\mfk{c}$ is symmetric. Furthermore, the kernel induced by the bilinear form $\mfk{b}$ is given by $\widetilde{\mathcal{V}} :=\bL^2(\O)\times\widetilde{\mathcal{V}}_0$, where
\begin{align*}
\widetilde{\mathcal{V}}_0 \,:=\,
\Big\{ \widetilde{\bB{\tau}} \in\bH_{\Gamma_{\rm N}}(\div_\varrho;\O): \quad \div(\widetilde{\bB{\tau}}) = 0 \qin\O \Big\}\,.
\end{align*}
The next lemma provides the inf-sup conditions of
$\mfk{a}_{\varphi}$ on its two arguments.

\begin{lemma}\label{lemma:infsup:a:heat}
Let $\varphi\in\L^\rho(\O)$, with $\rho$ within the range of values stipulated in \eqref{def:ranges}. Then, there exists a constant $\alpha_2>0$, independent of $\varphi$, such that
\begin{alignat*}{2}
\sup_{\substack{ \vsheat_{*}\in\widetilde{\mathcal{V}} \setminus\{\vec{\bB{0}}\} }}
\frac{ \mfk{a}_{\varphi}(\vtheat,\vsheat_{*}) }{ \|\vsheat_{*}\|_{\bH_2} }
\,\geq\, \alpha_2\|\vtheat\|_{\bH_2}
\qquad \text{and}\qquad
\sup_{\substack{ \vtheat_{*}\in\widetilde{\mathcal{V}} \setminus\{\vec{\bB{0}}\} }}
\frac{ \mfk{a}_{\varphi}(\vtheat_{*},\vsheat) }{ \|\vtheat_{*}\|_{\bH_2} }
\,\geq\,
\alpha_2\|\vsheat\|_{\bH_2}\,,
\end{alignat*}
for all $ \vtheat,\vsheat\in\widetilde{\mathcal{V}}$.
\end{lemma}
\begin{proof}
The proof follows the same arguments as those employed in Lemma~\ref{lemma:infsup:a:mass}. Indeed, for every $\widetilde{\bB{\tau}}\in\widetilde{\mathcal{V}}_0$, there holds $\|\widetilde{\bB{\tau}}\|_{\div_\varrho;\O} = \|\widetilde{\bB{\tau}}\|_{0;\O}$, and hence the inf-sup conditions for the two coupling terms are obtained by choosing the corresponding functions in $\bL^2(\O)$ equal to $\widetilde{\bB{\tau}}$ and $-\widetilde{\bB{\tau}}$, respectively. The result then follows from the uniform ellipticity of the bilinear form induced by $\mathcal{D}_2(\varphi)$ and the same nested saddle-point argument used in the proof of Lemma~\ref{lemma:infsup:a:mass}.
\end{proof}

Next, in order to establish the required hypotheses of \cite[Theorem 3.2]{Gatica2023}, we prove the continuous inf-sup condition for $\mfk{b}$.

\begin{lemma}\label{lemma:infsup:b:heat}
There exists $\beta_2>0$, depending only on $\Omega$, $\Gamma_{\rm D}$, and $\rho$, such that
\begin{equation*}
\sup_{\substack{ \vsheat\in\bH_2\setminus\{\vec{\bB{0}}\} }}
\frac{ \mfk{b}(\vsheat,\varkappa) }{ \|\vsheat\|_{\bH_2} }
\,\geq\, \beta_2\|\varkappa\|_{0,\rho;\O} \qquad \forall\,\varkappa\in\L^\rho(\O)\,.
\end{equation*}
\end{lemma}
\begin{proof}
The proof can be found in Appendix~\ref{appendix:proof:infsup:b:heat}.
\end{proof}

The positive semidefiniteness of $\mfk{a}_{\varphi}$ and $\mfk{c}$, the symmetry of $\mfk{c}$, and Lemmas~\ref{lemma:infsup:a:heat} and \ref{lemma:infsup:b:heat} allow us to apply \cite[Theorem 3.2]{Gatica2023}. Consequently, there exists $\alpha_{\mfk{A}}>0$, independent of $\varphi$, such that $\mfk{A}_{\varphi}$ satisfies two global inf-sup conditions (one for each argument) on $\bH_2\times\L^\rho(\O)$. The constant $\alpha_{\mfk{A}}$ depends only on $\alpha_2$, $\beta_2$, and the continuity constants of the bilinear forms involved. We next define the perturbed global bilinear form by
\begin{equation}\label{def:A:heat}
\begin{aligned}
\mfk{A}_{\varphi,\bB{w}}
\big((\vtheat,\theta),(\vsheat,\varkappa)\big)
\,:=\,
\mfk{A}_{\varphi}\big((\vtheat,\theta),(\vsheat,\varkappa)\big) - \mfk{\widetilde{b}}_{\bB{w}}(\vsheat,\theta).
\end{aligned}
\end{equation}
Similarly to the derivation of the inf-sup conditions for $\bB{A}_{\varphi,\bB{w}}$ in Section~\ref{sec:CF:WP:concentration}, using the inf-sup conditions for $\mfk{A}_{\varphi}$ and boundedness of $\mfk{\widetilde{b}}_{\bB{w}}$ in \eqref{eq:bounded:heat} with $\|\bB{w}\|_{0,r;\O} \leq \alpha_{\mfk{A}}/2$, we obtain the global inf-sup conditions for $\mfk{A}_{\varphi,\bB{w}}$.
The well-definedness of the operator $\mathrm{S}_2$ and the existence and uniqueness of solutions of \eqref{syst:FP:weak:heat} are provided in the following lemma.

\begin{lemma}\label{lemma:solv:heat}
Let $(\rho,\varrho)$ be a pair of conjugate exponents within the range of values in \eqref{def:ranges}. Then, for each $(\varphi,\bB{w})\in\mathrm{L}^{\rho}(\O)\times\bL^r(\O)$ satisfying $\|\bB{w}\|_{0,r;\O}\leq\alpha_{\mfk{A}}/2$, there exists a unique $(\vtheat,\theta)\in\bH_2\times\mathrm{L}^{\rho}(\O)$ solution to \eqref{syst:FP:weak:heat}. Hence, the operator $\mathrm{S}_2$ is well defined by $\mathrm{S}_2(\varphi,\bB{w})=\theta$. Moreover, there exists a positive constant $C_{\mathrm{S}_2}$, independent of $(\varphi,\bB{w})$, such that
\begin{align}\label{eq:apriori:heat}
\|\mathrm{S}_2(\varphi,\bB{w})\|
\,=\, \|\theta\|_{0,\rho;\O}
\,\leq\, \|\vtheat\|_{\bH_2} +  \|\theta\|_{0,\rho;\O} \,\leq\,
C_{\mathrm{S}_2}\dataheat\,.
\end{align}
where $\dataheat:= \|l_2\|_{0;\O} + \|\theta_{\rm D}\|_{1/2,\Gamma_{\rm D}}$.
\end{lemma}
\begin{proof}
The boundedness of $\mfk{A}_{\varphi,\bB{w}}$, together with the global inf-sup conditions related to the first and second argument of $\mfk{A}_{\varphi,\bB{w}}$ allow us to apply \cite[Theorem 2.6]{Ern2004}. Hence,
\eqref{syst:FP:weak:heat} admits a unique solution $(\vtheat,\theta)\in\bH_2\times\mathrm{L}^{\rho}(\O)$.
Furthermore, the estimate provided by \cite[Theorem 2.6, eqn.~(2.5)]{Ern2004}, together with the bounds for
$\mfk{F}$ and $\mfk{G}$, yields \eqref{eq:apriori:heat}. The constant $C_{\mathrm{S}_2}$ depends only on $\alpha_{\mfk{A}}$ and the constants entering the continuity bounds for $\mfk{F}$ and $\mfk{G}$; in particular, it is independent of $(\varphi,\bB{w})$.
\end{proof}

\subsection{Well-posedness flow equations} \label{sec:CF:WP:flow}

We now establish the well-posedness of the linearized weak formulation of the flow equations \eqref{syst:FP:weak:flow}, which is written as a perturbed saddle point problem with bilinear forms $\mbf{a}_{\varphi,\vartheta}\,$, $\mbf{b}$ and $\mbf{c}_{\varphi}$ (cf.~\eqref{def:bforms:flow}) for given functions $\varphi,\vartheta\in \L^\rho(\O)$. The analysis follows similarly as the one performed for the concentration and heat problems. As in the previous sections, we define the global bilinear form of \eqref{syst:FP:weak:flow} as
\begin{align}\label{def:A:flow}
 \mbb{A}_{\varphi,\vartheta} \big((\vec{\bB{\zeta}},\vec{\bB{u}}),(\vec{\bB{\eta}},\vec{\bB{w}})\big):=
\mbf{a}_{\varphi,\vartheta}(\vec{\bB{\zeta}},\vec{\bB{\eta}}) + \mbf{b}(\vec{\bB{\eta}},\vec{\bB{u}}) +
\mbf{b}(\vec{\bB{\zeta}},\vec{\bB{w}}) - \mbf{c}_\varphi(\vec{\bB{u}},\vec{\bB{w}})\,,
\end{align}
for all $(\vec{\bB{\zeta}},\vec{\bB{u}}), (\vec{\bB{\eta}},\vec{\bB{w}})\in \mX\times \mbb{M}$.
Furthermore, we observe that both $\mbf{a}_{\varphi,\vartheta}$ and $\mbf{c}_{\varphi}$ are positive semidefinite, that is
\begin{alignat*}{2}
\mbf{a}_{\varphi,\vartheta} (\vec{\bB{\zeta}},\vec{\bB{\zeta}}) &\,\geq\, \mu_{\min}\|\bB{\zeta}\|_{0;\O}^2 \,\geq\,0\,,\qquad
\mbf{c}_{\varphi}(\vec{\bB{w}},\vec{\bB{w}}) &\,\geq\, (\delta_3+\delta_4\kappa_{\min}) \|\bB{w}\|_{0;\O}^2 \,\geq\,0\,,
\end{alignat*}
for all $\vec{\bB{\zeta}}\in \mX$ and $\vec{\bB{w}}\in\mbb{M}$, and $\mbf{c}_{\varphi}$ is symmetric. The kernel of $\mbf{b}$ is given by
$\mathbcal{V}=\Ltwotr\times\mathbcal{V}_0$, where
\begin{align*}
\mathbcal{V}_0 := \Big\{ \bB{\tau}\in\Hdivtr: \quad \bdiv(\bB{\tau})=0 \qin\O \qan \bB{\tau}=\bB{\tau}^{\tt t} \Big\}\,.
\end{align*}
The following lemma states the inf-sup conditions for the bilinear form $\mbf{a}_{\varphi,\vartheta}$.
\begin{lemma}\label{lemma:infsup:a:flow}
Let $\varphi,\vartheta\in\L^\rho(\O)$, with $\rho$ within the range of values stipulated in \eqref{def:ranges}. Then, there exists a constant $\alpha_3>0$, independent of $\varphi$ and $\vartheta$, such that
\begin{align}
\sup_{\substack{\vec{\bB{\zeta}}_{*}\in\mathbcal{V}
\setminus\{\vec{\bB{0}}\}}} \frac{ \mbf{a}_{\varphi,\vartheta}(\vec{\bB{\zeta}}_{*},\vec{\bB{\eta}})}{
\|\vec{\bB{\zeta}}_{*}\|_{\mX}}
\geq
\alpha_3\|\vec{\bB{\eta}}\|_{\mX}
\qquad\text{and}\qquad
\sup_{\substack{ \vec{\bB{\eta}}_{*}\in\mathbcal{V}
\setminus\{\vec{\bB{0}}\} }}
\frac{ \mbf{a}_{\varphi,\vartheta}
(\vec{\bB{\zeta}},\vec{\bB{\eta}}_{*})}{ \|\vec{\bB{\eta}}_{*}\|_{\mX}}
\geq
\alpha_3\|\vec{\bB{\zeta}}\|_{\mX}\,, \label{eq:IS:a:flow}
\end{align}
for all $\vec{\bB{\eta}},\vec{\bB{\zeta}}\in\mathbcal{V}$.
\end{lemma}
\begin{proof}
We first observe that
$ \mbf{a}_{\varphi,\vartheta} (\vec{\bB{\zeta}},\vec{\bB{\eta}}) = \mbf{\bar{a}}_{\varphi,\vartheta} (\bB{\zeta},\bB{\eta})
+ \mbf{\bar{b}}_1(\bB{\zeta},\bB{\tau})
+ \mbf{\bar{b}}_2(\bB{\eta},\bB{\sigma})$,
where
\begin{align*}
\mbf{\bar{a}}_{\varphi,\vartheta} (\bB{\zeta},\bB{\eta}) := \int_\O \mu(\varphi,\vartheta) \bB{\zeta}:\bB{\eta}\,,\qquad
\mbf{\bar{b}}_1(\bB{\zeta},\bB{\tau})  := \int_\O \bB{\tau}^{\rmd}:\bB{\zeta}\,, \qquad
\mbf{\bar{b}}_2(\bB{\eta},\bB{\sigma}) := -\int_\O \bB{\sigma}^{\rmd}:\bB{\eta}\,.
\end{align*}
The uniform lower bound for $\mu$ implies that $\mbf{\bar{a}}_{\varphi,\vartheta}$ is elliptic on $\Ltwotr$, independently of $\varphi$ and $\vartheta$. On the other hand, the following inequality ensures the existence of $C_{\rmd}>0$, depending only on $\Omega$, such that
$$
\|\bB{\tau}\|_{0;\O} \leq C_{\rmd} \big( \|\bB{\tau}^{\rmd}\|_{0;\O} + \|\bdiv(\bB{\tau})\|_{0,s;\O} \big)
\qquad \forall\,\bB{\tau}\in\Hdivtr\,.
$$
Consequently, for every $\bB{\tau}\in\mathbcal{V}_0$, there holds $\|\bB{\tau}\|_{\bdiv_s;\O} = \|\bB{\tau}\|_{0;\O} \leq C_{\rmd}\|\bB{\tau}^{\rmd}\|_{0;\O}$. Then, given $\bB{\tau},\bB{\sigma}\in \mathbcal{V}_0$, where $\bB{\tau}^{\rmd},\bB{\sigma}^{\rmd}\in\Ltwotr$, and respectively choosing $\bB{\zeta}=\bB{\tau}^{\rmd}$ for $\mbf{\bar{b}}_1$ and $\bB{\eta}=-\bB{\sigma}^{\rmd}$ for $\mbf{\bar{b}}_2$, we have
\begin{align*}
\sup_{\substack{\bB{\zeta}\in\Ltwotr\setminus\{\bB{0}\}}}
\frac{\mbf{\bar{b}}_1(\bB{\zeta},\bB{\tau})}{\|\bB{\zeta}\|_{0;\O}}
\geq \|\bB{\tau}^{\rmd}\|_{0;\O}
\geq C_{\rmd}^{-1} \|\bB{\tau}\|_{\bdiv_s;\O}\,,
\qquad
\sup_{\substack{\bB{\eta}\in\Ltwotr\setminus\{\bB{0}\}}}\frac{\mbf{\bar{b}}_2(\bB{\eta},\bB{\sigma})}{\|\bB{\eta}\|_{0;\O}}
&\geq C_{\rmd}^{-1} \|\bB{\sigma}\|_{\bdiv_s;\O}\,.
\end{align*}
The uniform ellipticity of $\mbf{\bar{a}}_{\varphi,\vartheta}$ and the preceding inf-sup conditions allow us to apply the same nested saddle-point argument used in Lemma~\ref{lemma:infsup:a:mass}, which proves \eqref{eq:IS:a:flow}.
\end{proof}
The inf-sup condition for $\mbf{b}$ is stated next.
\begin{lemma}\label{lemma:infsup:b:flow}
There exists $\beta_3>0$, depending only on $\Omega$ and $r$, such
that
\begin{equation}\label{eq:IS:b:flow}
\sup_{\substack{
\vec{\bB{\eta}}\in\mX\setminus\{\vec{\bB{0}}\}
}}
\frac{
\mbf{b}(\vec{\bB{\eta}},\vec{\bB{w}})
}{
\|\vec{\bB{\eta}}\|_{\mX}
}
\geq
\beta_3\|\vec{\bB{w}}\|_{\mbb{M}}
\qquad
\forall\,\vec{\bB{w}}\in\mbb{M}\,.
\end{equation}
\end{lemma}
\begin{proof}
 The particular case $r=4$ and $s=4/3$ is established in \cite[Lemma 3.5]{Gatica2021}. The proof corresponds to a generalized version to general $r$ and $s$, and can be found in Appendix~\ref{appendix:proof:infsup:b:flow}.
\end{proof}

The positive semidefiniteness of $\mbf{a}_{\varphi,\vartheta}$ and $\mbf{c}_{\varphi}$, the symmetry of
$\mbf{c}_{\varphi}$, and the inf-sup conditions established in Lemmas~\ref{lemma:infsup:a:flow} and
\ref{lemma:infsup:b:flow} allow us to apply \cite[Theorem 3.2]{Gatica2023}. Hence, there exists
$\alpha_{\mbb{A}}>0$, independent of $\varphi$ and $\vartheta$, such that
\begin{alignat}{2}
\sup_{\substack{ (\vec{\bB{\eta}},\vec{\bB{w}}) \in\mX\times\mbb{M}\setminus\{(\vec{\bB{0}},\vec{\bB{0}})\}}}
\frac{ \mbb{A}_{\varphi,\vartheta} \big((\vec{\bB{\zeta}},\vec{\bB{u}}),(\vec{\bB{\eta}},\vec{\bB{w}})\big)}{
\|(\vec{\bB{\eta}},\vec{\bB{w}})\|}
&\,\geq\,
\alpha_{\mbb{A}}\|(\vec{\bB{\zeta}},\vec{\bB{u}})\|\qquad && \forall (\vec{\bB{\zeta}},\vec{\bB{u}})\in\mX\times\mbb{M}\,,
\label{eq:GIS:flow:1}\\
\sup_{\substack{ (\vec{\bB{\zeta}},\vec{\bB{u}})\in\mX\times\mbb{M}\setminus\{(\vec{\bB{0}},\vec{\bB{0}})\}}}
\frac{\mbb{A}_{\varphi,\vartheta}\big((\vec{\bB{\zeta}},\vec{\bB{u}}),(\vec{\bB{\eta}},\vec{\bB{w}})\big)}{\|(\vec{\bB{\zeta}},\vec{\bB{u}})\|}
&\,\geq\,
\alpha_{\mbb{A}} \|(\vec{\bB{\eta}},\vec{\bB{w}})\| \qquad && \forall (\vec{\bB{\eta}},\vec{\bB{w}})\in\mX\times\mbb{M}\,.
\label{eq:GIS:flow:2}
\end{alignat}
The solvability of \eqref{syst:FP:weak:flow} is stated as follows.

\begin{lemma}\label{lemma:solv:flow}
Let $(r,s)$ and $(\rho,\varrho)$ be pairs of conjugate exponents satisfying \eqref{def:ranges}. Then, for each
$\varphi,\vartheta\in\mathrm{L}^\rho(\O)$, there exists a unique $(\vec{\bB{\zeta}},\vec{\bB{u}})
\in\mX\times\mbb{M}$ solution to \eqref{syst:FP:weak:flow}. Hence, the operator $\bS$ is well defined
by $\bS(\varphi,\vartheta):=\bB{u}$. Moreover, there exists a positive constant $C_{\bS}$, independent
of $\varphi$ and $\vartheta$, such that
\begin{align}\label{eq:apriori:flow}
\|\bS(\varphi,\vartheta)\|_{0,r;\O}
\,=\, \|\bB{u}\|_{0,r;\O}
\,\leq\, \|\vec{\bB{\zeta}}\|_{\mX} + \|\vec{\bB{u}}\|_{\mbb{M}}
\leq C_{\bS} \dataflow\big(\|(\varphi,\vartheta)\|\big)\,.
\end{align}
where $\dataflow(x):= \|\mbf{f}\|x + \|\bB{l}_3\|_{0,s;\O}+\|\bB{u}_{\rm D}\|_{1/2,\Gamma}$ for all $x\in \R^+$.
\end{lemma}

\begin{proof}
The boundedness of $\mbb{A}_{\varphi,\vartheta}$ and the global inf-sup conditions \eqref{eq:GIS:flow:1} and \eqref{eq:GIS:flow:2} allow us to apply \cite[Theorem 2.6]{Ern2004}. Therefore, \eqref{syst:FP:weak:flow} admits a unique solution in
$\mX\times\mbb{M}$. Moreover, the corresponding stability estimate, together with the bounds for $\mbf{F}_{\varphi,\vartheta}$ and $\mbf{G}$, yields \eqref{eq:apriori:flow}. The constant $C_{\bS}$ depends only on $\alpha_{\mbb{A}}$ and the constants entering the continuity bounds for $\mbf{F}_{\varphi,\vartheta}$ and $\mbf{G}$; in particular, it is independent of $\varphi$ and $\vartheta$.
\end{proof}

\subsection{Solvability fully coupled system}\label{sec:CF:WP:coupled}

We now focus on showing the existence and uniqueness of solutions of the fully coupled weak formulation \eqref{syst:weak:coupled} using Banach's fixed-point theorem. To this end, we consider the fixed-point operator $\mbf{T}:(\mathrm{L}^\rho(\O))^2\to (\mathrm{L}^\rho(\O))^2$ defined in \eqref{def:T}, which thanks to Lemmas~\eqref{lemma:solv:mass}, \eqref{lemma:solv:heat} and \eqref{lemma:solv:flow} is well defined. Now, we start proving that $\mbf{T}$ is a self-mapping. Given $\epsilon > 0 $, we let $\mathcal{W}_{\epsilon}$ be the ball in $(\mathrm{L}^\rho(\O))^2$ defined by
\begin{align*}
 \mathcal{W}_\epsilon := \Big\{(\varphi,\vartheta)\in (\mathrm{L}^\rho(\O))^2:\quad \|\varphi\|_{0,\rho;\O} + \|\vartheta\|_{0,\rho;\O}\leq \epsilon \Big\}\,,
\end{align*}
The following lemma provides the conditions under which the fixed-point operator $\mbf{T}$ is a self-mapping on the $\epsilon$-ball $\mathcal{W}_\epsilon$.

\begin{lemma}\label{lemma:selfmapping}
Let $\rho$, $\varrho$, $r$ and $s$ be the real numbers within the range of values in \eqref{def:ranges}, and consider the constants $\alpha_{\bB{A}}$, $\alpha_{\mfk{A}}$, $C_{\mathrm{S}_1}$, $C_{\mathrm{S}_2}$ and $C_{\bS}$ the constants involved in the a priori estimates \eqref{eq:apriori:mass}, \eqref{eq:apriori:heat} and \eqref{eq:apriori:flow}, respectively.
Assume that the data are sufficiently small such that
\begin{align}
C_{\mathrm{S}_1}\datacons + C_{\mathrm{S}_2}\dataheat\leq \epsilon\,\qquad\text{and}\qquad C_{\bS}\,\dataflow(\epsilon)\,\leq\, \tfrac{1}{2}\min\{\alpha_{\bB{A}},\alpha_{\mfk{A}}\}\,.\label{eq:selfmap:cs}
\end{align}
Then, the fixed-point operator $\mbf{T}$ is a self-mapping over $\mathcal{W}_\epsilon$, i.e., $\mbf{T}(\mathcal{W}_\epsilon)\subseteq \mathcal{W}_\epsilon$.
\end{lemma}
\begin{proof}
 Given $(\varphi,\vartheta)\in \mathcal{W}_{\epsilon}$, from the a priori estimate \eqref{eq:apriori:flow}, we have that
\begin{align*}
 \|\bS(\varphi,\vartheta)\|_{0,r;\O} \,=\, \|\bB{u}\|_{0,r;\O} \,\leq\, C_{\bS}\dataflow(\epsilon)\,,
\end{align*}
and therefore condition \eqref{eq:selfmap:cs} implies that $\|\bB{u}\|_{0,r;\O}\leq \tfrac{1}{2}\alpha_{\bB{A}}$ and $\|\bB{u}\|_{0,r;\O}\leq \tfrac{1}{2}\alpha_{\mfk{A}}$. Then, using Lemmas~\ref{lemma:solv:mass} and \ref{lemma:solv:heat}, with velocity field given by $\bS(\varphi,\vartheta)$, we obtain
\begin{align*}
 \|\mbf{T}(\varphi,\vartheta)\|&\, = \,\|\mathrm{S}_1\big(\varphi,\bS(\varphi,\vartheta)\big)\|_{0,\rho;\O} + \|\mathrm{S}_2\big(\mathrm{S}_1(\varphi,\bS(\varphi,\vartheta)),\bS(\varphi,\vartheta)\big)\|_{0,\rho;\O} \,\leq \, C_{\mathrm{S}_1} \datacons\,+\, C_{\mathrm{S}_2}\dataheat\,\leq \, \epsilon\,.
\end{align*}
Thus, we conclude the proof.
\end{proof}

To address the Lipschitz continuity of the fixed-point operator $\mbf{T}$, we study the continuity of the three operators $\mathrm{S}_1$, $\mathrm{S}_2$ and $\bS$ related to the three decoupled problems \eqref{syst:FP:weak:mass}, \eqref{syst:FP:weak:heat} and \eqref{syst:FP:weak:flow}, respectively. For this end, and with the aim of obtaining a bounded Lipschitz continuity constant of $\mbf{T}$, we introduce the following regularity assumptions:

\begin{itemize}[leftmargin=\widthof{\indent\textbf{(H)}}]
\item[\textbf{(H)}]
There exists a regularity exponent $\varepsilon$ satisfying $\frac{d}{\rho} \leq \varepsilon < \frac{d}{2}$, such that
$\Phi\in {\rm H}^{-1/2+\varepsilon}(\Gamma)$, $\theta_{\rm D}\in {\rm H}^{1/2+\varepsilon}(\Gamma_{\rm D})$, and $\bB{u}_{\rm D}\in\bH^{1/2+\varepsilon}(\Gamma)$.
In addition, there exist positive constants $C_{\varepsilon,1}$, $C_{\varepsilon,2}$, and $C_{\varepsilon,3}$, independent of the corresponding input functions, such that the following properties hold.
\begin{itemize}[leftmargin=0.4cm]
\item[$1$.] For each $(\varphi,\bB{w})\in\L^\rho(\O)\times\bL^r(\O)$ with $\|\bB{w}\|_{0,r;\O}\leq\tfrac{1}{2}\alpha_{\bB{A}}$, the solution to \eqref{syst:FP:weak:mass} satisfies that $\bB{r}\in\bH^\varepsilon(\O)$ and
\begin{align}
\|\bB{r}\|_{\varepsilon;\O} \leq C_{\varepsilon,1} \datacons\,.\label{eq:bound:RH1}
\end{align}
\item[$2$.] For each $(\varphi,\bB{w})\in\L^\rho(\O)\times\bL^r(\O)$ with $\|\bB{w}\|_{0,r;\O}\leq\tfrac{1}{2}\alpha_{\mfk{A}}$, the solution to \eqref{syst:FP:weak:heat} satisfies that $\bB{t}\in\bH^\varepsilon(\O)$ and
\begin{align}
\|\bB{t}\|_{\varepsilon;\O} \leq C_{\varepsilon,2} \dataheat\,. \label{eq:bound:RH2}
\end{align}
\item[$3$.] For each $(\varphi,\vartheta)\in \L^\rho(\O)\times\L^\rho(\O)$, the solution to \eqref{syst:FP:weak:flow} satisfies that $\bB{\zeta}\in\mbb{H}^\varepsilon(\O)$ and
\begin{align}
\|\bB{\zeta}\|_{\varepsilon;\O} \leq C_{\varepsilon,3} \dataflow\big(\|(\varphi,\vartheta)\|\big)\,.\label{eq:bound:RH3}
\end{align}
\end{itemize}
\end{itemize}
For the regularity exponent $\varepsilon$ introduced in
\textbf{(H)}, we define the corresponding Sobolev exponent
\begin{align*}
\varepsilon^* := \frac{2d}{d-2\varepsilon}\,,
\end{align*}
where we have that $\varepsilon^*\geq r$.  Therefore, the Sobolev embedding theorem and the continuous inclusion $\mathrm L^{\varepsilon^*}(\O)\hookrightarrow\mathrm L^r(\O)$ yield $\bH^\varepsilon(\O) \hookrightarrow \bL^{\varepsilon^*}(\O) \hookrightarrow \bL^r(\O)$ and $
\mbb{H}^\varepsilon(\O) \hookrightarrow \mbb{L}^{\varepsilon^*}(\O) \hookrightarrow \mbb{L}^r(\O)\,$.
More precisely, denoting by $\bB{i}_{\varepsilon}: \bH^\varepsilon(\O)\to\bL^{\varepsilon^*}(\O)$ and $\bB{j}_{\varepsilon}: \mbb{H}^\varepsilon(\O)\to\mbb{L}^{\varepsilon^*}(\O)$ the corresponding Sobolev injections, we obtain
\begin{align}
\begin{aligned}
\|\bB{r}\|_{0,r;\O} \leq |\O|^{\widetilde{r}} \|\bB{i}_{\varepsilon}\| \|\bB{r}\|_{\varepsilon;\O}, \qquad
\|\bB{t}\|_{0,r;\O} \leq |\O|^{\widetilde{r}} \|\bB{i}_{\varepsilon}\| \|\bB{t}\|_{\varepsilon;\O}, \qquad
\|\bB{\zeta}\|_{0,r;\O} &\leq |\O|^{\widetilde{r}} \|\bB{j}_{\varepsilon}\| \|\bB{\zeta}\|_{\varepsilon;\O}, \label{eq:aux:RHs}
\end{aligned}
\end{align}
with $\tilde{r} = 1/r-1/\varepsilon^*$.
For ease in the read and considering $\varphi,\vartheta \in \mathcal{W}_{\epsilon}$, we introduce the following data-dependent constants arising in the right-hand sides of \eqref{eq:bound:RH1}, \eqref{eq:bound:RH2} and \eqref{eq:bound:RH3} respectively by
\begin{align}
 \widehat{C}_1 &:= C_{\varepsilon,1}\datacons\,, \qquad
 \widehat{C}_2 := C_{\varepsilon,2}\dataheat\,, \qquad
 \widehat{C}_3 := C_{\varepsilon,3}\dataflow(\epsilon)\,. \label{eq:bound:Cs}
\end{align}
Next, we present and prove the Lipschitz continuity of the operators $\mathrm{S}_1$, $\mathrm{S}_2$ and $\bS$.

\begin{lemma}\label{lemma:LC:mass}
Let $r$ and $\rho$ satisfy \eqref{def:ranges}, and assume the regularity hypothesis~\textbf{(H)}. Then, there exists a positive constant $L_{\mathrm{S}_1}$, independent of the data and of the input functions, such that
\begin{align}
\|\mathrm{S}_1(\varphi_1,\bB{w}_1) - \mathrm{S}_1(\varphi_2,\bB{w}_2)\|_{0,\rho;\O}
\leq L_{\mathrm{S}_1} \Big( \widehat{C}_1\big(L_{\mathcal{D}_1}+1\big) + L_{\rm fbk}\Big)
\|(\varphi_1,\bB{w}_1)-(\varphi_2,\bB{w}_2)\|\,,\label{eq:LC:mass}
\end{align}
for all $\varphi_1,\varphi_2\in\L^\rho(\O)$ and $\bB{w}_1,\bB{w}_2\in\bL^r(\O)$ with $\|\bB{w}_1\|_{0,r;\O},\|\bB{w}_2\|_{0,r;\O}
\leq\alpha_{\bB{A}}/2$.
\end{lemma}
\begin{proof}
The proof can be found in Appendix~\ref{appendix:proof:LC:mass}.
\end{proof}

\begin{lemma}\label{lemma:LC:heat}
Let $r$ and $\rho$ satisfy \eqref{def:ranges}, and assume the regularity hypothesis~\textbf{(H)}. Then, there exists a positive constant $L_{\mathrm{S}_2}$, independent of the data and of the input functions, such that
\begin{align}
\begin{aligned}\label{eq:LC:heat}
&\|\mathrm{S}_2(\varphi_1,\bB{w}_1) - \mathrm{S}_2(\varphi_2,\bB{w}_2)\|_{0,\rho;\O}
\leq L_{\mathrm{S}_2} \widehat{C}_2 \big(L_{\mathcal{D}_2}+1\big) \|(\varphi_1,\bB{w}_1)-(\varphi_2,\bB{w}_2)\|\,,
\end{aligned}
\end{align}
for all $\varphi_1,\varphi_2\in\L^\rho(\O)$ and $\bB{w}_1,\bB{w}_2\in\bL^r(\O)$ with $
\|\bB{w}_1\|_{0,r;\O}, \, \|\bB{w}_2\|_{0,r;\O} \leq \tfrac{1}{2}\alpha_{\mfk{A}}$.
\end{lemma}

\begin{proof}
The proof follows by the same arguments as the proof of Lemma~\ref{lemma:LC:mass}, using the global bilinear form \eqref{def:A:heat}, the Lipschitz continuity of $\mathcal{D}_2$ in \eqref{eq:LC:data}, and the regularity hypothesis~\textbf{(H)}.
\end{proof}

\begin{lemma}\label{lemma:LC:flow}
Let $r$ and $\rho$ satisfy \eqref{def:ranges}, and assume the regularity hypothesis~\textbf{(H)}. Then, there exists a positive constant $L_{\bS}$, independent of the data and the input functions, such that
\begin{align}
\begin{aligned}
&\|\bS(\varphi_1,\vartheta_1) - \bS(\varphi_2,\vartheta_2)\|_{0,r;\O}\\
&\qquad\leq L_{\bS}
\Big( \big(L_\mu + \delta_4L_{\mbf K}\big)\dataflow\big(\|(\varphi_2,\vartheta_2)\|\big) + L_{\mbf f}\Big)\|(\varphi_1,\vartheta_1)-(\varphi_2,\vartheta_2)\|\,,\label{eq:LC:S}
\end{aligned}
\end{align}
for all $(\varphi_1,\vartheta_1),(\varphi_2,\vartheta_2)\in(\L^\rho(\O))^2$.
\end{lemma}

\begin{proof}
The proof can be found in Appendix~\ref{appendix:proof:LC:flow}.
\end{proof}

\begin{lemma}\label{lemma:LC:T}
Let $(\rho,\varrho)$ and $(r,s)$ be pairs of conjugate exponents satisfying \eqref{def:ranges}. Assume that the regularity
hypothesis~\textbf{(H)} holds and that the data are sufficiently small so that, for some $\epsilon>0$, conditions \eqref{eq:selfmap:cs} are satisfied. Then, the restriction of the operator $\mbf T$, defined in \eqref{def:T}, to $\mathcal W_\epsilon$ is Lipschitz continuous. More precisely, there exists $L_{\mbf T}>0$ such that
\begin{align}
\|\mbf T(\varphi_1,\vartheta_1)
-
\mbf T(\varphi_2,\vartheta_2)\|
\leq
L_{\mbf T}
\|(\varphi_1,\vartheta_1)
-
(\varphi_2,\vartheta_2)\|\,,
\label{eq:LC:T}
\end{align}
for all $(\varphi_1,\vartheta_1),(\varphi_2,\vartheta_2) \in\mathcal W_\epsilon$.
\end{lemma}
\begin{proof}
For convenience, we introduce the total Lipschitz constants
associated with Lemmas~\ref{lemma:LC:mass}--\ref{lemma:LC:flow}:
\begin{align*}
\begin{aligned}%
\widetilde{L}_{\mathrm S_1} &:= L_{\mathrm S_1} \Big(\widehat C_1\big(L_{\mathcal D_1}+1\big) +  L_{\rm fbk} \Big)\,,\quad
\widetilde{L}_{\mathrm S_2} := L_{\mathrm S_2} \widehat C_2\big(L_{\mathcal D_2}+1\big)\,,\\[1ex]
\widetilde{L}_{\bS,\epsilon} &:= L_{\bS} \Big( \big(L_\mu+\delta_4L_{\mbf K}\big)\dataflow(\epsilon) + L_{\mbf f} \Big)\,.
\end{aligned}
\end{align*}
Given $(\varphi_i,\vartheta_i)\in\mathcal W_\epsilon$, $i\in\{1,2\}$, we set $\bB u_i:=\bS(\varphi_i,\vartheta_i)$,  $\phi_i := \mathrm S_1(\varphi_i,\bB u_i)$, and  $\theta_i:=\mathrm S_2(\phi_i,\bB u_i)$. Thus, $\mbf T(\varphi_i,\vartheta_i)=(\phi_i,\theta_i)$. By Lemma~\ref{lemma:selfmapping}, $\bB u_i$ satisfies the smallness conditions required in Lemmas~\ref{lemma:LC:mass} and \ref{lemma:LC:heat}. Since $(\varphi_2,\vartheta_2)\in\mathcal W_\epsilon$, there holds $\|\varphi_2\|_{0,\rho;\O} + \|\vartheta_2\|_{0,\rho;\O} \leq \epsilon$, and from Lemma~\ref{lemma:LC:flow}, we have
\begin{align}
\|\bB u_1-\bB u_2\|_{0,r;\O} \leq \widetilde{L}_{\bS,\epsilon}\|(\varphi_1,\vartheta_1) - (\varphi_2,\vartheta_2)\|\,.\label{eq:LC:T:flow}
\end{align}
In turn, Lemma~\ref{lemma:LC:mass} and \eqref{eq:LC:T:flow} imply
\begin{align}
\|\phi_1-\phi_2\|_{0,\rho;\O} \leq \widetilde{L}_{\mathrm S_1}\big(1+\widetilde{L}_{\bS,\epsilon}\big)\|(\varphi_1,\vartheta_1) - (\varphi_2,\vartheta_2)\|\,.\label{eq:LC:T:mass}
\end{align}
Similarly, using Lemma~\ref{lemma:LC:heat}, \eqref{eq:LC:T:flow}, and \eqref{eq:LC:T:mass}, we obtain
\begin{align}
\|\theta_1-\theta_2\|_{0,\rho;\O} \leq \widetilde{L}_{\mathrm S_2} \Big\{ \widetilde{L}_{\mathrm S_1} \big(1 + \widetilde{L}_{\bS,\epsilon}\big) + \widetilde{L}_{\bS,\epsilon}\Big\}
\|(\varphi_1,\vartheta_1) - (\varphi_2,\vartheta_2)\|\,.\label{eq:LC:T:heat}
\end{align}
Adding \eqref{eq:LC:T:mass} and \eqref{eq:LC:T:heat}, and defining
\begin{align}
L_{\mbf T} := \widetilde{L}_{\mathrm S_1} \big(1+\widetilde{L}_{\bS,\epsilon}\big) + \widetilde{L}_{\mathrm S_2}
\Big\{\widetilde{L}_{\mathrm S_1}\big(1 + \widetilde{L}_{\bS,\epsilon}\big) + \widetilde{L}_{\bS,\epsilon}\Big\}\,,\label{def:LC:T}
\end{align}
we obtain \eqref{eq:LC:T}, and thus conclude the proof.
\end{proof}
\begin{theorem}\label{thm:EU}
Let $(\rho,\varrho)$ and $(r,s)$ be pairs of conjugate exponents satisfying \eqref{def:ranges}. Assume that the regularity hypothesis~\textbf{(H)} holds and that, for some $\epsilon>0$, the data satisfy \eqref{eq:selfmap:cs} and
\begin{align}
&L_{\mbf T} < 1\,,\label{eq:bound:LC:T}
\end{align}
where $L_{\mbf T}$ is defined in \eqref{def:LC:T}. Then, the operator $\mbf T$, defined in \eqref{def:T}, has a unique fixed point $(\phi,\theta)\in\mathcal W_\epsilon$. Equivalently, the fully coupled problem \eqref{syst:weak:mass}, \eqref{syst:weak:heat}, and
\eqref{syst:weak:flow} has a unique solution
$(\vtmass,\vec{\bB{\phi}})\in\bH_1\times\bQ_1$, $(\vtheat,\theta)\in\bH_2\times\L^\rho(\O)$, and $(\vec{\bB{\zeta}},\vec{\bB{u}})
\in\mX\times\mbb M$, whose concentration-temperature component satisfies $(\phi,\theta)\in\mathcal W_\epsilon$. Moreover,
\begin{align}
\|(\vtmass,\vec{\bB{\phi}})\|       \leq C_{\mathrm S_1} \datacons\,, \qquad
\|(\vtheat,\theta)\|                \leq C_{\mathrm S_2} \dataheat\,,\qquad
\|(\vec{\bB{\zeta}},\vec{\bB{u}})\| \leq C_{\bS}\dataflow(\epsilon)\,.\label{eq:thm:aprioris}
\end{align}
\end{theorem}
\begin{proof}
The first two smallness conditions imply, by Lemma~\ref{lemma:selfmapping}, that $\mbf T(\mathcal W_\epsilon)\subseteq\mathcal W_\epsilon$. Furthermore, Lemma~\ref{lemma:LC:T} and \eqref{eq:bound:LC:T} show that $\mbf T$ is a contraction on $\mathcal W_\epsilon$. Since $\mathcal W_\epsilon$ is a closed subset of the Banach space $\L^\rho(\O)\times\L^\rho(\O)$, it is complete. Therefore, Banach's fixed-point theorem provides a unique fixed point $(\phi,\theta)\in\mathcal W_\epsilon$. The existence and uniqueness of the remaining components follow from the well-posedness of the three decoupled problems defining $\mathrm S_1$, $\mathrm S_2$, and $\bS$. Finally,
\eqref{eq:thm:aprioris}
follows from \eqref{eq:apriori:mass},
\eqref{eq:apriori:heat}, and \eqref{eq:apriori:flow},
respectively.
\end{proof}

\section{Discrete formulation and solvability analysis} \label{sec:DF:WP}

\subsection{Discrete formulation} \label{sec:DF}

In this section, we introduce the finite element spaces and discrete formulations for the decoupled systems \eqref{syst:weak:mass}, \eqref{syst:weak:heat}, and \eqref{syst:weak:flow}, as well as the scheme for the resulting fully coupled system. Although the continuous analysis is valid for $d\in\{2,3\}$, throughout this section we restrict the discrete analysis to the two-dimensional case $d=2$.

To begin with, given $h>0$, we let $\mathcal{T}_h$ be a regular triangulation of $\Omega$ composed of triangles $K\in\mathcal{T}_h$. The parameter $h$ denotes the mesh size and is defined as the largest diameter of the elements of $\mathcal{T}_h$. Over each element $K\in\mathcal{T}_h$, we denote by $\mathcal{P}_{\ell}(K)$ the space of polynomials of degree less than or equal to $\ell$. The space of homogeneous polynomials of degree exactly $\ell$ is denoted by
$\bar{\mathcal{P}}_\ell(K)$. The corresponding global polynomial spaces over $\O$ are then defined as
\begin{align*}
 \mathcal{P}_{\ell}(\O):= \Big\{\varphi_h\in \L^2(\O): \quad \varphi_h|_K\in \mathcal{P}_\ell(K)\quad \forall K\in \mathcal{T}_h  \Big\}\,.
\end{align*}
The vector and tensor versions of this global polynomial space, extended by components, are denoted by $\bBc{P}_{\ell}(\Omega)$ and $\mbb{P}_{\ell}(\O)$, respectively. The local Raviart-Thomas space of polynomial degree $\ell$ is defined by $\mbf{RT}_{\ell}(K) := \bBc{P}_{\ell}(K)\oplus \bar{\mathcal{P}}_{\ell}(K)\bB{x}$, for $\bB{x}\in \O$. The global $H(\divr)$-conforming extension of this space over $\O$ and its matrix version are
\begin{align*}
 \mbf{RT}_{\ell}(\O) &:= \Big\{\bB{w}_h\in \bH(\divr;\O): \quad \bB{w}_h|_K\in  \mbf{RT}_{\ell}(K)\quad \forall K\in \mathcal{T}_h\Big\}\,,\\[1ex]
 \mbb{RT}_{\ell}(\O) &:= \Big\{\bB{\tau}_h\in \mbb{H}(\bdiv;\O): \quad \bB{\tau}_{h,i}\in  \mbf{RT}_{\ell}(\O)\quad \forall i\in \{1,2\}\Big\}\,,
\end{align*}
where the notation $\bB{\tau}_{h,i}$ denotes the $i$-th row of the tensor $\bB{\tau}_h$. Now, before introducing the finite-dimensional spaces of our discrete formulation, we remark the following inclusions for the piecewise polynomial spaces here introduced:
\begin{align*}
 \mathcal{P}_\ell(\O)\subseteq \L^j(\O)\quad \text{and}\quad \mbf{RT}_\ell(\O)\subseteq \bH(\divr_j;\O) \qquad  \text{for}\quad j\in[1,+\infty)\,,
\end{align*}
which can be straightforwardly extended to $\bBc{P}_{\ell}(\O)$, $\mbb{P}_{\ell}(\O)$ and $\mbb{RT}_{\ell}(\O)$.
Then, the functional spaces arising in the weak formulations \eqref{syst:FP:weak:mass} and  \eqref{syst:FP:weak:heat} are discretized as described below
\begin{alignat*}{2}
\bH_{1,h} &:= \bBc{P}_{\ell}(\O) \times \mbf{RT}_{\ell}(\O)\,,
&\qquad
\bQ_{1,h} &:= {\rm Q}_h \times \widehat{\mathcal Q}_h\,, \\
\bH_{2,h} &:= \bBc{P}_{\ell}(\O) \times \Big(\mbf{RT}_{\ell}(\O)\cap\bH_{\Gamma_{\rm N}}(\div_{\varrho};\O)\Big)\,,
&\qquad
{\rm Q}_{2,h} &:= {\rm Q}_h\,,
\end{alignat*}
where ${\rm Q}_{h}:=\mathcal{P}_{\ell}(\O)$ is the discrete space employed for the approximation of the scalar unknowns in $L^{\rho}(\O)$. In addition, let $\mathcal{E}_{h}^{\Gamma}$ denote the partition of $\Gamma$ induced by the triangulation $\mathcal{T}_h$. We assume that the
boundary edges contained in each straight segment of $\Gamma$ can be grouped into pairs. Then, $\mathcal{E}_{2h}$ denotes the coarsened
partition of $\Gamma$ obtained by joining each pair of adjacent and collinear edges of $\mathcal{E}_{h}^{\Gamma}$. Following \cite[Section 4.4]{Gatica2014}, we define $\widehat{\mathcal{Q}}_{h}\subseteq {\rm H}^{1/2}(\Gamma)$ by
\begin{align*}
\widehat{\mathcal{Q}}_{h} := \Big\{\chi_h\in {\rm C}(\Gamma):\quad \chi_h|_e\in\mathcal{P}_1(e) \quad\forall\,e\in\mathcal{E}_{2h}\Big\}\,.
\end{align*}
Notice, in particular, that the constant functions belong to $\widehat{\mathcal{Q}}_h$. For the flow equations \eqref{syst:FP:weak:flow}, we employ the Arnold--Falk--Winther (AFW) elements \cite[Section 4.4.3]{Gatica2023}, therefore we set
\begin{alignat*}{2}
\mX_h &:=
\big(\mbb{P}_{\ell+1}(\O)\cap\Ltwotr\big)\times\left(\mbb{P}_{\ell+1}(\O)\cap\Hdivtr\right),\\[1ex]
\mbb M_h
&:=
\bBc{U}_h\times \left(\mbb{P}_{\ell}(\O)\cap\mbb L_{\tt skew}^2(\O)\right)\,,\qan \bBc U_h:=\bBc P_\ell(\O)\,.
\end{alignat*}
Alternatively, the PEERS-based spaces described in \cite[Section 4.4.2]{Gatica2023} and \cite{Gatica2021} can also be employed.

We next introduce the discrete solution operators $\mathrm{S}_{1,h},\mathrm{S}_{2,h}:{\rm Q}_h\times\bBc{U}_h\to {\rm Q}_h$ and $\bS_h: {\rm Q}_h\times{\rm Q}_h \to \bBc U_h$ which for $\varphi_h,\vartheta_h\in {\rm Q}_h $ and $\bB{w}_h\in\bBc{U}_h$ are given by
\begin{align*}
 \textrm{S}_{1,h}(\varphi_h,\bB{w}_h) = \phi_h\,,\qquad \textrm{S}_{2,h}(\varphi_h,\bB{w}_h) = \theta_h\,,\qquad \bS_h(\varphi_h,\vartheta_h) = \bB{u}_h\,,
\end{align*}
where $\phi_h\in {\rm Q}_h$, $\theta_h\in {\rm Q}_h$ and $\bB{u}_h\in \bBc{U}_h$ are components of the unique solutions (to be confirmed below) to the respective discretized formulations \eqref{syst:FP:weak:mass}, \eqref{syst:FP:weak:heat} and \eqref{syst:FP:weak:flow} given by:
\begin{itemize}[leftmargin=\widthof{\textbf{\quad(DP.3)} }]
 \item[\textbf{(DP.1)}] Find $(\vtmass_h,\vec{\bB{\phi}}_h)\in \bH_{1,h}\times \bQ_{1,h}$ such that:
\begin{equation}
\label{syst:D:FP:weak:mass}
\begin{array}{cclclcll}
a_{\varphi_h}\big(\vtmass_h,\vsmass_h\big) & + & b\big(\vsmass_h,\vec{\bB\phi}_h\big) & - & \widetilde b_{\bB w_h} \big(\vsmass_h,\vec{\bB\phi}_h\big)
&=& \mathcal F_{\varphi_h}(\vsmass_h) &\qquad \forall\,\vsmass_h\in\bH_{1,h}\,,\\[1ex]
b\big(\vtmass_h,\vec{\bB\psi}_h\big) & - & c\big(\vec{\bB\phi}_h,\vec{\bB\psi}_h\big) &&
&=& \mathcal G(\vec{\bB\psi}_h) &\qquad \forall\,\vec{\bB\psi}_h\in\bQ_{1,h}\,.
\end{array}
\end{equation}
\item[\textbf{(DP.2)}] Find $(\vtheat_h,\theta_h)\in \bH_{2,h}\times {\rm Q}_h$ such that:
\begin{equation}
\label{syst:D:FP:weak:heat}
\begin{array}{cclclcll}
\mfk a_{\varphi_h}\big(\vtheat_h,\vsheat_h\big) & + & \mfk b\big(\vsheat_h,\theta_h\big) & - & \mfk{\widetilde b}_{\bB w_h} \big(\vsheat_h,\theta_h\big)
& = & \mfk F(\vsheat_h) &\qquad \forall\,\vsheat_h\in\bH_{2,h}\,, \\[1ex]
\mfk b\big(\vtheat_h,\varkappa_h\big) & - & \mfk c(\theta_h,\varkappa_h) &&
& = & \mfk G(\varkappa_h) &\qquad \forall\,\varkappa_h\in{\rm Q}_h\,.
\end{array}
\end{equation}
\item[\textbf{(DP.3)}] Find $(\vec{\bB{\zeta}}_h,\vec{\bB{u}}_h)\in \mX_h\times \mbb{M}_h$ such that:
\begin{equation} \label{syst:D:FP:weak:flow}
\begin{array}{ccccll}
\mbf{a}_{\varphi_h,\vartheta_h}(\vec{\bB{\zeta}}_h,\vec{\bB{\eta}}_h) & + & \mbf{b}(\vec{\bB{\eta}}_h,\vec{\bB{u}}_h) & = & \mbf{G}(\vec{\bB\eta}_h) &\qquad \forall\,\vec{\bB{\eta}}_h\in \mX_h\,,\\[1ex]
\mbf{b}(\vec{\bB{\zeta}}_h,\vec{\bB{w}}_h) &-& \mbf{c}_{\varphi_h}(\vec{\bB{u}}_h,\vec{\bB{w}}_h) & = & \mbf{F}_{\varphi_h,\vartheta_h}(\vec{\bB{w}}_h) & \qquad \forall\, \bB{w}_h\in \mbb{M}_h\,.
\end{array}
\end{equation}
\end{itemize}
The fully coupled formulation composed of the discrete problems \textbf{(DP.1)}, \textbf{(DP.2)} and \textbf{(DP.3)} is established as follows: Find the three pairs $(\vec{\bB{r}}_h,\vec{\bB{\phi}}_h)\in \bH_{1,h}\times\mbf{Q}_{1,h}$, $(\vec{\bB{t}}_h,\theta_h)\in \bH_{2,h}\times {\rm Q}_{h}$, and $(\vec{\bB{\zeta}}_h,\vec{\bB{u}}_h)\in \mbb{X}_h\times \mbb{M}_h$ such that
\begin{equation}
\label{syst:D:weak:coupled}
\begin{array}{cclclcll}
a_{\phi_h}\big(\vtmass_h,\vsmass_h\big) & + & b\big(\vsmass_h,\vec{\bB\phi}_h\big) & - & \widetilde b_{\bB u_h} \big(\vsmass_h,\vec{\bB\phi}_h\big)
& = & \mathcal F_{\phi_h}(\vsmass_h) & \quad \forall\,\vsmass_h\in\bH_{1,h}\,,\\[1ex]
b\big(\vtmass_h,\vec{\bB\psi}_h\big) & - & c\big(\vec{\bB\phi}_h,\vec{\bB\psi}_h\big) &&
& = & \mathcal G(\vec{\bB\psi}_h)    & \quad \forall\,\vec{\bB\psi}_h\in\bQ_{1,h}\,,\\[1ex]
\mfk a_{\phi_h} \big(\vtheat_h,\vsheat_h\big) & + & \mfk b\big(\vsheat_h,\theta_h\big) & - & \mfk{\widetilde b}_{\bB u_h} \big(\vsheat_h,\theta_h\big)
& = & \mfk F(\vsheat_h) &\quad \forall\,\vsheat_h\in\bH_{2,h}\,,\\[1ex]
\mfk b\big(\vtheat_h,\varkappa_h\big) & - & \mfk c(\theta_h,\varkappa_h) &&
& = & \mfk G(\varkappa_h) &\quad \forall\,\varkappa_h\in{\rm Q}_h\,, \\[1ex]
\mbf a_{\phi_h,\theta_h} \big(\vec{\bB\zeta}_h,\vec{\bB\eta}_h\big) & + & \mbf b\big(\vec{\bB\eta}_h,\vec{\bB u}_h\big) &&
& = & \mbf G(\vec{\bB\eta}_h) &\quad \forall\,\vec{\bB\eta}_h\in\mX_h\,,\\[1ex]
\mbf b\big(\vec{\bB\zeta}_h,\vec{\bB w}_h\big) & - & \mbf c_{\phi_h} \big(\vec{\bB u}_h,\vec{\bB w}_h\big) &&
& = & \mbf F_{\phi_h,\theta_h} \big(\vec{\bB w}_h\big) &\quad \forall\,\vec{\bB w}_h\in\mbb M_h\,.
\end{array}
\end{equation}
The associated discrete fixed-point operator $\mbf{T}_h:({\rm Q}_h)^2\to ({\rm Q}_h)^2$ is defined by
\begin{equation}
\label{def:D:T}
\mbf T_h(\varphi_h,\vartheta_h) := \Big(\mathrm S_{1,h}\big(\varphi_h,\bS_h(\varphi_h,\vartheta_h)\big),
\mathrm S_{2,h}\Big(\mathrm S_{1,h}\big(\varphi_h,\bS_h(\varphi_h,\vartheta_h)\big),\bS_h(\varphi_h,\vartheta_h)\Big)\Big)\,.
\end{equation}

\subsection{Solvability of the decoupled discrete problems}
\label{sec:DF:WP:decoupled:problems}

We now proceed as in Section~\ref{sec:CF:WP:concentration} to establish the existence and uniqueness of the solutions to the decoupled discrete problems \eqref{syst:D:FP:weak:mass}, \eqref{syst:D:FP:weak:heat}, and \eqref{syst:D:FP:weak:flow}, respectively. We begin with the discrete linearized concentration problem \eqref{syst:D:FP:weak:mass}. As in the continuous setting, the bilinear forms $a_{\varphi_h}$ and $c$ are positive semidefinite, with bounds uniform with respect to $\varphi_h$ and $h$, and $c$ is symmetric. We introduce the discrete zero-mean subspace $\bQ_{1,h,0} := \big({\rm Q}_h\cap\L^\rho_0(\O)\big)\times\widehat{\mathcal Q}_h$.
Since the constant functions belong to both ${\rm Q}_h$ and $\widehat{\mathcal Q}_h$, we have $\vec{\bB e}:=(1,-1)\in\bQ_{1,h}$. Moreover, there holds the direct decomposition
\begin{align*}
\bQ_{1,h} = \bQ_{1,h,0}\oplus {\rm span}\{\vec{\bB e}\}\,.
\end{align*}
We next consider the kernel of the discrete $b$ restricted to $\bQ_{1,h,0}$, namely,
\begin{align*}
\mathcal{V}_h
& := \Big\{\vsmass_h=(\bB s_h,\bB\varsigma_h)\in\bH_{1,h}: \quad b(\vsmass_h,\vec{\bB\psi}_{0,h})=0 \quad \forall\,\vec{\bB\psi}_{0,h}\in\bQ_{1,h,0} \Big\} = \bBc{U}_h \times \mathcal V_{0,h}\,, \\
\mathcal{V}_{0,h}
& := \Big\{\bB\varsigma_h\in\mbf{RT}_\ell(\O):\quad\div(\bB\varsigma_h)=0\ \qin\O,\quad
\langle\bB\varsigma_h\cdot\bB\nu,\chi_h\rangle_\Gamma=0 \quad \forall\,\chi_h\in\widehat{\mathcal Q}_h \Big\}\,.
\end{align*}
The above characterization follows from $\div(\mbf{RT}_\ell(\O))\subseteq{\rm Q}_h$, and the fact that $\div(\mbf{RT}_\ell(\O))$ is orthogonal to ${\rm Q}_h\cap\L^\rho_0(\O)$, which also implies that $\div(\bB\varsigma_h)$ is constant in $\O$.

Then, for each $\vec{\bB\phi}_h,\vec{\bB{\psi}}_h\in\bQ_{1,h}$, there exist $\vec{\bB{\phi}}_{0,h},\vec{\bB{\psi}}_{0,h}\in\bQ_{1,h,0}$ such that $\vec{\bB\phi}_h = \vec{\bB{\phi}}_{0,h} + m_{\phi_h}\vec{\bB e}$ and $\vec{\bB{\psi}}_h = \vec{\bB{\psi}}_{0,h} + m_{\psi_h}\vec{\bB e}$, with $m_{\phi_h}$ and $m_{\psi_h}$ the mean of $\phi_h$ and $\psi_h$, respectively.
Given $\varphi_h\in Q_h$, we consider the bilinear form $\bB{A}_{\varphi_h}:(\bH_{1,h}\times\bQ_{1,h})^2\mapsto \R$ given by
\begin{align*}
\bB{A}_{\varphi_h} \big( (\vtmass_h,\vec{\bB{\phi}}_h), (\vsmass_h,\vec{\bB{\psi}}_h) \big)
:={}&
a_{\varphi_h}(\vtmass_h,\vsmass_h) + b(\vsmass_h,\vec{\bB{\phi}}_h) + b(\vtmass_h,\vec{\bB{\psi}}_h) - c(\vec{\bB{\phi}}_h,\vec{\bB{\psi}}_h)\,,
\end{align*}
for all $\vtmass_h,\vsmass_h\in\bH_{1,h}$. Similarly to \eqref{eq:A:decomp:mass}, we have
\begin{align*}
\bB A_{\varphi_h}\big((\vtmass_h,\vec{\bB\phi}_h),(\vsmass_h,\vec{\bB\psi}_h)\big)= \bB A_{\varphi_h,0}\big((\vtmass_h,\vec{\bB\phi}_{0,h}),(\vsmass_h,\vec{\bB\psi}_{0,h})\big) - \delta_1|\O|m_{\phi_h}m_{\psi_h}\,.
\end{align*}
where $\bB A_{\varphi_h,0}$ is the bilinear form restriction of $\bB A_{\varphi_h}$ to the subspace $(\bH_{1,h}\times\bQ_{1,h,0})^2$. The inf-sup conditions related to $\bB A_{\varphi_h,0}$ can be established by analysing $a_{\varphi_h}$, $b$ and $c$. For the discrete inf-sup conditions related to $a_{\varphi_h}$, we consider the bilinear forms $\bar a_{\varphi_h}$, $\bar b_1$, and $\bar b_2$ introduced in the proof of Lemma~\ref{lemma:infsup:a:mass}. The $\bL^2(\O)$-ellipticity of $\bar a_{\varphi_h}$ is naturally inherited in the discrete subspace $\bBc P_\ell(\O)$. In addition, the discrete inf-sup properties of $\bar b_1$ and $\bar b_2$
follow from \cite[eqn.~(2.22)]{Bernardi1988}. Consequently, there exists a constant $\widetilde\alpha_1>0$, independent of $h$ and
$\varphi_h$, such that the discrete analogues of \eqref{eq:IS:a:mass} hold on $\mathcal V_h$. For the discrete inf-sup condition related to $b$, we combine the discrete divergence stability of the Raviart--Thomas
spaces with the macro-edge construction of $\widehat{\mathcal Q}_h$ (cf.~\cite[Lemma 4.5]{Gatica2021b} and \cite[Section 4.4]{Gatica2014}), and conclude that there exists a constant $\widetilde\beta_1>0$, independent of $h$, such that
\begin{align*}
\sup_{\substack{\vsmass_h\in\bH_{1,h}\setminus\{\vec{\bB{0}}\}}} \frac{b(\vsmass_h,\vec{\bB\psi}_{0,h})}{\|\vsmass_h\|_{\bH_1}}
\geq
\widetilde\beta_1 \|\vec{\bB\psi}_{0,h}\|_{\bQ_1} \qquad \forall\,\vec{\bB\psi}_{0,h}\in\bQ_{1,h,0}\,.
\end{align*}
Thus, the discrete inf-sup conditions for $a_{\varphi_h}$ and $b$, together with the positive semi-definiteness of $a_{\varphi_h}$ and $c$, allow us to apply \cite[Theorem 4.1]{Gatica2023} to conclude that $\bB A_{\varphi_h,0}$ satisfies the two global inf-sup conditions on $\bH_{1,h}\times\bQ_{1,h,0}$ with constant $\widetilde\alpha_{\bB A,0}>0$, independent of $h$ and $\varphi_h$.

Then, following the same reasoning as in Section~\ref{sec:CF:concentration}, for $\bB{w}_h\in \bBc{U}_h$, such that $\|\bB w_h\|_{0,r;\O} \leq \tfrac{1}{2}\widetilde\alpha_{\bB A}$, the perturbed global form $\bB A_{\varphi_h,\bB w_h}$ satisfies global discrete inf-sup conditions on each of its two arguments, with constant $\widetilde\alpha_{\bB A}/2$. Thus, we can state the well-definedness of $\mathrm S_{1,h}$ in the following lemma.

\begin{lemma}
\label{lemma:D:solv:mass}
Let $r,\rho$ and their conjugate exponents $s,\varrho$ satisfy the conditions in \eqref{def:ranges}, together with $1/\rho+1/r=1/2$. Then, for each $(\varphi_h,\bB w_h)\in{\rm Q}_h\times\bBc U_h$ satisfying $\|\bB w_h\|_{0,r;\O} \leq \frac{\widetilde\alpha_{\bB A}}{2}$, there exists a unique $(\vtmass_h,\vec{\bB\phi}_h)\in\bH_{1,h}\times\bQ_{1,h}$ solving \eqref{syst:D:FP:weak:mass}. Consequently,
$\mathrm S_{1,h}$ is well-defined by $\mathrm S_{1,h}(\varphi_h,\bB w_h):=\phi_h$.
Moreover, there exists a constant $\widetilde C_{\mathrm S_1}>0$, independent of $\varphi_h$, $\bB w_h$, and $h$, such that
\begin{align}
\|\mathrm S_{1,h}(\varphi_h,\bB w_h)\|_{0,\rho;\O}
=     \|\phi_h\|_{0,\rho;\O} \leq \|\vtmass_h\|_{\bH_1} +  \|\vec{\bB\phi}_h\|_{\bQ_1}
\leq  \widetilde C_{\mathrm S_1}\datacons\,.\label{eq:D:a priori:mass}
\end{align}
\end{lemma}
\begin{proof}
The proof is a direct consequence of the preceding discussion and \cite[Theorem 1.2, Section II.1.2]{Brezzi1991}.
\end{proof}

Next, we observe that the well-posedness of the discrete linearized weak formulation of the heat equation \eqref{syst:D:FP:weak:heat} can be established in a similar way as in the previous lemma. Indeed, the positive semi-definiteness of $\mfk a_{\varphi_h}$ and the symmetry and positive semi-definiteness of $\mfk c$ are inherited by the corresponding discrete spaces. Moreover, it is straightforward to verify that the discrete kernel of $\mfk b$ is given by $\widetilde{\mathcal V}_h := \bBc P_\ell(\O)\times\widetilde{\mathcal V}_{0,h}$, where
\begin{equation*}
\widetilde{\mathcal V}_{0,h}
:= \Big\{\bB\tau_h\in\mbf{RT}_\ell(\O): \quad \div(\bB\tau_h)=0 \qin\O \qan \bB\tau_h\cdot\bB\nu = 0 \qon\Gamma_{\rm N}\Big\}\,.
\end{equation*}
Then, similarly to the discrete inf-sup conditions for $a_{\varphi_h}$, we can employ the same approach to establish the corresponding inf-sup condition for $\mfk{a}_{\varphi_h}$ with positive constant $\alpha_{2,h}$. Furthermore, a slight generalization of \cite[Lemma 4.1]{Caucao2020} allows us to conclude the existence of a $\beta_{2,h}>0$  such that the discrete inf-sup condition for $\mfk{b}$ is satisfied. This results, imply the discrete inf-sup condition for $\mfk{A}_{\varphi_h,\bB{w}_h}$, which is satisfied by each $(\varphi_h,\bB{w}_h)\in {\rm Q}_h \times\mbb{M}_h$, such that $\|\bB{w}_h\|_{0,r;\O} \leq \alpha_{\mfk{A}}/2$, where $\alpha_{\mfk{A}}$ is a positive constant depending only on $\alpha_{2,h}$, $\beta_{2,h}$, $\|\mfk{a}\|$ and $\|\mfk{c}\|$.

In summary, we have the following lemma.

\begin{lemma}\label{lemma:D:solv:heat}
	Let $(\rho,\varrho)$ a pair of conjugate numbers within the range of values in \eqref{def:ranges}.
	Then, for each $(\varphi_h,\bB{w}_h)\in {\rm Q}_h \times\mathbcal{P}_{\ell}(\O)$, such that $\|\bB{w}_h\|_{0,r;\O} \leq \widetilde{\alpha}_{\mfk{A}}/2$, there exists a unique $(\vtheat_h,\theta_h)\in \bH_{2,h}\times{\rm Q}_h$ solution to \eqref{syst:D:FP:weak:heat}, and hence one can define $\mathrm{S}_{2,h}(\varphi_h,\bB{w}_h) = \theta_h$. Moreover, there exists a positive constant $\widetilde{C}_{\mathrm{S}_2}$, independent of $(\varphi_h,\bB{w}_h)$ and $h$, such that
	\begin{align}\label{eq:D:a priori:heat}
		\|\mathrm{S}_{2,h}(\varphi_h,\bB{w}_h)\|  \,=\, \|\theta_h\|_{0,\rho;\O}\,\leq\,\|\vtheat_h\|_{\bH_2} + \|\theta_h\|_{0,\rho;\O}\leq \widetilde{C}_{\mathrm{S}_{2}}\dataheat\,.
	\end{align}
\end{lemma}
\begin{proof}
The proof is a direct consequence of the preceding discussion and \cite[Theorem 1.2, Section II.1.2]{Brezzi1991}.
\end{proof}

Finally, we provide the well-posedness result for the discrete weak formulation of the Stokes equation \eqref{syst:D:FP:weak:flow}, which is based on the scheme provided in \cite{Gatica2023}.
\begin{lemma}\label{lemma:D:solv:flow}
Let $(\rho,\varrho)$ a pair of conjugate numbers within the range of values in \eqref{def:ranges}.
Then, for each $\varphi_h,\vartheta_h\in {\rm Q}_h$, there exists a unique $(\vec{\bB{\zeta}}_h,\vec{\bB{u}}_h)\in \mX_h\times\mbb{M}_h$ solution to \eqref{syst:D:FP:weak:flow}, and hence one can define $\mathbf{S}_h(\varphi_h,\vartheta_h) = \bB{u}_h$. Moreover, there exists a positive constant $\widetilde{C}_{\mathbf{S}}$, independent of $\varphi_h,\vartheta_h$ and $h$, such that
\begin{equation} \label{eq:D:a priori:flow}
\begin{aligned}
\bS_h(\varphi_h,\vartheta_h)=\|\bB{u}_h\|_{0,r;\O}\leq \|\vec{\bB{\zeta}}_h\|_{\mX} + \|\bB{u}_h\|_{\mbb{M}}\,
\leq \widetilde{C}_{\mathbf{S}} \dataflow\big(\|(\varphi_h,\vartheta_h)\|\big)\,.
\end{aligned}
\end{equation}
\end{lemma}
\begin{proof}
Let us consider the following kernel related to the discrete bilinear form $\mbf{b}$:
\begin{align*}
\mathbcal V_h
:=
\big\{\vec{\bB\eta}_h\in\mX_h: \quad \mbf b(\vec{\bB\eta}_h,\vec{\bB w}_h)=0 \quad \forall\,\vec{\bB w}_h\in\mbb M_h \big\}\,.
\end{align*}
As in the continuous setting, the bilinear forms $\mbf a_{\varphi_h,\vartheta_h}$ and $\mbf c_{\varphi_h}$ are positive semidefinite. Moreover, the structure of the AFW spaces yields two discrete inf-sup conditions (on each argument) for $\mbf a_{\varphi_h,\vartheta_h}$ on $\mathbcal V_h$, with a constant $\widetilde\alpha_3>0$ independent of $h$, $\varphi_h$, and $\vartheta_h$. In addition, adapting the discrete inf-sup argument in \cite[Sections 6.2--6.3]{Caucao2025} to the conjugate exponents $r$ and $s$ implies the existence of a constant $\widetilde\beta_3>0$, independent of $h$, such that
\begin{align*}
\sup_{\substack{\vec{\bB\eta}_h\in\mX_h\setminus \{\vec{\bB{0}}\}}}
\frac{\mbf b(\vec{\bB\eta}_h,\vec{\bB w}_h)}{\|\vec{\bB\eta}_h\|_{\mX}}
\geq
\widetilde\beta_3 \|\vec{\bB w}_h\|_{\mbb M} \qquad \forall\,\vec{\bB w}_h\in\mbb M_h\,.
\end{align*}
Therefore, the discrete perturbed saddle-point problem satisfies the hypotheses of \cite[Theorem 3.2]{Gatica2023}. Consequently, there exists a constant $\widetilde\alpha_{\mbb A}>0$, independent of $h$, $\varphi_h$, and $\vartheta_h$, such that
\begin{align}
\sup_{\substack{
(\vec{\bB\eta}_h,\vec{\bB w}_h)\in\mX_h\times\mbb M_h\setminus\{(\vec{\bB 0},\vec{\bB 0})\}}}
\frac{\mbb A_{\varphi_h,\vartheta_h}\big((\vec{\bB\zeta}_h,\vec{\bB u}_h),(\vec{\bB\eta}_h,\vec{\bB w}_h)\big)
}{\|(\vec{\bB\eta}_h,\vec{\bB w}_h)\|}
&\geq
\widetilde\alpha_{\mbb A}\|(\vec{\bB\zeta}_h,\vec{\bB u}_h)\|
\quad &&
\forall(\vec{\bB\zeta}_h,\vec{\bB u}_h) \in\mX_h\times\mbb M_h\,, \label{eq:D:GIS:flow:1}\\
\sup_{\substack{(\vec{\bB\zeta}_h,\vec{\bB u}_h)\in\mX_h\times\mbb M_h\setminus\{(\vec{\bB 0},\vec{\bB 0})\}}}
\frac{ \mbb A_{\varphi_h,\vartheta_h}\big((\vec{\bB\zeta}_h,\vec{\bB u}_h),(\vec{\bB\eta}_h,\vec{\bB w}_h)\big)}{\|(\vec{\bB\zeta}_h,\vec{\bB u}_h)\|}
&\geq
\widetilde\alpha_{\mbb A}\|(\vec{\bB\eta}_h,\vec{\bB w}_h)\|
\quad && \forall (\vec{\bB\eta}_h,\vec{\bB w}_h)\in\mX_h\times\mbb M_h\,.\label{eq:D:GIS:flow:2}
\end{align}
The existence and uniqueness of the discrete solution now follow from \eqref{eq:D:GIS:flow:1}--\eqref{eq:D:GIS:flow:2} and \cite[Theorem 1.2, Section II.1.2]{Brezzi1991}. Finally, the boundedness of $\mbf G$ and $\mbf F_{\varphi_h,\vartheta_h}$ yields \eqref{eq:D:a priori:flow}.
\end{proof}

\subsection{Solvability of the fully coupled discrete problem}
\label{sec:DF:WP:coupled}

We now establish the existence of solutions to the fully coupled discrete problem. To this end, we consider the discrete fixed-point operator $\mbf T_h$ defined in \eqref{def:D:T}. The well-posedness results for the three decoupled discrete problems will first be used to show that $\mbf T_h$ maps a suitable closed ball of ${\rm Q}_h\times{\rm Q}_h$ into itself. We will then establish the continuity of this operator and apply Brouwer's fixed-point theorem.

Unlike the continuous analysis, no additional regularity hypothesis is required at the discrete level. Indeed, since the finite element spaces are finite-dimensional, every discrete function belongs to $\L^p(\O)$ for every finite $p$. However, the constants arising from the equivalence of norms or inverse inequalities generally depend on $h$. Consequently, these properties do not provide estimates that are uniform with respect to the mesh size. Nevertheless, for each fixed $h$, finite-dimensionality is sufficient to establish the continuity of the discrete solution operators and hence of $\mbf T_h$.

Given $\epsilon>0$, we introduce the closed ball
\begin{align*}
\mathcal W_{\epsilon,h}
:=
\Big\{ (\varphi_h,\vartheta_h)\in{\rm Q}_h\times{\rm Q}_h: \quad \|\varphi_h\|_{0,\rho;\O} + \|\vartheta_h\|_{0,\rho;\O}
\leq \epsilon \Big\}\,.
\end{align*}
We first establish that $\mbf T_h$ maps $\mathcal W_{\epsilon,h}$ into itself.

\begin{lemma}
\label{lemma:D:selfmapping}
Let $(\rho,\varrho)$ and $(r,s)$ be pairs of conjugate exponents satisfying \eqref{def:ranges}, together with $1/\rho+1/r=1/2$. Let $\widetilde\alpha_{\bB A}$ and $\widetilde\alpha_\mfk{A}$ denote the discrete stability constants associated with the concentration and heat problems, respectively, and let $\widetilde C_{\mathrm S_1}$, $\widetilde C_{\mathrm S_2}$, and $\widetilde C_{\bS}$ be the constants appearing in \eqref{eq:D:a priori:mass}, \eqref{eq:D:a priori:heat}, and \eqref{eq:D:a priori:flow}, respectively. Assume that, for some $\epsilon>0$, the data satisfy
\begin{align}
\widetilde{C}_{\mathrm S_1}
\datacons + \widetilde C_{\mathrm S_2}\dataheat \leq \epsilon\qquad \text{and}\qquad \widetilde{C}_{\bS} \dataflow(\epsilon)
\leq \tfrac{1}{2} \min \big\{ \widetilde\alpha_{\bB A}, \widetilde\alpha_\mfk{A}\big\}\,. \label{eq:D:selfmap:cs}
\end{align}
Then, $\mbf{T}_h$ is well-defined on $\mathcal W_{\epsilon,h}$ and $\mbf T_h(\mathcal W_{\epsilon,h}) \subseteq \mathcal W_{\epsilon,h}$.
\end{lemma}
\begin{proof}
This proof follows analogously to the one of Lemma~\ref{lemma:selfmapping} by making the corresponding changes in the constants of the a priori estimates of the decoupled discrete problems, and using the discrete operators instead of the continuous ones.
\end{proof}

We next establish the continuity required to apply Brouwer's
fixed-point theorem.

\begin{lemma}
\label{lemma:D:continuity:T}
Under the hypotheses of Lemma~\ref{lemma:D:selfmapping}, the operator $\mbf T_h: \mathcal W_{\epsilon,h} \to \mathcal W_{\epsilon,h}$ is continuous.
\end{lemma}

\begin{proof}
Let
$(\varphi_{i,h},\vartheta_{i,h})\in\mathcal W_{\epsilon,h}$,
$i\in\{1,2\}$, and set
$\bB u_{i,h}
:=
\bS_h(\varphi_{i,h},\vartheta_{i,h})$,
$\phi_{i,h}
:=
\mathrm S_{1,h}(\varphi_{i,h},\bB u_{i,h})$,
$\theta_{i,h}
:=
\mathrm S_{2,h}(\phi_{i,h},\bB u_{i,h})$. Thus,
$\mbf T_h(\varphi_{i,h},\vartheta_{i,h})
=
(\phi_{i,h},\theta_{i,h})$. The different estimates obtained in the proofs of Lemmas~\ref{lemma:LC:mass}--\ref{lemma:LC:flow}, before invoking the additional regularity hypothesis~\textbf{(H)}, remain valid when restricted to the corresponding finite-dimensional spaces. Indeed, subtracting the discrete problems and using the global discrete inf-sup of each discretised problem yields
\begin{align}
&
\|\mbf T_h(\varphi_{1,h},\vartheta_{1,h}) - \mbf T_h(\varphi_{2,h},\vartheta_{2,h})\|
\leq C_{2,h} \Big(\|\varphi_{1,h}-\varphi_{2,h}\|_{0,\rho;\O} + \|\vartheta_{1,h}-\vartheta_{2,h}\|_{0,\rho;\O}\Big)\,, \label{eq:D:continuity:T}
\end{align}
where $C_{2,h}>0$ depends on the corresponding stability and
Lipschitz constants, as well as on
\begin{align*}
&
\|\bB r_{2,h}\|_{0,r;\O}, \quad \|\phi_{2,h}\|_{0,\rho;\O}, \quad \|\bB t_{2,h}\|_{0,r;\O}, \quad \|\theta_{2,h}\|_{0,\rho;\O},\quad
\|\bB\zeta_{2,h}\|_{0,r;\O}, \quad\text{and}\quad \|\bB u_{2,h}\|_{0,r;\O}\,.
\end{align*}
These quantities are finite because the corresponding functions belong to finite-dimensional spaces. Therefore, for each fixed $(\varphi_{2,h},\vartheta_{2,h})\in\mathcal W_{\epsilon,h}$, the constant $C_{2,h}$ is finite, although it is not required to be uniform with respect to $h$. Hence, \eqref{eq:D:continuity:T} proves that $\mbf T_h$ is continuous on $\mathcal W_{\epsilon,h}$.
\end{proof}

We are now in a position to establish the solvability of the fully coupled discrete problem.

\begin{theorem}
\label{thm:D:EU}
Let $(\rho,\varrho)$ and $(r,s)$ be pairs of conjugate exponents satisfying \eqref{def:ranges}, together with
$1/\rho+1/r=1/2$. Assume that, for some $\epsilon>0$, conditions \eqref{eq:D:selfmap:cs} hold. Then, the
operator $\mbf T_h$, defined in \eqref{def:D:T}, has at least one fixed point $(\phi_h,\theta_h)\in \mathcal W_{\epsilon,h}$. Equivalently, the fully coupled discrete problem \eqref{syst:D:weak:coupled} has at least one solution
$(\vtmass_h,\vec{\bB{\phi}}_h) \in \bH_{1,h}\times\bQ_{1,h}$, $(\vtheat_h,\theta_h)\in \bH_{2,h}\times{\rm Q}_h$, $(\vec{\bB{\zeta}}_h,\vec{\bB{u}}_h)\in\mX_h\times\mbb M_h$, whose concentration--temperature component satisfies $(\phi_h,\theta_h)\in\mathcal W_{\epsilon,h}$. Moreover,
\begin{align}
\|(\vtmass_h,\vec{\bB{\phi}}_h)\|
\leq
\widetilde C_{\mathrm S_1} \datacons\,, \qquad
\|(\vtheat_h,\theta_h)\|
\leq \widetilde C_{\mathrm S_2} \dataheat\,,\qquad
\|(\vec{\bB{\zeta}}_h,\vec{\bB{u}}_h)\| &\leq \widetilde C_{\bS}\dataflow(\epsilon)\,. \label{eq:D:thm:aprioris}
\end{align}
\end{theorem}

\begin{proof}
The set $\mathcal W_{\epsilon,h}$ is nonempty, closed, bounded, and convex. Since ${\rm Q}_h\times{\rm Q}_h$ is finite-dimensional, $\mathcal W_{\epsilon,h}$ is also compact. Furthermore, Lemma~\ref{lemma:D:selfmapping} gives $\mbf T_h(\mathcal W_{\epsilon,h})\subseteq \mathcal W_{\epsilon,h}$, whereas Lemma~\ref{lemma:D:continuity:T} shows that $\mbf T_h$ is continuous on $\mathcal W_{\epsilon,h}$. Therefore, Brouwer's fixed-point theorem ensures the existence of at least one fixed point $(\phi_h,\theta_h)\in\mathcal W_{\epsilon,h}$.

For each fixed point, the remaining discrete components are uniquely determined by the well-posedness of the three decoupled problems defining $\mathrm S_{1,h}$, $\mathrm S_{2,h}$, and $\bS_h$. Consequently, the fixed point yields a solution to the fully coupled discrete problem \eqref{syst:D:weak:coupled}. Finally,
\eqref{eq:D:thm:aprioris}
follow directly from \eqref{eq:D:a priori:mass},
\eqref{eq:D:a priori:heat}, and
\eqref{eq:D:a priori:flow}, respectively.
\end{proof}

\section{A priori error estimates}
\label{sec:a:priori}

In this section, we establish an a priori error estimate for the Galerkin scheme \eqref{syst:D:weak:coupled}. To this end, let $(\vtmass,\vec{\bB{\phi}}) \in \bH_1\times\bQ_1$,  $(\vtheat,\theta) \in \bH_2\times\L^\rho(\O)$, $(\vec{\bB{\zeta}},\vec{\bB{u}}) \in \mX\times\mbb M$ be the components of the unique solution described in Theorem~\ref{thm:EU}, and let $(\vtmass_h,\vec{\bB{\phi}}_h) \in \bH_{1,h}\times\bQ_{1,h}$, $(\vtheat_h,\theta_h) \in \bH_{2,h}\times{\rm Q}_h$, $(\vec{\bB{\zeta}}_h,\vec{\bB{u}}_h) \in \mX_h\times\mbb M_h$ be the components of a discrete solution provided by
Theorem~\ref{thm:D:EU}. We define the component errors
\begin{align*}
E_{1,h} &:= \|(\vtmass,\vec{\bB{\phi}}) - (\vtmass_h,\vec{\bB{\phi}}_h)\|\,,\qquad
E_{2,h} := \|(\vtheat,\theta) - (\vtheat_h,\theta_h)\|\,,\qquad
E_{3,h} := \|(\vec{\bB{\zeta}},\vec{\bB{u}}) - (\vec{\bB{\zeta}}_h,\vec{\bB{u}}_h)\|\,.
\end{align*}
and the corresponding total error as $\mbf{E}_h := E_{1,h} + E_{2,h} + E_{3,h}$. Furthermore, given a subspace $Z_h$ of a Banach space
$(Z,\|\cdot\|_Z)$, we introduce the following distance function
\begin{align*}
{\rm dist}(z,Z_h) := \inf_{z_h\in Z_h} \|z-z_h\|_Z \qquad \forall\,z\in Z\,.
\end{align*}
For convenience, we introduce the effective discrete stability constants
$\alpha_1^\star := \tfrac{\widetilde\alpha_{\bB A}}{2}$,
$\alpha_2^\star := \tfrac{\widetilde\alpha_\mfk{A}}{2}$ and
$\alpha_3^\star := \widetilde\alpha_{\mbb A}$,
and set
\begin{align*}
\kappa_1 &:= \big(L_{\mathcal D_1}+1\big)\|\bB i_\varepsilon\|\dfrac{\widehat{C}_1}{\alpha_1^\star} + \frac{L_{\rm fbk}}{\alpha_1^\star}\,,\qquad
\kappa_2 := \big(L_{\mathcal D_2}+1\big)\|\bB i_\varepsilon\|\dfrac{\widehat{C}_2}{\alpha_2^\star}\,,\qquad
\kappa_3 := \big(L_\mu+\delta_4L_{\mbf K}\big)\|\bB j_\varepsilon\|\dfrac{\widehat{C}_3}{\alpha_3^\star} + \dfrac{L_{\mbf f}}{\alpha_3^\star}\,.
\end{align*}
We now establish a C\'ea-type estimate for the total error $\mbf E_h$.

\begin{theorem}
\label{thm:CEA}
Let the hypotheses of Theorems~\ref{thm:EU} and \ref{thm:D:EU} hold, and assume that
\begin{align}
\kappa_1+\kappa_2+\kappa_3 < \tfrac{1}{2}\,. \label{eq:aux:CEA}
\end{align}
Then, the total error is bounded as follows
\begin{align}
& \mbf E_h \,\leq\,
2 \big(1+\|\bB{A}\|/\alpha_1^\star\big){\rm dist} \big( (\vtmass,\vec{\bB{\phi}}), \bH_{1,h}\times\bQ_{1,h}\big) +
2 \big(1+\|\mfk{A}\|/\alpha_2^\star\big){\rm dist} \big((\vtheat,\theta),\bH_{2,h}\times{\rm Q}_h\big)\nonumber\\ &\qquad +
2 \big(1 + \|\mbb A\|/\alpha_3^\star\big){\rm dist}\big((\vec{\bB{\zeta}},\vec{\bB{u}}),\mX_h\times\mbb M_h\big)\,.\label{eq:CEA}
\end{align}
Here, $\|\bB A\|$, $\|\mfk{A}\|$, and $\|\mbb A\|$ denote continuity bounds for the corresponding global bilinear forms.
\end{theorem}
\begin{proof}
The standard argument for C\'ea estimates associated with saddle-point problems consists in adding and subtracting arbitrary elements of the corresponding discrete product spaces, applying the discrete global inf-sup conditions, and using the boundedness of the global bilinear forms. Then, using the same difference estimates as in Lemmas~\ref{lemma:LC:mass}--\ref{lemma:LC:flow}, we obtain
\begin{align}
E_{1,h}
&\leq \big( 1 +  \|\bB A\|/\alpha_1^\star\big)
{\rm dist}\big( (\vtmass,\vec{\bB{\phi}}), \bH_{1,h}\times\bQ_{1,h} \big)
+ \kappa_1\big( \|\phi-\phi_h\|_{0,\rho;\O} + \|\bB u-\bB u_h\|_{0,r;\O} \big)\,, \label{eq:CEA:mass} \\[1ex]
E_{2,h}
&\leq \big( 1 + \|\mfk{A}\|/\alpha_2^\star \big) {\rm dist} \big((\vtheat,\theta), \bH_{2,h}\times{\rm Q}_h\big)
+
\kappa_2 \big( \|\phi-\phi_h\|_{0,\rho;\O} + \|\bB u-\bB u_h\|_{0,r;\O}\big)\,,\label{eq:CEA:heat}\\[1ex]
E_{3,h}
&\leq \big( 1 + \|\mbb A\|/\alpha_3^\star\big){\rm dist} \big( (\vec{\bB{\zeta}},\vec{\bB{u}}), \mX_h\times\mbb M_h \big)
+ \kappa_3 \big( \|\phi-\phi_h\|_{0,\rho;\O} + \|\theta-\theta_h\|_{0,\rho;\O} \big)\,.\label{eq:CEA:flow}
\end{align}
By the definitions of the product norms, there hold
\begin{align*}
\|\phi-\phi_h\|_{0,\rho;\O} + \|\bB u-\bB u_h\|_{0,r;\O} &\leq E_{1,h}+E_{3,h} \leq \mbf E_h\,,\\
\|\phi-\phi_h\|_{0,\rho;\O} + \|\theta-\theta_h\|_{0,\rho;\O} &\leq E_{1,h} + E_{2,h} \leq \mbf E_h\,.
\end{align*}
Hence, adding \eqref{eq:CEA:mass}--\eqref{eq:CEA:flow}, we deduce that
\begin{align*}
\mbf E_h
&\leq \big( 1 + \|\bB A\|/\alpha_1^\star \big) {\rm dist} \big( (\vtmass,\vec{\bB{\phi}}),\bH_{1,h}\times\bQ_{1,h} \big)
+ \big(1 + \|\mfk{A}\|/\alpha_2^\star\big){\rm dist} \big((\vtheat,\theta),\bH_{2,h}\times{\rm Q}_h\big)\\
&\quad+
\big(1 + \|\mbb A\|/\alpha_3^\star\big){\rm dist} \big((\vec{\bB{\zeta}},\vec{\bB{u}}), \mX_h\times\mbb M_h \big) +
\big( \kappa_1 + \kappa_2 + \kappa_3\big) \mbf E_h\,.
\end{align*}
Then, from the above inequality and assumption \eqref{eq:aux:CEA}, we conclude \eqref{eq:CEA}.
\end{proof}
The next theorem establishes the rate of convergence of the developed numerical scheme under the assumption of additional regularity of the continuous solution.
\begin{theorem}
\label{thm:accuracy}
Let $(\vec{\bB r},\vec{\bB\phi})\in\bH_1\times\bQ_1$, $(\vec{\bB t},\theta)\in\bH_2\times\L^\rho(\O)$, and $(\vec{\bB\zeta},\vec{\bB u})\in\mX\times\mbb M$ be the unique solution to the coupled system \eqref{syst:weak:mass}, \eqref{syst:weak:heat}, and \eqref{syst:weak:flow} provided by Theorem~\ref{thm:EU}.
In addition, let
$(\vec{\bB r}_h,\vec{\bB\phi}_h)
\in\bH_{1,h}\times\bQ_{1,h}$,
$(\vec{\bB t}_h,\theta_h)
\in\bH_{2,h}\times{\rm Q}_h$, and
$(\vec{\bB\zeta}_h,\vec{\bB u}_h)
\in\mX_h\times\mbb M_h$
be a solution to the discrete scheme provided by Theorem~\ref{thm:D:EU}. Assume the hypotheses of Theorems~\ref{thm:EU} and \ref{thm:D:EU}, and that there exist $k\in[1,\ell+1]$ and $\delta\in(0,1]$ such that:
\begin{eqnarray*}
&\bB{r},\bB{z},\bB{t},\widetilde{\bB\sigma}\in\bH^k(\O)\,,\qquad
\divr(\bB z),\divr(\widetilde{\bB\sigma})\in{\rm W}^{k,\varrho}(\O)\,,\qquad
\phi,\theta\in{\rm W}^{k,\rho}(\O)\,,\qquad
\bB u\in\mbf W^{k,r}(\O)\,,
&\\[1ex]
&\bB{\zeta},\bB{\sigma}, \bB{\gamma}\in\mbb H^k(\O)\,,\qquad
\bdiv(\bB\sigma)\in\mbf W^{k,s}(\O)\qquad \text{and}\qquad
\xi\in{\rm H}^{1/2+\delta}(\Gamma)\,.&
\end{eqnarray*}
Then, there exists a positive constant $C$, independent of $h$,
such that
\begin{align*}
\mbf E_h
&\leq Ch^k
\Big\{
\|\bB r\|_{k,\O} +
\|\bB z\|_{k,\O} +
\|\divr(\bB z)\|_{k,\varrho;\O} +
\|\phi\|_{k,\rho;\O} +
\|\bB t\|_{k,\O} +
\|\widetilde{\bB\sigma}\|_{k,\O} +
\|\divr(\widetilde{\bB\sigma})\|_{k,\varrho;\O} \\
&\qquad\qquad +
\|\theta\|_{k,\rho;\O} +
\|\bB\zeta\|_{k,\O} +
\|\bB\sigma\|_{k,\O} +
\|\bdiv(\bB\sigma)\|_{k,s;\O} 
+\|\bB u\|_{k,r;\O} +
\|\bB\gamma\|_{k,\O} \Big\} +
Ch^\delta\|\xi\|_{1/2+\delta,\Gamma}\,.
\end{align*}
\end{theorem}

\begin{proof}
Proceeding as in \cite[Theorem~5.12]{Careaga2025a}, the result follows from the C\'ea estimate established in
Theorem~\ref{thm:CEA}, together with the standard approximation properties of the Raviart--Thomas interpolator and the local
polynomial projectors. These properties yield order $h^k$ for all the variables defined in the interior of $\O$, provided that $k\in[1,\ell+1]$. It only remains to consider the boundary variable $\xi$. Recall that $\widehat{\mathcal Q}_h$ is the space of continuous piecewise linear functions defined on the coarsened boundary partition $\mathcal E_{2h}$. This is precisely the macro-edge construction considered in \cite[(4.48)--(4.49)]{Dominguez2013}. The corresponding approximation property \cite[$({\rm AP}_{\Gamma,h})$]{Dominguez2013} establishes that, for each $\delta\in(0,1]$ and each $\xi\in{\rm H}^{1/2+\delta}(\Gamma)$,
\begin{align*}
{\rm dist}\big(\xi,\widehat{\mathcal Q}_h \big)
:= \inf_{\xi_h\in\widehat{\mathcal Q}_h} \|\xi-\xi_h\|_{1/2,\Gamma} \leq Ch^\delta\|\xi\|_{1/2+\delta,\Gamma}\,.
\end{align*}
Combining this estimate with the approximation properties of the remaining finite element spaces and applying Theorem~\ref{thm:CEA} completes the proof.
\end{proof}
\section{Numerical examples} \label{sec:numerical:examples}

In order to test the performance of the model and numerical scheme described in Section~\ref{sec:DF:WP}, we present a series of numerical examples encompassing tests of convergence with respect to manufactured solutions, and simulations of different scenarios in two-dimensional spatial dimensions. The numerical scheme has been implemented in the open source finite element library FreeFEM \cite{FreeFem}. The fixed-point iteration scheme is solved with a  relative tolerance of $\texttt{tol} = 10^{-6}$. Then, if $\mathbf{X}_h^{(m)}\in \R^{\tt dof}$ represents the vector of all degrees of freedom (${\tt dof}$) associated to the coupled discretization for the $m$-th iteration, the fixed-point scheme is stopped whenever
\begin{equation*}
\frac{\|\mathbf X_h^{(m)}-\mathbf X_h^{(m-1)}\|_{\ell^2}}
{\|\mathbf X_h^{(m)}\|_{\ell^2}}
\le \texttt{tol}\,,
\end{equation*}
where $\|\cdot\|_{\ell^2}$ is the $2$-norm in $\R^{\tt dof}$.
\subsection{Example 1: Accuracy test} \label{example:accuracy}
\begin{table}[t!]
\setlength{\cmidrulewidth}{0.5pt}
\setlength{\tabcolsep}{5pt}
\centering
\caption{Example 1: Numerical errors and experimental rates for the unknowns related to the concentration equation (top table), and unknowns related to the heat equation (bottom table) for a sequence of mesh sizes $h$. The errors have been computed with the manufactured solution \eqref{eq:manufactured:soln}.} \label{tab:accuracy-concentration}
\smallskip
\begin{tabular}{c | cc cc cc cc }
\multicolumn{9}{c}{Error table concentration equation}\\
\toprule
$h$ & $e_h(\bB{r})$ & $\mathcal R$ & $e_h(\bB{z})$ & $\mathcal R$ & $e_h(\phi)$ & $\mathcal R$ & $e_h(\xi)$ & $\mathcal R$\\
\cmidrule(lr){1-1} \cmidrule(lr){2-3}\cmidrule(lr){4-5}\cmidrule(lr){6-7} \cmidrule(lr){8-9}
0.3727 & 5.12e-2 & --   & 1.06e-1 & --   & 8.83e-3 & --   & 8.01e-2 & --   \\
0.1964 & 2.25e-2 & 1.28 & 5.45e-2 & 1.04 & 4.88e-3 & 0.93 & 2.22e-2 & 2.00 \\
0.0978 & 1.09e-2 & 1.04 & 2.68e-2 & 1.01 & 2.42e-3 & 1.01 & 7.02e-3 & 1.65 \\
0.0481 & 5.30e-3 & 1.02 & 1.35e-2 & 0.97 & 1.23e-3 & 0.95 & 1.75e-3 & 1.96 \\
0.0279 & 2.64e-3 & 1.28 & 6.72e-3 & 1.29 & 6.18e-4 & 1.26 & 4.80e-4 & 2.37 \\
\bottomrule
\end{tabular}\\
\smallskip
\begin{tabular}{c | cc cc cc}
\multicolumn{7}{c}{Error table heat equation}\\
\toprule
$h$ & $e_h(\bB{t})$ & $\mathcal R$ & $e_h(\widetilde{\bB{\sigma}})$ & $\mathcal R$ & $e_h(\theta)$ & $\mathcal R$\\
\cmidrule(lr){1-1} \cmidrule(lr){2-3}\cmidrule(lr){4-5}\cmidrule(lr){6-7}
0.3727 & 2.28e-1 & --   & 1.09e-1 & --   & 2.55e-2 & --   \\
0.1964 & 1.02e-1 & 1.26 & 5.51e-2 & 1.06 & 1.30e-2 & 1.05 \\
0.0978 & 5.15e-2 & 0.98 & 2.79e-2 & 0.98 & 6.65e-3 & 0.96 \\
0.0481 & 2.48e-2 & 1.03 & 1.41e-2 & 0.96 & 3.46e-3 & 0.92 \\
0.0279 & 1.27e-2 & 1.23 & 7.10e-3 & 1.27 & 1.72e-3 & 1.29 \\
\bottomrule
\end{tabular}
\end{table}

We consider the unit square $\Omega=(0,1)^2$, whose boundary consists of a Dirichlet part $\Gamma_{\rm D} := \{(x,0)\text{ or }(x,1): 0\leq x\leq 1\}$ and a Neumann part $\Gamma_{\rm N} := \{(0,x)\text{ or }(1,x): 0\leq x\leq 1\}$. Now, to illustrate the convergence estimates established in Theorem~\ref{thm:accuracy}, we choose the following manufactured solutions:
\begin{alignat}{2} \label{eq:manufactured:soln}
\begin{aligned}
\phi(\bB{x})   &= 1-0.9\exp \bigl(x_1x_2(x_1-1)(x_2-1)\bigr)\,,&\qquad
\theta(\bB{x}) &= 1+0.1\cos(2\pi x_1)\cos(\pi x_2)\,, \\
\bB{u}(\bB{x}) &= 0.1
\begin{pmatrix}
\sin(\pi x_1)\cos(\pi x_2)
\\[1mm]
-\cos(\pi x_1)\sin(\pi x_2)
\end{pmatrix},&\qquad
p(\bB{x})
&=
\sin(\pi x_1)\sin(\pi x_2)\,.
\end{aligned}
\end{alignat}
Here, the nonlinear constitutive functions used are given by
\begin{alignat*}{3}
f_{\rm bk}(\phi) &= \tfrac{1}{2}\phi\big(1-\tfrac{1}{2}\phi\big)^2\,, &\qquad \mathcal{D}_1(\phi) &= \exp(-\phi^2)\,,&\qquad \mathcal D_2(\phi) &= D_{20}+D_{21}\exp(-\phi^2)\,,\\
\mu(\phi) &= \big(1-\tfrac{1}{2}\phi\big)^{-2}\,,&\qquad \mbf{K}(\phi) &= K_0(1+\phi)\mbb{I}\,,&\qquad \mbf{f}(\phi,\theta) &= - (\beta_{\phi}\phi+\beta_{\theta}\theta)\bB{k}\,.
\end{alignat*}
We set $\delta_i=1$ for $i=1,\dots,4$, $D_{20}=0.1$, $D_{21}=0.01$, $K_0=0.01$ and $\beta_{\phi}=\beta_{\theta}=0.25$. The source terms and boundary data are obtained by inserting the manufactured solutions into the corresponding strong equations. In particular, the prescribed concentration flux is $\Phi=\bB{z}\cdot\bB{\nu}$ on $\Gamma$, whereas $\xi=-\phi|_{\Gamma}$ is the exact trace variable as $\phi$ is sufficiently smooth.

For the spatial discretization, we employ the lowest-order Raviart--Thomas spaces for the concentration and heat subproblems and the lowest-order Arnold--Falk--Winther element for the flow subproblem. The nonlinear system is solved by a Picard iteration, initialized with $\phi_h^{(0)}=0.1$ and $\theta_h^{(0)}=1$. We choose $r=\rho=4$ and $s=\varrho=4/3$. The errors are defined by
\begin{alignat*}{4}
e_h(\bB{r}) &:= \|\bB{r}-\bB{r}_h\|_{0,\Omega},&\quad
e_h(\bB{z}) &:= \|\bB{z}-\bB{z}_h\|_{\div_{\varrho};\Omega}, &\quad
e_h(\phi)   &:= \|\phi-\phi_h\|_{0,\rho;\Omega},&\quad
e_h^0(\xi)  &:= \|\xi-\xi_h\|_{0,\Gamma},\\
e_h(\bB{t}) &:= \|\bB{t}-\bB{t}_h\|_{0,\Omega}, &\quad
e_h(\widetilde{\bB{\sigma}}) &:= \|\widetilde{\bB{\sigma}} - \widetilde{\bB{\sigma}}_h\|_{\div_{\varrho};\Omega}, &\quad
e_h(\theta)     &:= \|\theta-\theta_h\|_{0,\rho;\Omega},&\quad
e_h(\bB{\zeta}) &:= \|\bB{\zeta}-\bB{\zeta}_h\|_{0,\Omega}, \\
e_h(\bB{\sigma}) &:= \|\bB{\sigma} - \widehat{\bB{\sigma}}_h\|_{\div_s;\Omega}, &\quad
e_h(\bB{u}) &:= \|\bB{u}-\bB{u}_h\|_{0,r;\Omega}, &\quad
e_h(\bB{\gamma}) &:= \|\bB{\gamma}-\bB{\gamma}_h\|_{0,\Omega}, &\quad
e_h(p)      &:= \|p-p_h\|_{0,\Omega},\\
e_h(\xi) &:= \|\xi-\xi_h\|_{1/2,\Gamma}.
\end{alignat*}
As in \cite[Section~5.1]{Gatica2023}, we restore the prescribed pressure mean. In our implementation, the pressure is computed as
\begin{align*}
p_h|_K
=-\frac{1}{d|K|}\int_K\tr(\bB{\sigma}_h)+\frac{4}{\pi^2}
\qquad\forall K\in\mathcal T_h\,.
\end{align*}
The experimental rate between two consecutive mesh sizes $h_1$ and $h_2$, with respective errors $e_{h_1}$ and $e_{h_2}$, is computed as
\begin{equation*}
\mathcal{R} = \frac{\log\bigl(e_{h_1}/e_{h_2}\bigr)}{\log(h_1/h_2)}\,.
\end{equation*}

\begin{table}[t!]
\setlength{\cmidrulewidth}{0.5pt}
\setlength{\tabcolsep}{5pt}
\centering
\caption{Example 1: Numerical errors and experimental rates for the unknowns related to the flow equations for a sequence of mesh sizes $h$. The errors have been computed with the manufactured solution \eqref{eq:manufactured:soln}.}\label{tab:accuracy-flow}
\smallskip
\begin{tabular}{c | cc cc cc | cc cc}
\multicolumn{11}{c}{Error table flow equations}\\
\toprule
$h$ & $e_h(\bB{\zeta})$ & $\mathcal R$ & $e_h(\bB{\sigma})$ & $\mathcal R$ & $e_h(\bB{u})$ & $\mathcal R$
& $e_h(\bB{\gamma})$ & $\mathcal R$ & $e_h(p)$ & $\mathcal R$\\
\cmidrule(lr){1-1}\cmidrule(lr){2-3} \cmidrule(lr){4-5} \cmidrule(lr){6-7}\cmidrule(lr){8-9} \cmidrule(lr){10-11}
0.3727 & 3.70e-2 & --   & 6.65e-1 & --   & 2.48e-2 & --   & 4.86e-2 & --   & 1.10e-1 & --   \\
0.1964 & 1.87e-2 & 1.07 & 2.82e-1 & 1.34 & 1.09e-2 & 1.29 & 2.61e-2 & 0.97 & 5.87e-2 & 0.99 \\
0.0978 & 9.16e-3 & 1.02 & 1.40e-1 & 1.01 & 5.45e-3 & 0.99 & 1.29e-2 & 1.01 & 2.91e-2 & 1.01 \\
0.0481 & 4.56e-3 & 0.98 & 6.81e-2 & 1.01 & 2.63e-3 & 1.03 & 6.45e-3 & 0.98 & 1.45e-2 & 0.98 \\
0.0279 & 2.29e-3 & 1.27 & 3.44e-2 & 1.26 & 1.34e-3 & 1.24 & 3.24e-3 & 1.26 & 7.30e-3 & 1.26 \\
\bottomrule
\end{tabular}
\end{table}

The results show first-order convergence for all the variables and for the total error with the exception of the trace variable $\xi_h$, in agreement with the estimates of Theorem~\ref{thm:accuracy} for the lowest-order method.
In particular, after adopting the AFW choice $\gamma_h\in\mathbb P_0(\mathcal T_h)$, both $\bB{\zeta}_h$ and $\gamma_h$
converge with order one. The trace variable $\xi_h$ exhibits approximately second-order convergence in $\mathrm H^{1/2}(\Gamma)$ which suggests that the theoretical error bound could be further improved.
Finally, all the discrete systems reach the prescribed Picard tolerance in five iterations.

\subsection{Example 2: Batch sedimentation}\label{sec:example:2}

\begin{figure}[t!]
\centering
 \includegraphics[scale=0.8]{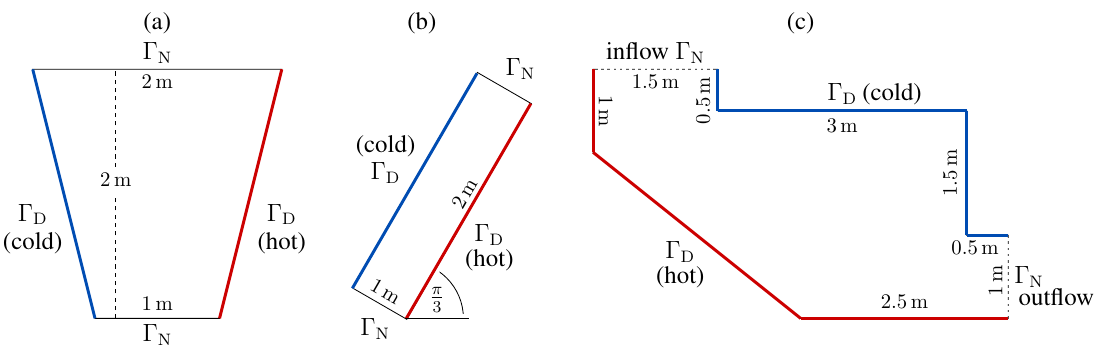}
 \caption{Vessel geometries (spatial domains) for Example 2: batch Kynch test (a) and Diehl test (b); and for Example 3: continuous sedimentation (c). Dirichlet and Neumann boundaries, respectively denoted by $\Gamma_{\rm D}$ and $\Gamma_{\rm N}$, are shown in each diagram.}\label{fig:domains}
\end{figure}

In this example, we show the performance of the proposed mixed finite element formulation for the coupled heat-assisted sedimentation problem in the evolutionary case, where we solve for $\phi_h^{n+1}$, $\theta_h^{n+1}$ and $\bB{u}_h^{n+1}$ at the time $t^{n+1} = (n+1)\Delta t$ with $\Delta t>0$ being the time step. In this case, $\delta_1 = \delta_2=\delta_3 = 1/\Delta t$, and $l_1 = \phi^n_h/\Delta t$, $l_2 = \theta^n_h/\Delta t$ and $\bB{l}_3 = \bB{u}_h^{n}/\Delta t$ in agreement with a backward Euler approximation of a time derivative. In this test, we consider the batch sedimentation process in which $\Phi = 0$ therefore there is no volumetric flow entering or leaving the domain. Furthermore, we consider the following functions
\begin{align*}
 \fbk(\phi) & =  v_0 \phi (1 - \phi)^2\,,\qquad \mathcal{D}_1(\phi)= 0.015\exp(-\phi^2)\,,\qquad \mathcal{D}_2(\phi) = \dfrac{k_{\rm f}(1 - \phi) + k_{\rm s}\phi}{c_{\rm pf}(1 - \phi) + c_{\rm ps}\phi}\,,\\
 \mu(\phi,\theta) &=  0.1 (1 - 0.8\phi)^{-2.5} 10^{247.8/(\theta + 273.15 - 140)}\,,\quad
\mbf{f}(\phi,\theta)= g\Big( \rho(\phi) - \rho_{\rm f}(1 - \phi) \beta_{\rm f}(\theta - \theta_{\rm a})\Big)\bB{k}\,,\\
\rho(\phi) & = \rho_{\rm s}\phi + \rho_{\rm f} (1-\phi)\,.
 \end{align*}
where $v_0 = 0.5\, \rm m/s$, $g = 9.8\,\rm m/s^2$, $k_{\rm s} = 3\,\rm W/mK$, $k_{\rm f} = 0.58\,\rm W/mK$, $c_{\rm ps} = 0.0741\,\rm J/kg$, $c_{\rm pf} = 4.18\,\rm J/kg$, $\rho_{\rm f} = 1000\,\rm kg/m^3$, $\rho_{\rm s} = 2650\,\rm kg/m^3$, $\Delta \rho  = 1650\,\rm kg/m^3$, $\beta_{\rm f} = 0.000207\,\rm 1/C^{\circ}$, $\theta_{\rm a} = 25 \rm C^{\circ}$. The temperature dependent factor in function $\mu$ is an amplified version of the one stated for the viscosity function in \cite{Reyes2019}. This setting corresponds to a highly viscous mixture with a large settling velocity $v_0$ to better see the combined effect of sedimentation and bulk velocity, thus serving primarily as an illustrative academic example.  Furthermore, we set $\delta_4=0$ such that no influence of $\mbf{K}$ is considered.

We simulate two sub-examples, a standard batch Kynch test and a Diehl test \cite{Careaga2025} where the discontinuous initial concentration is assumed to be concentrated at the top of the vessel. For the first test, we consider a domain with the shape of a downward-pointing truncated triangle, see vessel (a) in Figure~\ref{fig:domains}, with left and right inclined Dirichlet boundaries and top and bottom Neumann boundaries.
The homogeneous initial conditions for this example are: $\phi^0_h = 0.3$, $\theta^0_h = \theta_{\rm a}$ and $\bB{u}_h^0 = \bB{0}$. For the cold (left) boundary we set $\theta_{\rm D} = 20 \,\rm C^{\circ}$ and for the hot (right) boundary $\theta_{\rm D} = 40\,\rm C^{\circ}$. We set $\Delta t=0.005\,\rm s$ and simulate the system up to $1000$ time iterations with polynomial degree $\ell=0$ and meshsize $h=0.0496$.
In Figure~\ref{fig:batch1}, we show the approximated solution for $\phi_h$, $\theta_h$ and $\bB{u}_h$ at $t = 0.5,1.5, 3,5\,\rm s$. For each time, 4 Picard iterations were needed. The evolution of the sedimentation process is visible as time evolves, where at $t=5\,\rm s$ the solid phase is almost entirely separated from the liquid. In addition, we observe how the temperature gets to an equilibrium, and that the effect of temperature on the magnitude of the velocity field is more pronounced at earlier times. The total mass is preserved for each time iteration.

For the second sub-example, the Diehl test, we consider an inclined channel whose computational domain corresponds to a rectangle inclined $\pi/3$ with respect to the horizontal line, see vessel (b) in Figure~\ref{fig:domains}. As before, the inclined left boundary corresponds to the Dirichlet cold boundary with $\theta_{\rm D} = 20 \,\rm C^{\circ}$  and the inclined right boundary is the Dirichlet hot boundary with $\theta_{\rm D} = 40 \,\rm C^{\circ}$. The main difference with respect to the previous example is the initial condition for the concentration, which is set as
\begin{align}\label{eq:phi0:DT}
 \phi_h^0 = \begin{cases}
              0.3 &\text{ if } x_1 + \sqrt{3}x_2\geq 0.8\,,\\
             0 &\text{ if } x_1 + \sqrt{3}x_2 < 0.8\,.
            \end{cases}
\end{align}
We keep the rest of the parameters and initial conditions as for the Kynch test. The meshsize in this case is $h=0.0398$. In Figure~\ref{fig:batch2}, we show the approximated solution for $t = 0.05, 0.25, 0.5, 1, 1.5, 3\, \rm s$. We observe that the volume fraction of solid particles, initially concentrated at the 30\% upper part of the vessel begins to move towards the right boundary, and slide through it until it reaches the bottom of the vessel where it settles up to a high concentration. As before, the heat flow propagates along the domain reaching an equilibrium already at $t=3\,\rm s$.
The velocity field, as also shown in \cite{Careaga2025} develops an initially circular behaviour, which latter on due to the increase of heat in right wall, the magnitude along this wall increases. In this example, the average Picard iteration for the first was 6.

\begin{figure}[t!]
\centering
\includegraphics[scale=0.7]{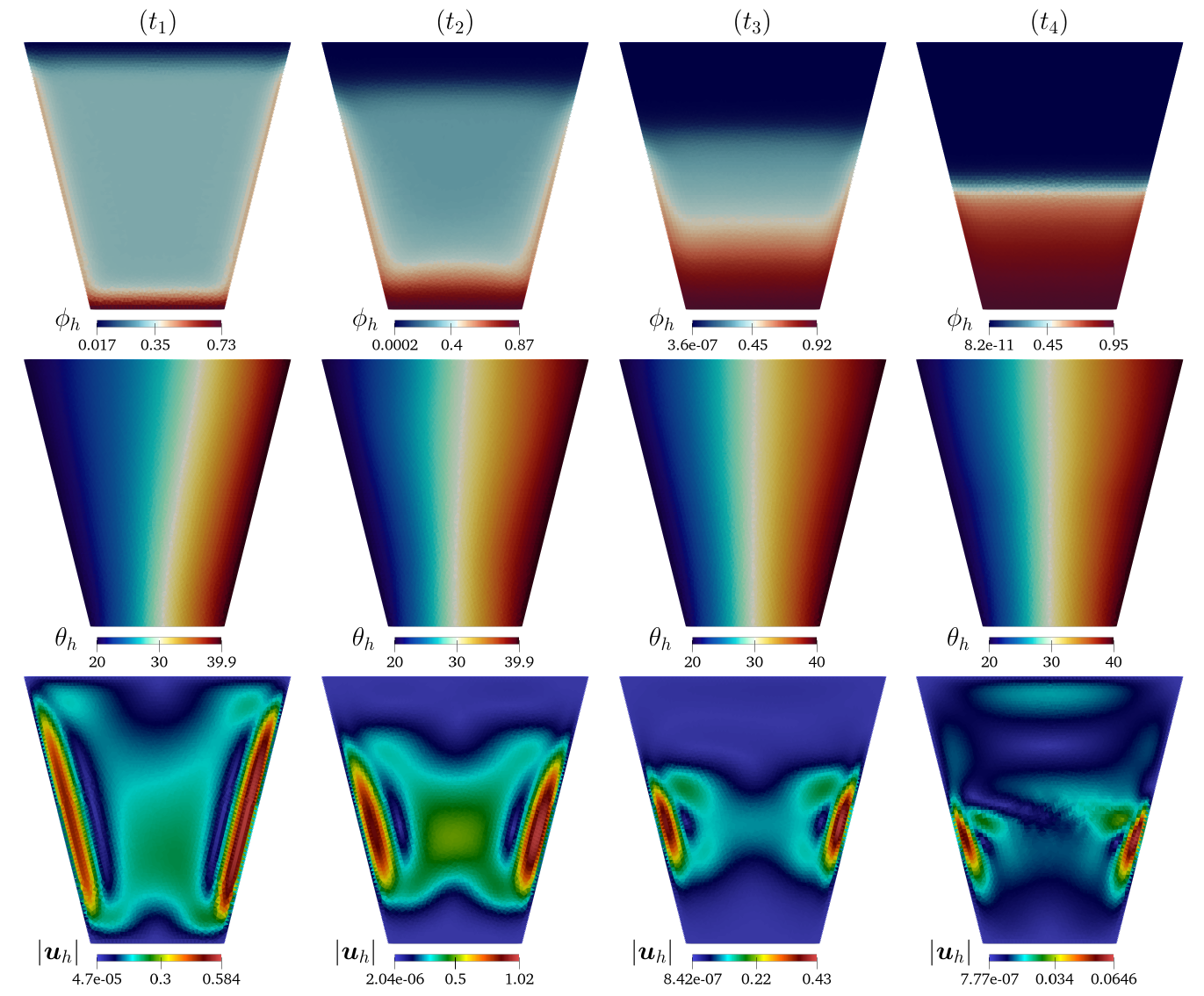}
\caption{Example 2: snapshots of the numerical solution for the batch Kynch test with homogeneous initial condition $\phi^0_h = 0.3$, $\theta^0_h = 25\,\rm C^{\circ}$ and $\bB{u}^0_h = \bB{0}$ at four times $t_1 = 0.5\,\rm s$, $t_2 = 1.5\,\rm s$, $t_3 = 3\,\rm s$, $t_4 = 5\,\rm s$ with $\Delta t = 0.005\,\rm s$. The meshsize is $h=0.0496$, polynomial degree $\ell=0$ and total number of degrees of freedom is ${\tt dof} = 284560$.
}\label{fig:batch1}
\end{figure}

\begin{figure}[t!]
\centering
\includegraphics[scale=0.7]{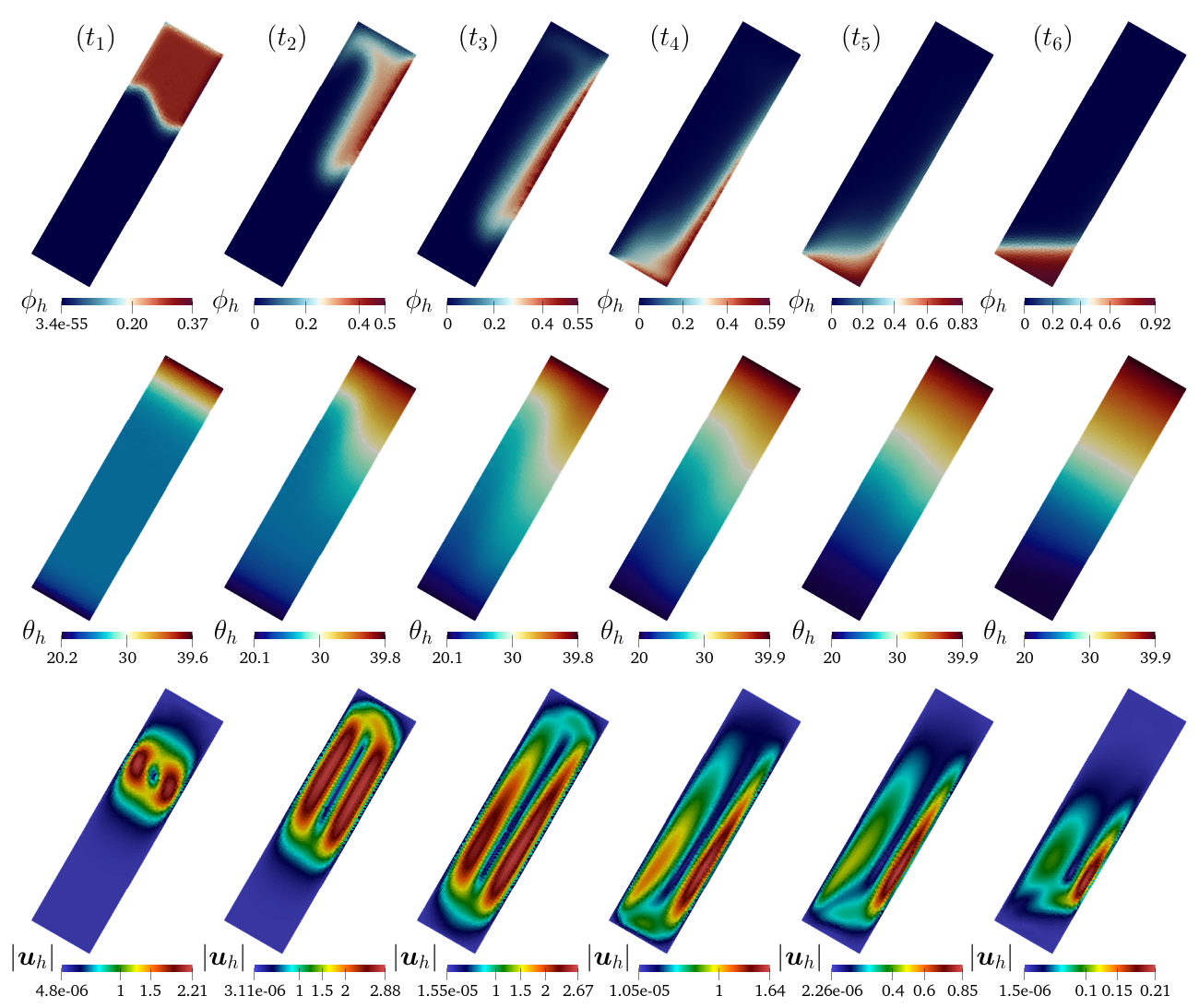}
\caption{Example 2: snapshots of the numerical solution for the Diehl test with initial concentration $\phi_h^0$ given by \eqref{eq:phi0:DT} and $\theta^0_h = 25\,\rm C^{\circ}$ and $\bB{u}^0_h = \bB{0}$ at six times $t_1 = 0.05\,\rm s$, $t_2 = 0.25\,\rm s$, $t_3 = 0.5\,\rm s$, $t_4 = 1\,\rm s$, $t_5 = 1.5\,\rm s$, $t_6 = 3\,\rm s$ with $\Delta t = 0.005\,\rm s$. The meshsize is $h= 0.0398$, polynomial degree $\ell=0$ and total number of degrees of freedom is ${\tt dof} = 136176$.}\label{fig:batch2}
\end{figure}

\subsection{Example 3: Continuous sedimentation}\label{sec:example:3}
\begin{figure}[t!]
\centering
\includegraphics[scale=0.8]{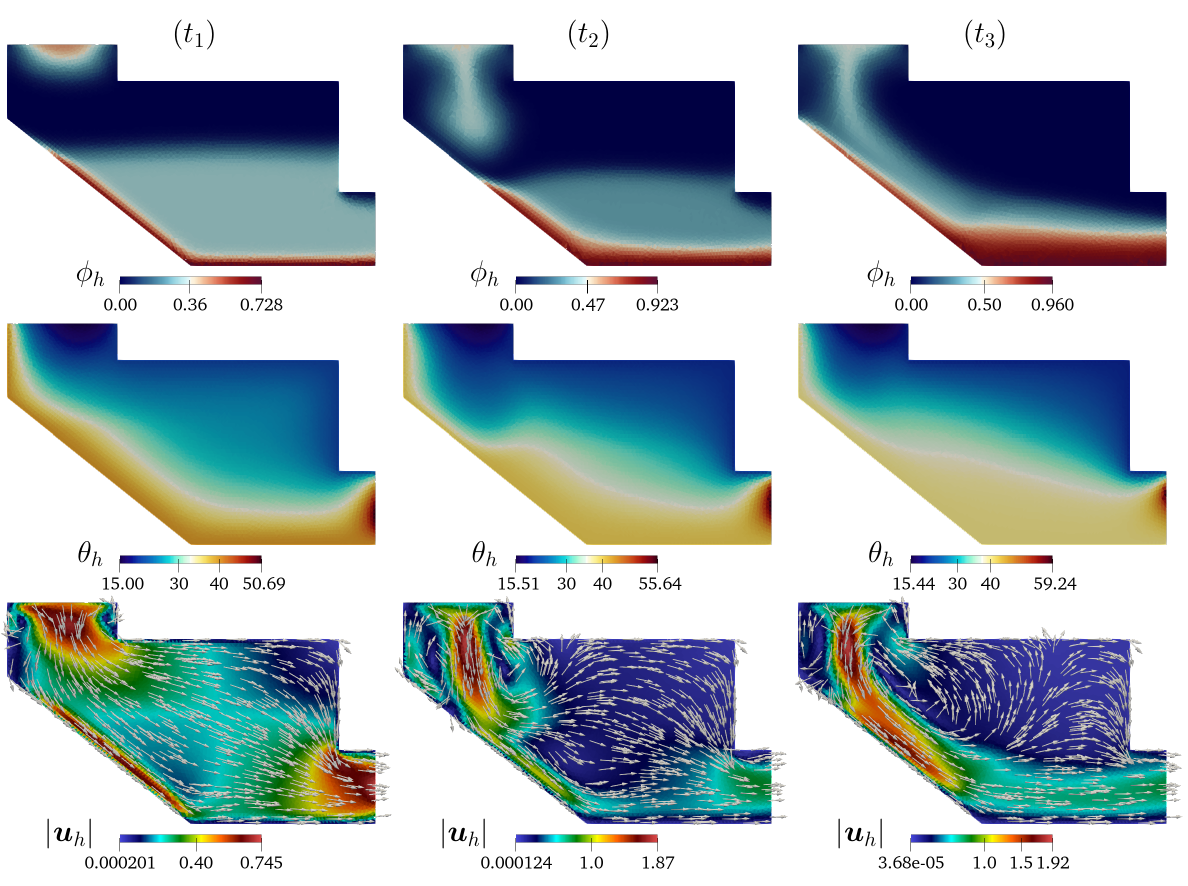}
\caption{Example 3: snapshots of the numerical solution of the continuous sedimentation example with initial concentration $\phi_h^0$ given by \eqref{eq:phi0:cont} and $\theta^0_h = 25\,\rm C^{\circ}$ and $\bB{u}^0_h = \bB{0}$ at three times $t_1 = 0.5\,\rm s$, $t_2 = 1.5\,\rm s$, $t_3 = 3\,\rm s$ with $\Delta t = 0.001\,\rm s$. The meshsize is $h = 0.0826$, polynomial degree $\ell=0$ and total number of degrees of freedom is ${\tt dof} = 320820$.}\label{fig:cont1}
\end{figure}
In this example, we consider the continuous sedimentation process in a time-dependent setting, as in Example~\ref{sec:example:2}. The computational domain, vessel (c) in Figure~\ref{fig:domains}, simulates a sedimentation tank with inlet and outlet (Neumann) boundaries. In this case, the Dirichlet boundaries corresponding to the cold and hot walls are chosen as described in Figure~\ref{fig:domains}, in which the temperature is set to $\theta_{\rm D} = 40\,\rm C^{\circ}$ for the hot wall, and $\theta_{\rm D}= 20\,\rm C^{\circ}$ for the cold wall. The initial condition for the temperature and velocity are set as in Example~\ref{sec:example:2}, $\theta_h^{0} = 25\,\rm C^{\circ}$ and $\bB{u}_h^{0} = \bB{0}$, while for the concentration we have
\begin{align}\label{eq:phi0:cont}
 \phi_h^0(x_1,x_2) = \begin{cases}
             \max\big\{ 0.6(2-x_2), 0\big\} &\text{ if }x_2\geq 1.5\,,\\
             0.3 &\text{ if }  x_2 < 1.5\,.
            \end{cases}
\end{align}
The length of the inflow and outflow boundaries are respectively $L_{\rm in} = 1.5\,\rm m$ and $L_{\rm out} = 1\,\rm m$, and given $u_{\rm max} = 0.5\,\rm m/s$ a maximum bulk velocity, with $u_{\rm in} = 4u_{\rm max}/L_{\rm in}^2$ and $u_{\rm out} = 4u_{\rm max}L_{\rm in}/L_{\rm out}^3$, the Dirichlet velocity is given by
\begin{align*}
 \bB{u}_{\rm D}(x_1,x_2) =
 \begin{cases}
 \big(0, -u_{\rm in}x_1(L_{\rm in}-x_1)\big) & \text{ if } x_2 = 3,\quad 0\leq x_1\leq L_{\rm in}\,,\\
 \big(u_{\rm out}x_2(L_{\rm out}-x_2),0\big) & \text{ if } x_1 = 5,\quad 0\leq x_2\leq L_{\rm out}\,,\\
 \bB{0} & \text{ otherwise},
 \end{cases}
\end{align*}
where $\bB{u}_{\rm D} = \big(u_{{\rm D},1},u_{{\rm D},2}\big)$. Note that in both inlet and outlet boundaries the velocity $\bB{u}_{\rm D}$ is described by a parabolic profile. The volumetric flow is then defined as
\begin{align*}
 \Phi(x_1,x_2) =
 \begin{cases}
 -0.6 u_{{\rm D},2}(x_1,x_2)& \text{ if } x_2 = 3,\quad 0\leq x_1\leq L_{\rm in}\,,\\
 -\phi_h^n u_{{\rm D},1}(x_1,x_2) & \text{ if } x_1 = 5,\quad 0\leq x_2\leq L_{\rm out}\,.
 \end{cases}
\end{align*}
Note that the outflow depends on the concentration of solid particles at the previous time while the inflow has a fixed volume fraction of $0.6$. The rest of the functions and parameters are taken as in Example~\ref{sec:example:2}.

In Figure~\ref{fig:cont1} we show the numerical solution at three times $t = 0.5, 1.5, 3\,\rm s$, with $\Delta t = 0.001\,\rm s$, computed in a triangular mesh with $h= 0.0826$. The average number of Picard iterations was 5. We observe that the concentration remains positive and bounded, and that sedimentation take place while solid particles are entering from the top and leaving the tank from the rightmost boundary. The temperature variations in this example show a more pronounced effect due to the velocity $\bB{u}$, where the maximum temperature is assumed at the outlet boundary. For the velocity,
we also visualize the velocity magnitude and direction using arrows. Although the flow is primarily determined by the Dirichlet boundary conditions at the inlet and outlet, the influence of temperature can also be observed, most notably on the inclined wall and near the outlet.

\section{Conclusions} \label{sec:conclusions}

A pseudo-steady-state model of heat-assisted multidimensional sedimentation was introduced. The model consisting of two convection diffusion equations, and flow equations, including several coupling terms, was studied combining three types of boundary conditions; Neumann, Dirichlet and mixed. The strategy followed to analyse the model equations was to formulate three decoupled and linearised saddle-point problems, for which the Babu$\rm\check{s}$ka-Brezzi theory in Banach spaces was applied. After constructing an operator connecting the solutions of the three saddle-point problems, the existence and uniqueness of the fully nonlinearly coupled system of equations is shown using Banach's fixed-point theorem on the coupling operator. A drawback is that due to the specific couplings through coefficient functions, and because of the technique employed to show that the coupling operator is a contraction, an additional regularity assumption is required. This type of regularity conditions, which states that some components of the solution should belong to an additional slightly smoother Sobolev space, have already been employed in the literature \cite{Alvarez2016b,Alvarez2026,Careaga2025a}.

The discrete formulation of the three saddle-point problems and fully coupled system was based on the continuous formulation. A combination of piecewise polynomial and Raviart-Thomas spaces, together with the Arnold--Falk--Winther elements for the flow equations, was used to ensure inf-sup stability for the numerical approximations.
Alternatively, PEERS elements can also be employed for the flow equations. For the functional space over the boundary, the approximation is carried out over a coarser boundary mesh, following \cite{Dominguez2013}. The well-posedness of the discrete formulation follows a similar approach as for the continuous case.
A priori estimates are obtained by combining the global inf-sup conditions for the continuous and discrete formulations of the three saddle-point problems. Rates of convergence for sufficiently smooth solutions were shown using approximation properties of the chosen discrete spaces as in \cite{Careaga2023,Careaga2025a,Gatica2025}.

The current work can be further extended in different lines. In terms of the governing equations, further studies can be conducted to model the case in which the flow equations involve an additional nonlinear term in $\bB{u}$ typically arising in the material derivative in the Navier-Stokes equations.  Furthermore, investigating the limiting case in which $\mathcal{D}_1$ vanishes in an interval of $\phi$, the so-called strongly-degenerate parabolic case, is also of research interest. The analysis of the fully time dependent case is left for the future.
Regarding the treatment of the boundary terms, the approach followed to analyse the boundary terms by using a coarser boundary mesh is still pending to be extended to the three-dimensional case. Finally, extensions of heat-assisted sedimentation to the modelling of reactive settling in activated sludge processes \cite{Careaga2021,SDAMM_SBR1,SDAMM_SBR2} is of particular interest as chemical reactions between bacteria and substrates have an impact on temperature variations.

\section*{Acknowledgements}
This work has been partially supported by ANID Chile through the Fondecyt projects no.~3230553 and no.~3260398.

\appendix

\section{Appendix: Postponed proofs} \label{appendix:postponed:proofs}

\subsection{Proof of Lemma \ref{lemma:infsup:b:mass}}\label{appendix:proof:infsup:b:mass}
\begin{proof}
Let $\vec{\bB{\psi}}=(\psi,\chi)\in\bQ_{1,0}$, where $\psi\in \L_0^{\rho}(\Omega)$, and define function $\psi_0^\varrho\in \L_0^\varrho(\O)$ as follows:
\begin{align*}
\psi_0^{\varrho} := |\psi|^{\rho-2}\psi - \frac{1}{|\O|} \int_\O |\psi|^{\rho-2}\psi
\quad \text{where}\quad \int_\O
\psi\,\psi_0^{\varrho}
=\|\psi\|_{0,\rho;\O}^\rho\quad\text{and}\quad
\|\psi_0^{\varrho}\|_{0,\varrho;\O} \leq C_0 \|\psi\|_{0,\rho;\O}^{\rho-1}\,,
\end{align*}
where $C_0>0$ depends only on $\O$ and $\rho$. Then, we let $z_\psi\in\mathrm H^1(\O)$ be the unique solution to the auxiliary problem
\begin{align*}
-\Delta z_\psi = \psi_0^{\varrho} \qin\O\,,\qquad
\nabla z_\psi\cdot\bB{\nu} = 0 \qon\Gamma, \quad\qan\quad
\int_\O z_\psi = 0\,.
\end{align*}
The zero-mean condition on $\psi_0^{\varrho}$ guarantees existence, whereas the normalization of $z_\psi$ guarantees uniqueness. Moreover, $z_{\psi}$ satisfies the a priori bound $\|z_\psi\|_{1;\O} \leq \widetilde{C}_1 \|\psi_0^{\varrho}\|_{0,\varrho;\O}$, with $\widetilde{C}_{1}>0$ a constant independent of $\psi$. Now, defining
$\bB{\varsigma}_\psi:=-\nabla z_\psi$, we obtain $\div(\bB{\varsigma}_\psi) = \psi_0^{\varrho}$ and $\bB{\varsigma}_\psi\cdot\bB{\nu}=0$, and
\begin{align*}
\|\bB{\varsigma}_\psi\|_{\div_\varrho;\O}
\leq
C_1
\|\psi\|_{0,\rho;\O}^{\rho-1}\,,
\end{align*}
with $C_1$ a positive constant also independent of $\psi$. Hence, we have
\begin{align}
b\big( (\bB{0},\bB{\varsigma}_\psi), (\psi,\chi) \big) &
= \int_\O \psi\,\psi_0^{\varrho}
= \|\psi\|_{0,\rho;\O}^\rho
\qquad\text{and}\qquad
\sup_{\substack{ \vsmass\in\bH_1\setminus\{\vec{\bB{0}}\} }}
\frac{ b(\vsmass,(\psi,\chi)) }{ \|\vsmass\|_{\bH_1} }
\geq
C_1^{-1} \|\psi\|_{0,\rho;\O}\,. \label{eq:cota:aux1}
\end{align}
To control the boundary component $\chi$, we first make use of the dual characterization of the norm in $\mathrm{H}^{1/2}(\Omega)$, which implies the existence of $\eta\in\mathrm H^{-1/2}(\Gamma)$ such that $ \|\eta\|_{-1/2,\Gamma} = 1$ and $\langle\eta,\chi\rangle_\Gamma \geq \tfrac12 \|\chi\|_{1/2,\Gamma}$. Then, we let $z_\chi\in\mathrm H^1(\O)$ be the unique solution to the auxiliary problem
\begin{align*}
\Delta z_\chi = \frac{1}{|\O|} \langle\eta,1\rangle_\Gamma \qin\O, \qquad
\nabla z_\chi\cdot\bB{\nu} = \eta \qon\Gamma, \qan
\int_\O z_\chi = 0\,,
\end{align*}
which satisfies an a priori bound $\|z_\chi\|_{1;\O} \leq \widetilde{C}_2 \|\eta\|_{-1/2,\Gamma} = \widetilde{C}_{2}$, with $\widetilde{C}_{2}>0$ independent of $\chi$.
Moreover, defining $\bB{\varsigma}_\chi:=\nabla z_\chi$, we have that $\div(\bB{\varsigma}_\chi) = |\O|^{-1}\langle\eta,1\rangle_\Gamma$,  $ \bB{\varsigma}_\chi\cdot\bB{\nu} = \eta$, and $ \|\bB{\varsigma}_\chi\|_{\div_\varrho;\O} \leq C_2$, with $C_2>0$ also independent of $\chi$. Thus, since $\psi$ has zero mean, we conclude that
\begin{align}
b\big( (\bB{0},\bB{\varsigma}_\chi), (\psi,\chi) \big)
= \langle\eta,\chi\rangle_\Gamma\qquad \text{and}\qquad
\sup_{\substack{ \vsmass\in\bH_1\setminus\{\vec{\bB{0}}\} }}
\frac{ b(\vsmass,(\psi,\chi)) }{ \|\vsmass\|_{\bH_1} }
\geq
\frac{1}{2C_2} \|\chi\|_{1/2,\Gamma}\,. \label{eq:cota:aux2}
\end{align}
Denoting by $M$ the supremum appearing in \eqref{eq:cota:aux1} and \eqref{eq:cota:aux2}, and combining both inequalities we have
\begin{align*}
\begin{aligned}
M \geq \dfrac{1}{2} \Big(C_1^{-1} \|\psi\|_{0,\rho;\O} + \frac{1}{2C_2} \|\chi\|_{1/2,\Gamma}\Big)
\geq \dfrac{1}{4} \min \{ 2C_1^{-1}, C_2^{-1}\} \Big(\|\psi\|_{0,\rho;\O} + \|\chi\|_{1/2,\Gamma} \Big)\,.
\end{aligned}
\end{align*}
This proves \eqref{eq:IS:b:mass} with $\beta_1
:= \min \{ 2C_1^{-1}, C_2^{-1}\}/4$.
\end{proof}

\subsection{Proof of Lemma~\ref{lemma:infsup:b:heat}}\label{appendix:proof:infsup:b:heat}
\begin{proof}
The result is immediate for $\varkappa=0$. Hence, let $\varkappa\in\L^\rho(\O)\setminus\{0\}$ and define $\varkappa_\varrho := |\varkappa|^{\rho-2}\varkappa \in\L^\varrho(\O)$, which satisfies
\begin{align*}
\|\varkappa_\varrho\|_{0,\varrho;\O} = \|\varkappa\|_{0,\rho;\O}^{\rho-1}
\qan
\int_\O\varkappa\,\varkappa_\varrho = \|\varkappa\|_{0,\rho;\O}^{\rho}\,.
\end{align*}
Since $\Gamma_{\rm D}$ has positive surface measure, the Friedrichs--Poincaré inequality holds in $\mathrm H^1_{\Gamma_{\rm D}}(\O)$, therefore there exists $c_{\rm p}>0$ such that
\begin{align*}
c_{\rm p}\|v\|_{1;\O}^2 \leq \|\nabla v\|_{0;\O}^2
\qquad
\forall\,v\in\mathrm H^1_{\Gamma_{\rm D}}(\O)\,.
\end{align*}
Equivalently, $z$ satisfies
\begin{align*}
-\Delta z = \varkappa_\varrho \qin\O, \qquad z=0 \qon\Gamma_{\rm D},
\qquad
\nabla z\cdot\bB{\nu}=0 \qon\Gamma_{\rm N}\,.
\end{align*}
We let $z\in\mathrm H^1_{\Gamma_{\rm D}}(\O)$ be the unique weak solution to
\begin{equation*}
\int_\O\nabla z\cdot\nabla v = \int_\O\varkappa_\varrho\,v
\qquad
\forall\,v\in\mathrm H^1_{\Gamma_{\rm D}}(\O)\,.
\end{equation*}
The, replacing $v=z$ above and using the continuous injection $i_\rho:\mathrm H^1(\O)\to\L^\rho(\O)$, we obtain the a priori bound $\|z\|_{1;\O} \leq \|i_\rho\|c_{\rm p}^{-1} \|\varkappa_\varrho\|_{0,\varrho;\O}$. Now, defining $\widetilde{\bB{\varsigma}}:=-\nabla z$, it follows that $\div(\widetilde{\bB{\varsigma}}) = \varkappa_\varrho$ in $\O $ and $ \widetilde{\bB{\varsigma}}\cdot\bB{\nu} = 0$ on $\Gamma_{\rm N}$ thus $\widetilde{\bB{\varsigma}}\in\bH_{\Gamma_{\rm N}}(\div_\varrho;\O)$. Using the a priori estimate above, we have
\begin{align*}
\|\widetilde{\bB{\varsigma}}\|_{\div_\varrho;\O}
\leq (1 + \|i_\rho\|/c_{\rm p}) \|\varkappa_\varrho\|_{0,\varrho;\O}\,.
\end{align*}
Therefore, choosing $(\bB{0},\widetilde{\bB{\varsigma}})\in\bH_2$, we find that
\begin{align*}
\sup_{\substack{\vsheat\in\bH_2\setminus\{\vec{\bB{0}}\}}}
\frac{ \mfk{b}(\vsheat,\varkappa) }{ \|\vsheat\|_{\bH_2} }
&\,\geq\,
\frac{ \mfk{b} \big( (\bB{0},\widetilde{\bB{\varsigma}}), \varkappa \big) }{ \|(\bB{0},\widetilde{\bB{\varsigma}})\|_{\bH_2} }
\,=\,
\frac{\|\varkappa\|_{0,\rho;\O}^{\rho}}{ \|\widetilde{\bB{\varsigma}}\|_{\div_\varrho;\O} }
\,\geq\,
\left(1 + \frac{\|i_\rho\|}{c_{\rm p}}\right)^{-1} \|\varkappa\|_{0,\rho;\O}\,.
\end{align*}
The result follows with $\beta_2 := ( 1+ \|i_\rho\|/c_{\rm p})^{-1}$.
\end{proof}

\subsection{Proof of Lemma~\ref{lemma:infsup:b:flow}} \label{appendix:proof:infsup:b:flow}
\begin{proof}
The particular case $r=4$ and $s=4/3$ is established in \cite[Lemma 3.5]{Gatica2021}. We adapt its proof to the conjugate exponents $(r,s)$ satisfying \eqref{def:ranges}. In this way, we let $\vec{\bB{w}}=(\bB{w},\widetilde{\bB{\gamma}}) \in\mbb{M}$ and define $\bB{w}_s\in\bL^s(\O)$ whose components are given by $w_{s,j}:=|w_j|^{r-2}w_j$,  $j=1,\ldots,d$. Then, there holds
\begin{align*}
\|\bB{w}_s\|_{0,s;\O}^s    = \|\bB{w}\|_{0,r;\O}^r \qan
\int_\O\bB{w}\cdot\bB{w}_s = \|\bB{w}\|_{0,r;\O}^r\,.
\end{align*}
Let $\bB{z}_{w}\in\bH_0^1(\O)$ be the unique solution to the auxiliary problem
\begin{align*}
\int_\O \bB{e}(\bB{z}_{w}):\bB{e}(\bB{v}) = -\int_\O\bB{w}_s\cdot\bB{v}
\qquad \forall\,\bB{v}\in\bH_0^1(\O)\,.
\end{align*}
The continuous injection $\bB{i}_r:\bH^1(\O)\to\bL^r(\O)$, together with the Poincaré and Korn inequalities, gives the a priori estimate
$ \|\bB{e}(\bB{z}_{w})\|_{0;\O} \leq C_w\|\bB{w}_s\|_{0,s;\O}$, where $C_w>0$ is independent of $\bB{w}_s$. Next, let $\bB{\tau}_{w}\in\Hdivtr$ be the zero-mean-trace component of
$\bB{e}(\bB{z}_{w})$. Then, it follows that $\bdiv(\bB{\tau}_{w})=\bB{w}_s$ and $\bB{\tau}_{w}=\bB{\tau}_{w}^{\tt t}$, and there exists $C_1>0$ such that
\begin{align*}
\|\bB{\tau}_{w}\|_{\bdiv_s;\O} \leq C_1\|\bB{w}_s\|_{0,s;\O}\,.
\end{align*}
Therefore, we obtain the following bound
\begin{align}\label{eq:IS:b:flow:velocity}
\sup_{\substack{\vec{\bB{\eta}}\in\mX\setminus\{\vec{\bB{0}}\} }}
\frac{\mbf{b}(\vec{\bB{\eta}},\vec{\bB{w}})}{\|\vec{\bB{\eta}}\|_{\mX}}
\geq \frac{\mbf{b}\big((\bB{0},\bB{\tau}_{w}),\vec{\bB{w}}\big)}{\|\bB{\tau}_{w}\|_{\bdiv_s;\O}}
= \frac{ \|\bB{w}\|_{0,r;\O}^r}{\|\bB{\tau}_{w}\|_{\bdiv_s;\O}}
\geq C_1^{-1}\|\bB{w}\|_{0,r;\O}\,.
\end{align}
We next control the vorticity component. Let $\bB{z}_{\gamma}\in\bH_0^1(\O)$ be the unique solution to
\begin{align*}
\int_\O \bB{e}(\bB{z}_{\gamma}):\bB{e}(\bB{v}) = -\int_\O \widetilde{\bB{\gamma}}:\nabla\bB{v}
\qquad \forall\,\bB{v}\in\bH_0^1(\O)\,.
\end{align*}
The Korn inequality implies $\|\bB{e}(\bB{z}_{\gamma})\|_{0;\O} \leq C_\gamma \|\widetilde{\bB{\gamma}}\|_{0;\O}$ with $C_{\gamma}>0$ independent of $\widetilde{\bB{\gamma}}$. Now, let $\bB{\tau}_{\gamma}\in\Hdivtr$ be the
zero-mean-trace component of $\bB{e}(\bB{z}_{\gamma})+\,\widetilde{\bB{\gamma}}$. The preceding variational problem implies that $\bdiv(\bB{\tau}_{\gamma})=\bB{0}$.
Moreover, since
$\bB{e}(\bB{z}_{\gamma})$ and $\mbb{I}$ are symmetric, whereas $\widetilde{\bB{\gamma}}$ is skew-symmetric, there hold
\begin{align*}
\int_\O \bB{\tau}_{\gamma}: \widetilde{\bB{\gamma}}
= \|\widetilde{\bB{\gamma}}\|_{0;\O}^2 \qan \|\bB{\tau}_{\gamma}\|_{\bdiv_s;\O} \leq C_2\|\widetilde{\bB{\gamma}}\|_{0;\O}\,,
\end{align*}
with $C_2>0$ independent of $\widetilde{\bB{\gamma}}$. Consequently, we get the bound
\begin{align}\label{eq:IS:b:flow:vorticity}
\sup_{\substack{ \vec{\bB{\eta}}\in\mX\setminus\{\vec{\bB{0}}\} }}
\frac{\mbf{b}(\vec{\bB{\eta}},\vec{\bB{w}})}{\|\vec{\bB{\eta}}\|_{\mX}}
& \geq
\frac{\mbf{b}\big((\bB{0},\bB{\tau}_{\gamma}),\vec{\bB{w}}\big)}{\|\bB{\tau}_{\gamma}\|_{\bdiv_s;\O}}
\geq C_2^{-1}\|\widetilde{\bB{\gamma}}\|_{0;\O}\,.
\end{align}
Combining \eqref{eq:IS:b:flow:velocity} and \eqref{eq:IS:b:flow:vorticity}, we obtain \eqref{eq:IS:b:flow} with $\beta_3 := \frac{1}{2} \min\{C_1^{-1},C_2^{-1}\}$.
\end{proof}

\subsection{Proof of Lemma~\ref{lemma:LC:mass}} \label{appendix:proof:LC:mass}

\begin{proof}
Given $\varphi_1,\varphi_2\in\L^\rho(\O)$ and $\bB{w}_1,\bB{w}_2\in\bL^r(\O)$ as stated, let $\phi_i=\mathrm{S}_1(\varphi_i,\bB{w}_i)\in\L^\rho(\O)$, $i\in\{1,2\}$, be such that $(\vtmass_i,\vec{\bB{\phi}}_i)\in\bH_1\times\bQ_1$ is the solution to \eqref{syst:FP:weak:mass} corresponding to $(\varphi_i,\bB{w}_i)$. Using the global bilinear form associated with \eqref{syst:FP:weak:mass}, defined in \eqref{def:A:mass}, we have
\begin{equation}\label{eq:aux:mass:1}
\begin{aligned}
& \bB{A}_{\varphi_1,\bB{w}_1} \Big((\vtmass_1,\vec{\bB{\phi}}_1)-(\vtmass_2,\vec{\bB{\phi}}_2),(\vsmass,\vec{\bB{\psi}})\Big)\\
&=
\bB{A}_{\varphi_2,\bB{w}_2} \Big((\vtmass_2,\vec{\bB{\phi}}_2),(\vsmass,\vec{\bB{\psi}})\Big) -
\bB{A}_{\varphi_1,\bB{w}_1} \Big((\vtmass_2,\vec{\bB{\phi}}_2),(\vsmass,\vec{\bB{\psi}})\Big) +
\mathcal{F}_{\varphi_1}(\vsmass) - \mathcal{F}_{\varphi_2}(\vsmass) \\
&=
\int_{\O} \big(\mathcal{D}_1(\varphi_2)-\mathcal{D}_1(\varphi_1)\big)\bB{r}_2\cdot\bB{s} -
\int_{\O} \phi_2(\bB{w}_2-\bB{w}_1)\cdot\bB{s}
+ \int_{\O} \big(\fbk(\varphi_1)-\fbk(\varphi_2)\big)\bB{k}\cdot\bB{s}\,,
\end{aligned}
\end{equation}
for all $(\vsmass,\vec{\bB{\psi}}) =((\bB{s},\bB{\varsigma}),(\psi,\chi))\in\bH_1\times\bQ_1$. Then, using the Lipschitz continuity of $\mathcal{D}_1$ and $f_{\rm bk}$ in \eqref{eq:LC:data}, and applying H\"older's inequality, we obtain
\begin{align}\label{eq:aux:mass:2}
\begin{aligned}
\Big|\int_{\O} \big(\mathcal{D}_1(\varphi_2) - \mathcal{D}_1(\varphi_1) \big) \bB{r}_2\cdot\bB{s}\,\Big|
&\leq
L_{\mathcal{D}_1}\|\varphi_1-\varphi_2\|_{0,\rho;\O} \|\bB{r}_2\|_{0,r;\O} \|\bB{s}\|_{0;\O}\,, \\[1ex]
\Big|\int_{\O} \phi_2(\bB{w}_2-\bB{w}_1)\cdot\bB{s}\, \Big|
&\leq
\|\phi_2\|_{0,\rho;\O} \|\bB{w}_1-\bB{w}_2\|_{0,r;\O} \|\bB{s}\|_{0;\O}\,,\\[1ex]
\Big|\int_{\O}\big(\fbk(\varphi_1) - \fbk(\varphi_2)\big)\bB{k}\cdot\bB{s}\,\Big|
&\leq
L_{\rm fbk}|\O|^{1/r}\|\varphi_1-\varphi_2\|_{0,\rho;\O}\|\bB{s}\|_{0;\O}\,.
\end{aligned}
\end{align}
The regularity hypothesis~\textbf{(H)}, the continuous injection $\bB{i}_{\varepsilon,r}:\bH^\varepsilon(\O)\to\bL^r(\O)$, and \eqref{eq:bound:Cs} yield
\begin{align*}
\|\bB{r}_2\|_{0,r;\O}
\leq
\|\bB{i}_{\varepsilon,r}\|
\|\bB{r}_2\|_{\varepsilon;\O}
\leq
\|\bB{i}_{\varepsilon,r}\|
\widehat{C}_1\,.
\end{align*}
Furthermore, the a priori estimate \eqref{eq:apriori:mass} and the continuous inclusion ${\rm H}^{-1/2+\varepsilon}(\Gamma) \hookrightarrow{\rm H}^{-1/2}(\Gamma)$ imply:
\begin{align}
\|\phi_2\|_{0,\rho;\O}
&\leq C_{\mathrm{S}_1}\datacons
\leq C_1\widehat{C}_1\,, \label{eq:bound:phi2:mass}
\end{align}
where $C_1>0$ is independent of the data and the input functions. Hence, applying the global inf-sup condition \eqref{eq:GIS:mass:1} to \eqref{eq:aux:mass:1}, and using \eqref{eq:aux:mass:2}--\eqref{eq:bound:phi2:mass}, we obtain
\begin{align*}
\frac{\alpha_{\bB{A}}}{2} \|(\vtmass_1,\vec{\bB{\phi}}_1) - (\vtmass_2,\vec{\bB{\phi}}_2)\|_{\bH_1\times\bQ_1}
&\leq
\Big\{L_{\mathcal{D}_1}\|\bB{i}_{\varepsilon,r}\|\widehat{C}_1 + L_{\rm fbk}|\O|^{1/r}\Big\}
\|\varphi_1-\varphi_2\|_{0,\rho;\O} + C_1\widehat{C}_1 \|\bB{w}_1-\bB{w}_2\|_{0,r;\O}\,.
\end{align*}
Taking the maximum over the coefficients multiplying the norms on the right-hand side and using
\begin{align*}
\|\phi_1-\phi_2\|_{0,\rho;\O} \leq \| (\vtmass_1,\vec{\bB{\phi}}_1) - (\vtmass_2,\vec{\bB{\phi}}_2)\|_{\bH_1\times\bQ_1}\,,
\end{align*}
we obtain \eqref{eq:LC:mass} with $L_{\mathrm{S}_1}:=\frac{2}{\alpha_{\bB{A}}} \max\left\{\|\bB{i}_{\varepsilon,r}\|,|\O|^{1/r},C_1\right\}$.
\end{proof}

\subsection{Proof of Lemma~\ref{lemma:LC:flow}}\label{appendix:proof:LC:flow}
\begin{proof}
Given $(\varphi_1,\vartheta_1),(\varphi_2,\vartheta_2)\in\L^\rho(\O)\times\L^\rho(\O)$, for each $i\in\{1,2\}$, we set $\bB u_i := \bS(\varphi_i,\vartheta_i) \in\bL^r(\O)$, and let $(\vec{\bB{\zeta}}_i,\vec{\bB{u}}_i)\in\mX\times\mbb M$, with $\vec{\bB{u}}_i = (\bB u_i,\bB\gamma_i)$, be the corresponding solution to \eqref{syst:FP:weak:flow}. Using the global bilinear form \eqref{def:A:flow}, we obtain
\begin{equation}\label{eq:aux:flow:1}
\begin{aligned}
& \mbb A_{\varphi_1,\vartheta_1}\big((\vec{\bB{\zeta}}_1,\vec{\bB{u}}_1) - (\vec{\bB{\zeta}}_2,\vec{\bB{u}}_2),(\vec{\bB{\eta}},\vec{\bB{w}})\big)\\
& \quad =
\mbb A_{\varphi_2,\vartheta_2}\big((\vec{\bB{\zeta}}_2,\vec{\bB{u}}_2),(\vec{\bB{\eta}},\vec{\bB{w}})\big)
- \mbb A_{\varphi_1,\vartheta_1}\big((\vec{\bB{\zeta}}_2,\vec{\bB{u}}_2),(\vec{\bB{\eta}},\vec{\bB{w}})\big)
+ \mbf F_{\varphi_1,\vartheta_1}(\vec{\bB{w}}) - \mbf F_{\varphi_2,\vartheta_2}(\vec{\bB{w}})\\
& \quad = \int_{\O}\big(\mu(\varphi_2,\vartheta_2) - \mu(\varphi_1,\vartheta_1)\big)\bB{\zeta}_2:\bB{\eta}
- \delta_4 \int_{\O}\big(\mbf K^{-1}(\varphi_2) - \mbf K^{-1}(\varphi_1)\big)\bB u_2\cdot\bB{w}\\
&\qquad
- \int_{\O} \big(\mbf f(\varphi_1,\vartheta_1) - \mbf f(\varphi_2,\vartheta_2)\big)\cdot\bB{w}\,.
\end{aligned}
\end{equation}
Using the Lipschitz continuity assumptions in \eqref{eq:LC:data} and H\"older's inequality with the exponents
$(\rho,r,2)$, we obtain
\begin{align}
\left| \int_{\O}\big(\mu(\varphi_2,\vartheta_2) - \mu(\varphi_1,\vartheta_1)\big)\bB{\zeta}_2:\bB{\eta}\,\right|
& \leq
L_\mu \Big(\|\varphi_1-\varphi_2\|_{0,\rho;\O} + \|\vartheta_1-\vartheta_2\|_{0,\rho;\O}\Big)
\|\bB{\zeta}_2\|_{0,r;\O}\|\bB{\eta}\|_{0;\O}\,,\label{eq:bound:mu:flow}\\[1ex]
\delta_4 \left|\int_{\O}\big(\mbf K^{-1}(\varphi_2) - \mbf K^{-1}(\varphi_1)\big)\bB u_2\cdot\bB{w}\,\right|
& \leq
\delta_4L_{\mbf K}\|\varphi_1-\varphi_2\|_{0,\rho;\O}\|\bB u_2\|_{0;\O}\|\bB{w}\|_{0,r;\O}\,, \label{eq:bound:K:flow}\\[1ex]
\left|\int_{\O} \big(\mbf f(\varphi_1,\vartheta_1) - \mbf f(\varphi_2,\vartheta_2)\big)\cdot\bB{w}\,\right|
& \leq
L_{\mbf f}|\O|^{1/2} \Big(\|\varphi_1-\varphi_2\|_{0,\rho;\O} + \|\vartheta_1-\vartheta_2\|_{0,\rho;\O}\Big)\|\bB{w}\|_{0,r;\O}\,.\label{eq:bound:f:flow}
\end{align}
The regularity hypothesis~\textbf{(H)} and the continuous injection $\bB j_{\varepsilon,r}: \mbb H^\varepsilon(\O)\to\mbb L^r(\O)$ imply
\begin{align}
\|\bB{\zeta}_2\|_{0,r;\O} & \leq \|\bB j_{\varepsilon,r}\| \|\bB{\zeta}_2\|_{\varepsilon;\O} \nonumber
\leq C_{\varepsilon,3} \|\bB j_{\varepsilon,r}\|
 \dataflow\big(\|(\varphi_2,\vartheta_2)\|\big)\,.\label{eq:bound:zeta2:flow}
\end{align}
Furthermore, the a priori estimate \eqref{eq:apriori:flow}, the inclusion $\bL^r(\O)\hookrightarrow\bL^2(\O)$, and $\bH^{1/2+\varepsilon}(\Gamma)\hookrightarrow\bH^{1/2}(\Gamma)$ give
\begin{equation}\label{eq:bound:u2:flow}
\|\bB u_2\|_{0;\O}
\leq C_3 \dataflow\big(\|(\varphi_2,\vartheta_2)\|\big)\,,
\end{equation}
where $C_3>0$ is independent of the data and the input functions. Finally, applying the global inf-sup condition
\eqref{eq:GIS:flow:1} to \eqref{eq:aux:flow:1}, and using \eqref{eq:bound:mu:flow}--\eqref{eq:bound:u2:flow}, we obtain
\eqref{eq:LC:S}, after absorbing the resulting structural constants into $L_{\bS}$.
\end{proof}

\bibliographystyle{plain}

\end{document}